\documentclass[11pt]{article}
\usepackage{glauber}
\usepackage{physics}
\usepackage{mathtools}

\usepackage{quasi}
\allowdisplaybreaks

\begin{document}

\title{Weak Poincar{\'e} Inequalities, Simulated Annealing, and Sampling from Spherical Spin Glasses}

\author{Brice Huang\thanks{MIT. Email: \texttt{bmhuang@mit.edu}. Supported by a Google PhD Fellowship, NSF-Simons collaboration grant DMS-2031883, NSF CAREER grant DMS-1940092, and the Solomon Buchsbaum Research Fund at MIT.}
    \and Sidhanth Mohanty\thanks{MIT. Email: \texttt{sidm@mit.edu}. Supported by NSF Award DMS-2022448.}
  \and Amit Rajaraman\thanks{MIT. Email: \texttt{amit\_r@mit.edu}. Supported by CSAIL.}
  \and David X. Wu\thanks{UC Berkeley. Email: \texttt{david\_wu@berkeley.edu}. Supported by an OpenAI Superalignment grant.}}
\date{\today}
\maketitle

\begin{abstract}
    There has been a recent surge of powerful tools to show rapid mixing of Markov chains, via functional inequalities such as \emph{\Poincare inequalities}.
    In many situations, Markov chains fail to mix rapidly from a worst-case initialization, yet are expected to approximately sample from a random initialization.
    For example, this occurs if the target distribution has \emph{metastable states}, small clusters accounting for a vanishing fraction of the mass that are essentially disconnected from the bulk of the measure.
    Under such conditions, a \Poincare inequality cannot hold, necessitating new tools to prove sampling guarantees.

    We develop a framework to analyze simulated annealing, based on establishing so-called \emph{weak \Poincare inequalities}.
    These inequalities imply mixing from a suitably warm start, and simulated annealing provides a way to chain such warm starts together into a sampling algorithm.
    We further identify a local-to-global principle to prove weak \Poincare inequalities, mirroring the spectral independence and localization schemes frameworks for analyzing mixing times of Markov chains.

    As our main application, we prove that simulated annealing samples from the Gibbs measure of a spherical spin glass for inverse temperatures up to a natural threshold, matching recent algorithms based on algorithmic stochastic localization.
    This provides the first Markov chain sampling guarantee that holds beyond the \emph{uniqueness threshold} for spherical spin glasses, where mixing from a worst-case initialization is provably slow due to the presence of metastable states.
    As an ingredient in our proof, we prove bounds on the operator norm of the covariance matrix of spherical spin glasses in the full replica-symmetric regime.
    
    Additionally, we resolve a question related to sampling using data-based initializations.
\end{abstract}

\thispagestyle{empty}
\setcounter{page}{0}
\newpage
\tableofcontents
\thispagestyle{empty}
\setcounter{page}{0}
\newpage

\newcommand{\bhsuggestion}[1]{{\color{blue} [\textbf{BH} : #1]}}
\newcommand{\smsuggestion}[1]{{\color{Fuchsia} [\textbf{SM} : #1]}}
\newcommand{\arsuggestion}[1]{{\color{Fuchsia} [\textbf{AR} : #1]}}
\newcommand{\dwsuggestion}[1]{{\color{Emerald} [\textbf{DW} : #1]}}

\section{Introduction}
\label{sec:intro}

A common task of interest in computer science, probability, and physics is to efficiently sample from {Gibbs distributions}.
For a {Hamiltonian} energy function $H:\Omega\to\R$ over state space $\Omega\subseteq\R^N$, the associated Gibbs distribution $\mu_H$ is defined by $\dif\mu_H(x) \propto \exp\parens*{H(x)} \dif x$.

The class of \emph{Markov chain Monte Carlo} (MCMC) algorithms is arguably the most widely used tool for sampling from Gibbs distributions.
In this paradigm, one sets up a Markov chain $P_H$ whose stationary distribution is $\mu_H$, and outputs the final state of a $\poly(N)$-time random walk according to $P_H$.
Common choices include the {Glauber dynamics}, for discrete state spaces such as $\Omega = \{\pm 1\}^N$, and the {Langevin diffusion}, for continuous state spaces such as $\Omega = \R^N$ or $\sqrt{N} \cdot \bbS^{N-1}$.

To prove that such an algorithm indeed correctly samples from $\mu_H$, one bounds the \emph{mixing time} of the Markov chain.
A common route to prove a bound on the mixing time is to establish functional inequalities, such as \emph{\Poincare} inequalities.
There are now powerful frameworks for proving such functional inequalities, such as \emph{spectral independence} \cite{ALO21} and \emph{localization schemes} \cite{CE22}.
The development of these frameworks has led to a flurry of activity in analyzing mixing times of Markov chains, including the resolution of several long-standing open problems in the algorithmic theory of counting and sampling \cite{ALOV24,ALO21,AJKPV22,EKZ22,CE22}.

The implications of these inequalities are quite strong.
In particular, they imply that for \emph{any} initial distribution $\nu$, for an appropriate divergence function, a single step of the Markov chain shrinks the distance to the stationary distribution by a significant multiplicative factor:
\[
    \mathrm{Divergence}(P_H \nu \| \mu_H) \le \parens*{1-\frac{1}{\poly(N)}} \mathrm{Divergence}(P_H \nu \| \mu_H)\mper
\]
The presence of such a functional inequality typically implies that a Markov chain mixes rapidly from a \emph{worst-case initialization}.

\parhead{Sampling from random initializations.}
Many natural Markov chains are expected to produce approximate samples from the Gibbs measure when started at a \emph{random} initialization, but fail to mix rapidly from a worst-case initialization.
Often, this is because the Gibbs measure contains pathological clusters (termed \emph{metastable states} in the physics literature) that are essentially disconnected from most of the measure, and account for a vanishing fraction of the total mass.
A Markov chain initialized in such a cluster will remain trapped inside it and fail to mix, and therefore methods that show mixing from worst-case initializations cannot give effective guarantees in such settings.

However, one may still hope to show that from a random initialization, the Markov chain samples from the non-pathological part of the Gibbs measure, which is statistically indistinguishable from the true Gibbs measure.
In our work, we prove that under suitable conditions, the \emph{simulated annealing} algorithm samples from a distribution close to the Gibbs measure.

\parhead{Simulated annealing.}
In the simulated annealing algorithm, one defines a ``{schedule}'' of inverse temperatures, i.e. for $i=0,\ldots,T$, let $\beta_i \coloneqq i/T$.
The algorithm initializes at a sample from the uniform distribution $\mu_{\beta_0 H}$.
Then, for $i=1,\ldots,T$, the $i$-th stage of the algorithm runs the Markov chain $P_{\beta_i H}$ corresponding to $\mu_{\beta_i H}$ for $\poly(N)$ time, initialized at the output of the previous stage.

The underlying idea of this algorithm is that, for $T$ sufficiently large, the Gibbs distribution $\mu_{\beta_{i-1} H}$ is a ``warm start'' for $\mu_{\beta_i H}$, i.e. an initialization with suitably bounded likelihood ratio with $\mu_{\beta_i H}$.
So, if one could show that each of the Markov chains $P_{\beta_i H}$ (approximately) mixes rapidly from a warm start, one may inductively argue that the output of the $i$-th stage of the algorithm is an approximate sample from $\mu_{\beta_i H}$.
In other words, simulated annealing chains a sequence of warm starts together into a sampling algorithm.

This algorithmic idea is widely used empirically, and has also been employed to obtain algorithms for approximating the volumes of convex bodies \cite{DFK91,DF91,LS90,KLS97}, approximating the number of perfect matchings in a bipartite graph \cite{JSV04}, and sampling from the random field Ising model at sufficiently high temperatures \cite{EAEGP23}, among others.
However, we lack a general theory for why simulated annealing achieves provable guarantees beyond the settings of sampling from log-concave distributions and convex bodies.
Indeed, in contrast to the general recipes available to prove mixing from worst-case initialization, proofs of rapid mixing from warm starts often employ ad-hoc arguments.

One of the main contributions of this work is a framework for proving mixing from warm starts, which combined with the above discussion provides general sufficient conditions under which simulated annealing samples from the Gibbs measure.
We achieve this by generalizing the frameworks of spectral independence and localization schemes, previously employed to prove mixing from worst-case initialization, to show mixing from warm starts (see \Cref{sec:stopped-ls} for details).
As we discuss just below, our framework gives sampling guarantees for simulated annealing in regimes where mixing from worst-case initializations is provably false.

\parhead{Main application: spherical spin glasses.}
In a \emph{spherical mixed $p$-spin glass}, $H:\sqrt{N}\cdot\bbS^{N-1}\to\R$ is a random Hamiltonian parameterized by coefficients $\beta, \gamma_2,\dots,\gamma_{p_*}\ge 0$ where:
\begin{equation}
    \label{eq:hamiltonian}
    H(\sigma) = \beta \sum_{p\ge 2} \frac{\gamma_p}{N^{(p-1)/2}} \sum_{i_1,\dots,i_p} \bg_{i_1,\dots,i_p} \sigma_{i_1}\cdots\sigma_{i_p},
\end{equation}
for i.i.d. $\bg_{i_1,\dots,i_p} \sim \calN(0,1)$.
The Gibbs distribution $\mu_H$ is very well-studied in probability, statistical physics, and average-case algorithm design, as it simultaneously exhibits rich behavior and is amenable to analytic tools.
Notably, this model undergoes numerous sharp phase transitions as one increases $\beta$.
For small $\beta$, the model satisfies a \Poincare inequality \cite{GJ19}.
Beyond a \emph{uniqueness} transition $\beta_{\mathrm{uniq}}$, small isolated clusters in $\mu_H$ known as \emph{metastable states} start to appear \cite{AJ24a}.
In particular, the natural Markov chain \emph{Langevin diffusion} initialized from such states mixes slowly, thereby precluding a \Poincare inequality.
However, these states account for a vanishing fraction of the measure under $\mu_H$, and the Langevin diffusion with a random initialization is expected to still mix rapidly over a $1-o_N(1)$ fraction of $\mu_H$, thereby producing a sample with vanishing total variation distance from $\mu_H$. 

The threshold for efficient algorithmic sampling is believed to occur at the \emph{shattering} transition $\beta_{\mathrm{\ShatteredPsyche}}$ --- beyond this transition, the Gibbs measure shatters into an exponential number of poorly-connected clusters with exponentially small mass, and mixing is provably slow \cite{CHS93,AMS23,GJK23}.
It is expected that all efficient algorithms fail to sample from the Gibbs measure above $\beta_{\mathrm{\ShatteredPsyche}}$, and recently \cite{AMS23} gave rigorous evidence for this picture by showing that all \emph{stable} algorithms fail.

We use our framework to prove that \emph{annealed} Langevin diffusion, where one begins by running Langevin diffusion for $\beta_0 = 0$, and slowly increases the inverse temperature to the target $\beta$, samples from the spherical mixed $p$-spin glass.
This leads to the first rigorous guarantee in this problem for a Markov chain beyond the uniqueness threshold.
\begin{theorem}[Informal]
    For any choice of $\gamma_2,\dots,\gamma_{p_*}$, there is a threshold $\beta_{\SL} \le \beta_{\ShatteredPsyche}$ such that  
    for any $\beta < \beta_{\SL}$, with probability at least $1-e^{-\Omega(N^{1/5})}$ over the randomness of $H$, annealed Langevin diffusion run for $\poly(N)$ time samples from a distribution whose total variation distance to $\mu_H$ is at most $e^{-\Omega(N^{1/5})}$. 
\end{theorem}
The thresholds $\beta_{\SL}$ and $\beta_{\ShatteredPsyche}$ are formally defined as the supremal $\beta$ such that the inequalities \eqref{eq:SL-condition} and \eqref{eq:non-shatter-condition} below hold.
The recent work \cite{HMP24} produces a different sampling algorithm based on algorithmic stochastic localization, which succeeds to the same threshold $\beta_{\SL}$; see below for further discussion.
This threshold is a fundamental barrier for stochastic localization based approaches, and we explain its physical significance in \Cref{rmk:SL-fundamental-barrier}.
\begin{remark}
    In many models, we have $\beta_{\mathrm{uniq}} < \beta_{\SL} < \beta_{\ShatteredPsyche}$, and $\beta_{\SL}$ is close to $\beta_{\ShatteredPsyche}$.
    For example, for the pure $p$-spin models (where $\gamma_p = 1$ and all the other $\gamma_i$ are equal to $0$), $\beta_{\mathrm{uniq}} \asymp (\log p)^{-1/2}$, while $\beta_{\SL},\beta_{\ShatteredPsyche}\asymp 1$ and $\beta_{\SL} / \beta_{\ShatteredPsyche}$ is bounded from below by the universal constant $\sqrt{e}/2$. 
    See \cite[Remark 1.1, Eq. 1.8]{HMP24}.
\end{remark}

\subsection{Weak \Poincare inequalities and localization schemes}

The starting point of our work is a relaxation of \Poincare inequalities, known as \emph{weak \Poincare inequalities}, which can be leveraged to prove mixing from warm starts.
To simplify the discussion, we restrict here to the setting of discrete Markov chains. 
Our main application is to a continuous Markov chain, namely the Langevin diffusion for a spherical spin glass, and we outline the differences in \Cref{rmk:localization-schemes-cts-markov-chains} below.

Let $P_H$ be a time-reversible Markov chain with stationary distribution $\mu_H$.
For any functions $f,g:\Omega\to\R$, define the \emph{Dirichlet form} as $\calE(f,g) \coloneqq \E_{\bx\sim\mu_H} \E_{\by\sim_{P_H}\bx} (f(\bx)-f(\by))(g(\bx)-g(\by))$.
We say $P_H$ satisfies a $C$-\emph{\Poincare inequality} if for any function $f:\Omega\to\R$:
\[
    \calE(f,f) \ge C\cdot\Var[f]\mcom
\]
for $C \ge 1/\poly(N)$.
A \Poincare inequality has a classic implication for rapid mixing.
In particular, for $\nu_t$ as the distribution obtained by running $P_H$ for continuous time $t$ starting at a distribution $\nu_0$, we have:
\[
    \chi^2\parens*{\nu_t\| \mu_H} \le \exp(-Ct) \cdot \chi^2\parens*{\nu_0 \| \mu_H}\mper
\]
We say $P_H$ satisfies a $(C,\eps)$-weak \Poincare inequality if for any function $f:\Omega\to\R$:
\[
    \calE(f,f) \ge C\cdot\Var[f] - \eps \cdot \norm*{ f - \E f }_{\infty}^2
\]
One can derive the following mixing guarantee from a weak \Poincare inequality; see, e.g., \cite[Theorem 2.1]{RW01}.
\[
    \chi^2\parens*{\nu_t \| \mu_H} \le \exp(-Ct) \cdot \chi^2\parens*{\nu_0 \| \mu_H} + \eps \cdot \left\| \frac{\dif\nu_0}{\dif\mu_H} - 1 \right\|_{\infty}^2\mper \numberthis \label{eq:intro-wpi}
\]
In particular, if $\nu_0$ is a warm start for $\mu_H$ in the sense that $\left\| \frac{\dif \nu_0}{\dif \mu_H} - 1\right\|_\infty$ is suitably small, this implies that the Markov chain's output distribution $\nu_t$ approximates $\mu_H$.

Since the target measure in one stage of simulated annealing is a warm start for that of the next stage, such a guarantee allows one to inductively argue that simulated annealing succeeds at sampling.
We summarize this implication below. 
\begin{theorem}[Informal, see \Cref{th:annealing}]
    If $P_{\beta H}$ satisfies a weak \Poincare inequality with suitable parameters for every $\beta\in[0,1]$, then simulated annealing succeeds at sampling from $\mu_H$.
\end{theorem}

\parhead{Localization schemes for weak \Poincare inequalities.}
We restrict 
to the following simple setting:
$\mu_H$ is a distribution on $\{\pm1\}^N$.
Let $P_H$ be the \emph{Glauber dynamics} Markov chain where in a single step from $x$, we pick a uniformly random coordinate $i\sim[N]$, and toggle $x_i$ with probability:
\[
    \frac{\mu_H(x^{\oplus i})}{\mu_H(x) + \mu_H(x^{\oplus i})}\mper
\]
A special case of the localization schemes framework is the \emph{spectral independence} framework of Anari, Liu, and Oveis Gharan \cite{ALO21}.
\begin{theorem}[{\cite{AL20,ALO21}}]    \label{thm:spec-ind}
The following \emph{local-to-global principle} reduces proving a \Poincare inequality to establishing bounds on the spectrum of \emph{influence matrices}.
    Suppose for {\bf every} $S\subseteq [N]$, and {\bf every} pinning $x_S$ of coordinates in $S$, the spectral norm of its \emph{influence matrix} $\Psi_{S,x_S}$ is at most $\alpha$, then the Glauber dynamics chain satisfies a $n^{-O(\alpha)}$-\Poincare inequality.
    Here, the influence matrix $\Psi_{S,x_S}$ is an $(n-|S|)\times(n-|S|)$ matrix indexed by vertices $v\notin S$, where
    \[
        \Psi_{S,x_S}[a,b] \coloneqq \Pr[x_a = +1 | x_b = +1] - \Pr[x_a = -1 | x_b = +1]\mper
    \]
\end{theorem}
While the above theorem has been influential and useful in proving mixing time bounds for a variety of Markov chains relevant to sampling combinatorial structures, the ``for every'' requirement in the above theorem is quite punishing in average-case settings.
For example, in the presence of metastable states, $P_H$ does not satisfy a \Poincare inequality, but may nevertheless satisfy a weak \Poincare inequality.
In such cases, there are choices of $S$ and $x_S$ for which $\Psi_{S,x_S}$ has large spectral norm, and the above statement has no implications for the mixing time of $P_H$.

We give a general local-to-global principle to prove weak \Poincare inequalities.
A one-line summary of this local-to-global principle is:
\begin{displayquote}
    \emph{Bounded influence over \textbf{all} pinnings implies a \Poincare inequality.}
\end{displayquote}
An analogous summary of the local-to-global principle in the present paper is:
\begin{displayquote}
    \emph{Bounded influencce over {\textbf{typical}} pinnings implies a {\textbf{weak}} \Poincare inequality.}
\end{displayquote}
To give a more concrete instantiation of our message, our result implies a ``softer'' version of \Cref{thm:spec-ind}, tolerant to some ``bad'' pinnings, which we state below.
\begin{theorem}[Special case of \Cref{lem:stopped-scheme-to-weak-poi,rem:other-ls-to-weak-poi}] \label{thm:soft-spec-ind}
    Let $i_1,\dots,i_N$ be a random permutation of $[N]$, let $S_t\coloneqq\{i_1,\dots,i_t\}$, and let $\bx\sim\mu_H$.
    Suppose with probability $1-\eps$ over the randomness of $\bx$ and the permutation $i_1,\dots,i_N$, we have that for every $t\in[N]$, the influence matrix $\Psi_{S_t,\bx_{S_t}}$ has spectral norm bounded by $\alpha$.
    Then, $P_H$ satisfies a $(n^{-O(\alpha)}, O(\eps))$-weak \Poincare inequality.
\end{theorem}

\begin{remark}
    The reader should think of the spectral norm of $\Psi_{S_t, \bx_{S_t}}$ as quantifying how much variance of the distribution $\mu_H|\bx_{S_t}$ is ``explained'' by revealing $\bx_{i_{t+1}}$.
\end{remark}
\begin{remark}
    \label{rmk:localization-schemes-cts-markov-chains}
    \Cref{thm:soft-spec-ind} holds at a more general level, for a large family of \emph{localization schemes}; see \cite{CE22} for examples of localization schemes and further discussion.
    The localization scheme at play in the above local-to-global principle is process of revealing coordinates of a Gibbs sample $\bx$ in random order.

    In our main application of sampling from a spherical spin glass using simulated annealing of Langevin diffusion, we consider a different localization scheme, \emph{stochastic localization}, where the revealed information at time $t$ is $\by_t = t\bx + B_t$ where $(B_t)_{t\ge 0}$ is a standard Brownian motion.
    Analyzing this localization scheme requires studying \emph{exponential tilts} rather than pinnings of $\mu_H$. 
    The analogous local-to-global principle in this setting is: 
    \begin{displayquote}
        \emph{Bounded covariance over \textbf{typical} {exponential tilts} implies a \textbf{weak} \Poincare inequality.}
    \end{displayquote}
    We defer a technical discussion to \Cref{sec:tech-overview}, and refer to \Cref{lem:stopped-scheme-to-weak-poi,rem:other-ls-to-weak-poi} for a formal statement.
    
\end{remark}

\subsection{Sampling from spherical spin glasses}
We now state our main results for sampling from spherical spin glasses. 
We will encode the coefficients $\gamma_2,\ldots,\gamma_{p_*}$ in \eqref{eq:hamiltonian} into the \emph{mixture function} $\xi(q) = \sum_{p=2}^{p_*} \gamma_p^2 q^p$.
Note that the parameter $\beta$ in \eqref{eq:hamiltonian} can of course be absorbed into the $\gamma_p$, so we can state thresholds directly in terms of the function $\xi$.
Physics heuristics \cite{CHS93} suggest that  Glauber dynamics and Langevin diffusion, with random initialization, sample from $\mu_H$ with vanishing total variation error under the following condition.
Note that this and the below conditions take the form of an upper bound on $\xi$ or its derivatives, and therefore demarcate a region of sufficiently high temperature.
\begin{equation} \label{eq:non-shatter-condition}
    \xi'(q) < \frac{q}{1-q} \text{ for all } q \in (0,1). \tag{$\mathsf{Non}\text{-}\mathsf{shattering}$}
\end{equation}
Recent work by one of the authors, Montanari, and Pham \cite{HMP24} gives an algorithm based on simulating Eldan's stochastic localization process \cite{Eld13,Eld20} (see below), which samples from $\mu_H$ with vanishing total variation error under the following condition.
\begin{equation} \label{eq:SL-condition}
    \xi''(q) < \frac{1}{(1-q)^2} \text{ for all } q \in [0,1). \tag{$\mathsf{SL}$}
\end{equation}
Note that this condition implies \eqref{eq:non-shatter-condition}, which can be seen by integrating the inequality.
\cite{HMP24} also shows a matching hardness result, that for any strictly replica symmetric model (see \eqref{eq:strict-RS-condition} below) \emph{not} satisfying \eqref{eq:SL-condition}, a generalized family of stochastic localization algorithms fails to sample from $\mu_H$.

Our main result is that simulated annealing samples from $\mu_H$ in the same regime. 
\begin{theorem}[See \Cref{thm:main-p-spin}]
    Under \eqref{eq:SL-condition}, with probability at least $1-e^{-\Omega(N^{1/5})}$ over the randomness of $H$, annealed Langevin dynamics produces a sample whose total variation distance to $\mu_H$ is at most $e^{-\Omega(N^{1/5})}$.
\end{theorem}
As alluded to in the above discussion,
the main input to our framework is a high-probability covariance bound on the \emph{random} exponential tilts of the Gibbs measure encountered along the stochastic localization process.
Combined with our weak \Poincare inequality framework, this implies that simulated annealing samples from the Gibbs measure.
On the way to proving these covariance bounds, we establish a high-probability covariance bound on all spherical spin glasses in the \emph{(strictly) replica symmetric} phase, a high-temperature phase where the model enjoys a certain notion of correlation decay. 
\begin{equation} \label{eq:strict-RS-condition}
    \xi''(0) < 1 \text{ and } \xi(q) + q + \log(1-q) < 0 \text{ for all } q \in (0,1). \tag{$\mathsf{Strict}\,\mathsf{RS}$}
\end{equation}

\begin{theorem}[Informal, see \Cref{thm:partition-fn-and-covariance}]
    \label{thm:intro-cov-bd}
    Under \eqref{eq:strict-RS-condition}, with probability $1-e^{-\Omega(N^{1/5})}$ over the randomness of $H$, $\|\Cov(\mu_H)\|_{\op} = O(1)$.
\end{theorem}
This is the first covariance bound to cover the entire replica symmetric phase with higher order interactions, and we believe it is interesting in its own right.
This result is sharp: in the complement of the replica symmetric regime, arguing as in \cite[Proposition 4.2]{AG24} shows that $\E \|\Cov(\mu_H)\|_{\op}$ is diverging, of order $\Omega(\sqrt{N})$.

The relation between \eqref{eq:SL-condition} and \eqref{eq:strict-RS-condition} is as follows.
First, \eqref{eq:strict-RS-condition} follows from \eqref{eq:SL-condition} by integrating twice.
Second, \eqref{eq:SL-condition} is equivalent to the condition that random exponential tilts of $\mu_H$ of any magnitude are typically replica symmetric.
This is needed for the algorithmic stochastic localization approach of \cite{HMP24}, and arises in the current work (where stochastic localization appears as an analysis tool, rather than as an algorithm) for a similar reason, see \Cref{rmk:SL-fundamental-barrier}.

The connection from \Cref{thm:intro-cov-bd} to high-probability covariance bounds on the tilted measures encountered along the localization process relies on a reduction developed in \cite{HMP24}.
This reduction implies that typically, the vast majority of the mass of these tilted measures live near a certain codimension-$2$ band passing through a \emph{TAP fixed point}, which behaves like a spin glass in two fewer dimensions.
The proof of \Cref{thm:intro-cov-bd} also builds on tools developed in \cite{HMP24}, and by one of the authors and Sellke in \cite{HS23}, which together provide high-precision control of partition functions in the replica symmetric regime. 

On the other hand, our approach also leads to several improvements over earlier results.
First, we obtain a sampler with total variation error $e^{-\Omega(N^{1/5})}$ with probability $1-e^{-\Omega(N^{1/5})}$, whereas \cite{HMP24} obtains total variation error $N^{-\eps}$ with probability $1-N^{-\eps}$, for small constant $\eps$.
Our total variation error is close to the best possible, as beyond the uniqueness threshold, at least a $e^{-O(N)}$ fraction of $\mu_H$ is typically trapped in metastable states \cite{AJ24a}, which are hard to reach.
Moreover, there is no longer a need to encode a mean estimator for the stochastic localization process (see below) directly in the algorithm; running a natural Markov chain is sufficient.

More conceptually, our work gives the first analysis of a Markov chain for this problem that ``sees" the benignness of a random initialization and overcomes the uniqueness threshold.

\subsection{Weak \Poincare inequalities beyond annealing}

The discussion thus far has been focused on proving mixing time bounds for Markov chains initialized at warm starts.
In fact, our framework extends beyond this and can be used to prove rapid mixing of a Markov chain initialized at a distribution that ``sees'' the different components of the target distribution. 
For instance, consider the simple scenario where the target distribution $\pi$ is a mixture of two disconnected component distributions, each of which satisfies a (true) \Poincare inequality. 
The disconnectedness means that the full distribution $\pi$ does not satisfy a true \Poincare inequality. However, if we initialize at a distribution that splits its mass equally between the two components, we would expect a Markov chain to rapidly mix to the target distribution.

How does one convert this belief to a (generalizable) proof? The key is that while the distribution may not satisfy a \Poincare inequality for \emph{all} functions, a variant of such an inequality does hold for functions encountered along the trajectory of the Markov chain. More concretely, we may prove the following theorem.

\begin{theorem}[Informal, see \Cref{th:approx-pi-main}]
    Consider the trajectory $(\nu_t)_{t \ge 0}$ of a Markov chain with stationary distribution $\pi$, initialized at a distribution $\nu_0$. Suppose that for all $s \le T$,
    \[ \calE\left( \frac{\dif \nu_s}{\dif \pi} , \frac{\dif \nu_s}{\dif \pi} \right) \ge \CPoi \left(\Var_{\pi} \left[ \frac{\dif \nu_s}{\dif \pi} \right] - \delta\right). \]
    Then,
    \[ \chitwo{\nu_T}{\pi} \le e^{-2\CPoi T} \chitwo{\nu_0}{\pi} + \delta. \]
\end{theorem}

We remark that our earlier equation \eqref{eq:intro-wpi} is a near-immediate consequence of the above. Returning to the above example with two disconnected components, if $\nu_s$ placed exactly half its mass on each of the two components, the error $\delta$ can be taken to be $0$.

For our first application in \Cref{sec:mixture-advice}, we use this picture of how the initialization can capture ``symmetries'' in the distribution.

\parhead{Sampling from mixtures of log-concave distributions with advice.}
An example of a distribution where we can take advantage of ``symmetries'' is the following. Suppose we have a distribution $\pi$ which is a mixture of $K$ distributions
\[ \pi = \sum_{i=1}^{K} p_i \pi_i, \]
each of which is well-connected (e.g., satisfies a \Poincare inequality). 
We do not expect a Markov chain to rapidly mix to $\pi$ from a worst-case initialization. Does the scenario change if we initialize more cleverly? To be concrete, suppose we are given $m$ samples $x_1,\ldots,x_m$ from $\pi$, and initialize our Markov chain at the empirical distribution $\sum_{i=1}^{m} \delta_{x_i}$. If the component measures $(\pi_i)$ are ``far apart'' and do not interact with each other, we would expect the Markov chain to rapidly mix from this initialization if the fraction of points in each cluster is (approximately) equal to the correct fraction $p_i$. On the other hand, if the component measures were very close together, we would expect their mixture to also satisfy a \Poincare inequality.

However, it is unclear how to translate this intuition to a proof. In previous work \cite{KV23}, sampling guarantees are provided for this algorithm, but the running time has a doubly exponential dependence on the number of components $K$.
Our second illustration of weak \Poincare inequalities provides high-probability sampling guarantees for this problem, by running Langevin diffusion for time that is polynomial in all parameters involved. 
We refer the reader to \Cref{sec:mixture-advice} for the details of the theorem statement and its (self-contained) proof.

This problem is studied extensively in an independent work of Koehler, Lee, and Vuong \cite{KLV24}.
Motivated by the success of \emph{score matching} methods in modern machine learning, they prove that Langevin dynamics and Glauber dynamics converge to the stationary distribution when initialized from the above empirical distribution under similar conditions to our setting, even if the Markov chain updates come from a slightly perturbed distribution (i.e. if they were learned by a score matching algorithm).
They also use their techniques to give an efficient algorithm for learning approximately low-rank Ising models.

\subsection{Related work}\label{sec:related-work}

\parhead{Markov chain mixing and localization schemes.}
The first use of the local-to-global phenomenon in mixing was in the work of \cite{ALOV24} on establishing rapid mixing of the ``basis exchange'' walk on bases of a matroid, which used the local-to-global theorem for simplicial complexes from \cite{KO20}.
Their approach was later formalized into the framework of \emph{spectral independence} \cite{ALO21}, which was widely successful in resolving numerous problems in algorithmic sampling and counting; see \cite{Liu23} for a comprehensive literature survey.

In the world of sampling from continuous distributions, most recent progress on the KLS conjecture on the \Poincare constant of isotropic log-concave distributions (see \cite{LV24} and the recent survey \cite{KL24}) has employed Eldan's stochastic localization \cite{Eld13}.
Later, stochastic localization was used in the work of Eldan, Koehler, and Zeitouni \cite{EKZ22} to analyze the \Poincare constant for Glauber dynamics on Ising models.
The seemingly unrelated techniques of spectral independence and stochastic localization approaches to analyzing mixing times were unified under the framework of localization schemes \cite{CE22}, which, as an application, also simplified the proof of \cite{EKZ22}.

\parhead{Weak \Poincare inequalities.}
The study of weak \Poincare inequalities was initiated in the work of Aida \cite{Aid98} and Mathieu \cite{Mat06} in the context of proving other functional inequalities.
The work of R{\"o}ckner and Wang \cite{RW01} observed the connection between a Markov chain satisfying a weak \Poincare inequality, and rapid mixing from ``sufficiently warm starts''.
We refer the reader to the monograph of Wang \cite[Chapter 4]{Wang06} for a comprehensive treatment of weak \Poincare inequalities and their implications to mixing and concentration.

Weak \Poincare inequalities are also related to the notion of $s$-conductance, a weakened version of conductance introduced in \cite{LS93} which has been used frequently in the literature on sampling from convex bodies (see \cite[Section 7.4.2]{Che23} for a textbook treatment).
This connection is explained in \cite{GMT06}.
We also refer the reader to \cite{CGG07}, which defines a notion of weak log-Sobolev inequality and uses it to derive a rapid mixing result.

The work \cite{EAEGP23} gives a sampling algorithm for the ferromagnetic random-field Ising model on a finite domain $D \subseteq \mathbb{Z}^d$, which follows an approach of chaining warm starts similar to the present work, inspired by convex body sampling literature \cite{LS93}.
\cite{EAEGP23} shows that in a certain parameter regime, the Glauber dynamics for this model satisfy a weak \Poincare inequality.
They then construct an increasing sequence of sub-domains $D_0 \subset D_1 \subset \cdots \subset D_T = D$ and show that a sample from the model on $D_i$ can be converted to a warm start for the model on $D_{i+1}$.
Since the weak \Poincare inequality implies mixing from a warm start, this yields a sampling algorithm based on running the Glauber dynamics on this increasing sequence of models.

The work \cite{AJPKV21b} introduces a related notion of restricted modified log-Sobolev inequality, which implies entropy contraction (without an additive error, in contrast to a weak \Poincare inequality) for all warm starts.
This is used to derive optimal mixing times for several Markov chains.

\parhead{Sampling from random initializations.}
The separation between worst-case mixing times and mixing from a random initialization has been studied in a variety of other settings.
\cite{CDLLPS12,BGZ24} characterize which product measure initializations enjoy rapid mixing in a temperature range where worst-case mixing is exponential for the Curie-Weiss Potts model. 
Notably, as discussed in \cite[Section 1.3]{BGZ24}, their analysis also characterizes mixing from initializations constructed by simulated annealing.  
\cite{LS16,LS17} show that a uniform initialization halves the mixing time for Glauber dynamics for the ferromagnetic Ising model on bounded degree graphs, such as the 1D lattice. 
\cite{GS22} introduces the notion of weak spatial mixing in a phase, and proves that Glauber dynamics for the ferromagnetic Ising model on the 2D lattice has rapid mixing when initialized uniformly at $\pm \vec{\bone}$. \cite{GS24} uses the same notion to study mixing from a similar random initialization for a certain natural Markov chain for the random cluster model.
\cite{BNN24} show rapid mixing for Glauber dynamics for the exponential random graph model when initialized from a carefully chosen Erd\H{o}s--R\'{e}nyi random graph. 

\parhead{Sampling from spherical spin glasses and algorithmic stochastic localization.}
There is a long history of work studying Markov chain dynamics on spin glasses.
An important line of work \cite{CHS93,CK93,BCKM98,BADG06,BAGJ20,CCM21,Sel24} studies the Langevin dynamics for spherical spin glasses on an $N$-independent time scale.
While the Langevin dynamics do not mix on this time scale, these works capture important statistics of the trajectory such as the energy attained by the Langevin dynamics after a given time, and uncover deep phenomena such as \emph{aging}.

Rapid mixing guarantees at sufficiently high temperature were obtained in \cite{GJ19} for the Langevin dynamics for spherical spin glasses, and in \cite{BB19,EKZ22,AJKPV22,ABXY24,AJKPV24,AKV24} for the Glauber dynamics for the Sherrington--Kirkpatrick model \cite{SK75} and Ising spin glasses.
These approaches show mixing from a worst-case initialization via a functional inequality.

Recently, \cite{AMS22,AMS23diffusion} introduced a new sampling algorithm based on simulating Eldan's stochastic localization scheme \cite{Eld13, Eld20}.
This approach has since been used in applications such as Bayesian posterior sampling \cite{MW23,MW24}, and is closely related to the denoising diffusions method in machine learning \cite{SWMG15,HJA20,SSKEP21} (see \cite{Mon23} for details).
The resulting algorithm samples in a wider range of temperatures, though with the weaker guarantee of vanishing \emph{Wasserstein} rather than total variation error.
The recent work \cite{HMP24} improved this guarantee to total variation, and the resulting algorithm succeeds to the same threshold \eqref{eq:SL-condition} as in the present work.

Within the algorithmic stochastic localization approach, the main task is to estimate the means of a sequence of exponential tilts of the Gibbs measure, which appear as the drift process of a stochastic differential equation parametrizing the localization process.
In \cite{AMS22}, this is achieved with an estimator based on approximate message passing (AMP), which is accurate to leading order.
\cite{HMP24} develops an improved estimator with a suitable correction term, which improves the algorithm's guarantee from Wasserstein to total variation error.

\parhead{Covariance bounds for spin glasses.}
There has been a great deal of recent work on covariance bounds for spin glasses \cite{BXY23,AG24,BSXY24}, in part due to the connection between covariance bounds and functional inequalities developed in the localization schemes literature.
In particular, \cite{AG24,BSXY24} address the case of the Sherrington--Kirkpatrick (SK) model, and \cite{BXY23} addresses the SK model with external field.

\subsection{Open problems}
\label{sec:open-probs}

We conclude with several open problems.

\parhead{Non-sampling guarantees for simulated annealing.}
While we initiate a study of simulated annealing to attain sampling guarantees, one could ask how to analyze simulated annealing beyond sampling. In recent work \cite{LMRRW24}, three of the authors, Liu, and Raghavendra introduce the framework of \emph{locally stationary distributions} to analyze slow-mixing Markov chains, and leverage it to obtain recovery guarantees for the spiked Wigner and stochastic block model inference problems.
We start by reiterating \cite[Problems 1.20 and 1.21]{LMRRW24} --- is simulated annealing  computationally optimal for random CSPs with planted solutions? 

Further, consider the problem of optimizing the Hamiltonian \eqref{eq:hamiltonian} of the mixed $p$-spin model.
Historically, simulated annealing was one of the earliest algorithms developed for this problem \cite{CHKW23}.
The works \cite{Mon21,Sub21,AMS21,Sel24b} develop algorithms that are optimal among suitably Lipschitz algorithms \cite{HS22} and conjecturally among all efficient algorithms.
The limiting energy obtained by natural Markov chain dynamics is an long-standing question in its own right \cite{CK93}, which was solved for pure models in \cite{Sel24} but is otherwise open. We ask:
\begin{problem}
    What energy does simulated annealing obtain when run on the Hamiltonian \eqref{eq:hamiltonian}?
\end{problem}

We refer the reader to \cite{MRT04,FFRT21} and references therein for relevant discussion.
We also ask the following question, which seems instrumental to making progress towards the above.

\begin{problem}
    How does a non-worst-case initialization (such as one constructed by simulated annealing) affect the locally stationary distribution that is reached by a Markov chain?
\end{problem}

Along similar lines, we have the following concrete question about understanding Markov chains from non-worst-case initializations.

\parhead{Worst-case combinatorial optimization via simulated annealing.}
The paradigm of solving a semidefinite program and rounding its solution has been extremely successful at achieving optimal approximation guarantees for a wide variety of combinatorial optimization problems, especially constraint satisfaction problems \cite{KKMO07,Rag08}.

However, on large families of instances (sparse ones for instance), the solutions produced by these SDPs can be refined locally to improve the approximation ratio, but these improvements do not match the corresponding hardness thresholds.
For example, for the problem of Max Cut, the classical SDP algorithm \cite{GW95} gives an $\alpha_{\mathrm{GW}}$-approximation for $\alpha_{GW}\approx 0.878$, and a local refinement \cite{HK22} produces an $\alpha_{\mathrm{GW}}+\Omega\left(\frac{1}{d^2}\right)$-approximation.
On the other hand, it is (UG-)hard \cite{Tre01} to approximate the max-cut better than $\alpha_{\mathrm{GW}}+O\left(\frac{1}{\sqrt{d}}\right)$.

\begin{problem}
    Does a Markov chain initialized at the SDP solution attain a $\alpha_{\mathrm{GW}} + \Omega\left( \frac{1}{\sqrt{d}} \right)$-approximation to the max-cut in a bounded degree graph?
\end{problem}

\parhead{Sampling from spin glasses up to the shattering threshold.}
It is conjectured that the Langevin diffusion with uniform random initialization samples from spherical $p$-spin models for inverse temperatures up to the shattering threshold \eqref{eq:non-shatter-condition} \cite{CHS93,CK93}. Similarly, this is conjectured for the Glauber dynamics Markov chain for models over the hypercube $\{\pm 1\}^N$ instead of the sphere $S_N$, for an analogous shattering threshold. As a start, can we show such guarantees for simulated annealing (as opposed to a fixed-temperature Markov chain from uniform initialization)?

\begin{problem}
    Does simulated annealing sample from $p$-spin models up to the shattering threshold?
\end{problem}

The failure of algorithmic stochastic localization beyond the \eqref{eq:SL-condition} condition \cite[Section 10]{HMP24} suggests that ideas beyond our proof strategy are required to prove the above.




\parhead{Simulated annealing in more general models.}
For sampling from the spherical $p$-spin model, our results show that simulated annealing succeeds in the regime \eqref{eq:SL-condition} where algorithmic stochastic localization succeeds.
At the level of proofs, these methods are also closely related, as both revolve around suitable control of the localization process: in the algorithmic stochastic localization approach, this is used to construct a mean estimator for the localized measures, and in our approach it is used to bound the localized measures' covariances.
These tasks are closely linked; see \Cref{rmk:SL-fundamental-barrier}.

One question is whether simulated annealing succeeds in more general models.
In particular, samplers based on algorithmic stochastic localization have been developed for the Sherrington-Kirkpatrick model in the replica symmetric regime \cite{AMS22,Cel24}, $p$-spin models over the hypercube \cite{AMS23diffusion}, and posteriors of spiked matrix models \cite{MW23}.
These samplers are proven to have vanishing Wasserstein error, and sampling with vanishing total variation error remains an open problem in these models.
It would be interesting to show that simulated annealing achieves this. 
More speculatively, we may ask if there is a general reduction from a sampling guarantee for algorithmic stochastic localization to one for simulated annealing.

\parhead{\bis.}
A major open problem in the field of approximate counting is settling the complexity of \bis: where the algorithmic task is to approximate the number of independent sets in a bipartite graph.
So far, algorithmic progress for this problem has been limited to restricted classes of graphs, such as lattices \& tori \cite{HPR19}, and expander graphs \cite{JKP20}.
Numerous interesting approximate counting problems have been shown to be \bis-hard \cite{CGM12,GJ12,CGGGJSV16,GSVY16}.
While vanilla Glauber dynamics fails at the corresponding sampling task, it is plausible that a variant of simulated annealing succeeds.

\begin{problem}
    Does (a simple variant of) simulated annealing succeed at sampling a uniformly random independent set in a bipartite graph?
\end{problem}

\parhead{Structural guarantees from weak \Poincare inequalities.}
According to physics heuristics, the Gibbs measure of a spherical mixed $p$-spin glass between $\beta_{\mathrm{uniq}}$ and $\beta_{\ShatteredPsyche}$ consists of one main cluster accounting for nearly all the mass, and metastable states with exponentially small mass that are poorly connected to the main cluster and each other.
We do not prove this picture, but the weak \Poincare inequality we obtain (up to $\beta_{\SL}$) is sufficient to imply a sampling guarantee for simulated annealing.
One open direction is to show that the above picture holds, and that the main cluster satisfies a genuine \Poincare inequality. 
More generally, one may ask: 
\begin{problem}
    If a distribution satisfies a weak \Poincare inequality, is it TV-close to a distribution satisfying a true \Poincare inequality?
\end{problem}
We note that \Cref{lem:langevin-close-to-poincare-weak-poincare,lem:close-to-poincare-weak-poincare} show a converse of this statement, that if we perturb a distribution satisfying a true \Poincare inequality (for the Langevin diffusion or Glauber dynamics Markov chains), the resulting distribution satisfies a weak \Poincare inequality. 



\subsection{Organization}
In \Cref{sec:tech-overview}, we give a technical overview of how we use weak \Poincare inequalities to analyze simulated annealing for our main application of sampling from spherical $p$-spin distributions. 

In \Cref{sec:prelims}, we cover some basic preliminaries that will be useful.
Then, in \Cref{sec:weak-functional}, we formally define weak functional inequalities and establish some of their basic properties. 

In \Cref{sec:mixture-advice}, we demonstrate the effectiveness of this framework by showing how to sample from a mixture of distributions satisfying \Poincare inequalities from data-based initializations. 

Our main application to spherical $p$-spin models spans \Crefrange{sec:stopped-ls}{sec:high-prob-cov}, and requires more background in stochastic localization and spin glass theory.
In \Cref{sec:stopped-ls}, we review some basic properties of stochastic localization and show how to adapt the framework of localization schemes from \cite{CE22} to prove weak functional inequalities. 
Then, in \Cref{sec:sec8}, we initiate the discussion of weak \Poincare inequalities for spherical $p$-spin models. 
To assist the reader in understanding the proof of a weak \Poincare inequality, we provide a separate technical overview in \Cref{sec:tech-overview-cov}. The rest of \Cref{sec:sec8} reduces the proof to proving high-probability covariance bounds for strictly replica-symmetric models with small external field, which is then established in \Cref{sec:high-prob-cov}.

\parhead{Acknowledgements.}
BH is extremely grateful to Andrea Montanari and Huy Tuan Pham for early discussions on this problem, and to Ahmed El Alaoui, Sinho Chewi, and Mark Sellke for enlightening conversations. We would also like to thank Sitan Chen, Jason Gaitonde, Kuikui Liu, and Francisco Pernice for insightful discussions. We would like to thank Thiago Bergamaschi for pointing out an error in an application to the ferromagnetic Potts model in an earlier version of this paper.

\section{Technical overview}    \label{sec:tech-overview}
Let $H$ be a Hamiltonian on state space $\Omega$, and let $\mu_H$ be its Gibbs distribution.
Our goal in this section is to describe our strategy to prove that simulated annealing succeeds at sampling.
In our application, $\Omega$ is the scaled sphere $\ScS_N \coloneqq \sqrt{N}\cdot \bbS^{N-1}$, and $\mu_H$ comes with an associated Markov chain known as \emph{Langevin diffusion}, which we denote with $P_H$.
For ease of exposition, we restrict the discussion to this setting, though much of it holds in a more general setting.

\begin{definition}[Simulated annealing, informal]
    Initialize at the uniform distribution on $\ScS_N$ (which is equal to $\mu_0$), and for each $i \in [m]$, run $P_{\frac{i}{m} H}$ for time $T$.
\end{definition}

For the sequel, we abbreviate $P_{\frac{i}{m}H}$ and $\mu_{\frac{i}{m}H}$ as $P_i$ and $\mu_i$, and we use $P_{i,t}$ to denote running $P_i$ for time $t$.
The strategy to prove that simulated annealing succeeds at sampling is to establish a weak \Poincare inequality for $P_i$ for all $i$.

Let $\Lap$ be the infinitesimal generator of $P_i$.
For functions $f,g:\Omega\to\R$, we define the \emph{Dirichlet form} $\calE(f,g)$ as $\E_{\mu_i}\bracks*{f\,\Lap\, g}$.
\begin{remark}
    In the case of Langevin diffusion for a distribution $\pi$, the Dirichlet form can be evaluated as
    \[
        \calE(f,g) = \E_{\mu_i}\angles*{ \grad f, \grad g }\mcom
    \]
    where $\grad$ denotes the Euclidean gradient if $\pi$ is supported on $\R^N$, and the Riemannian gradient on $\ScS_N$ if $\pi$ is supported on $\ScS_N$.
\end{remark}

As discussed in \Cref{sec:intro}, we say a Markov chain satisfies a \emph{weak \Poincare inequality} with parameters $(C,\eps)$ if
\[
    \calE(f,f) \ge C\cdot \Var[f] - \eps \cdot \norm*{ f - \E f }_{\infty}^2 - \eps \cdot \sup_{x\in\Omega} \norm*{\grad f(x)}^2 \mcom
\]
which implies the following mixing result \Cref{th:approx-pi-main} for the chi-squared divergence; see also \cite[Theorem 2.1]{RW01}. 
Defining $\nu_t$ as the distribution after running the Markov chain for time $t$ from initial distribution $\nu_0$, we have
\[
    \chi^2\parens*{ \nu_t \| \mu_i } \le \exp\parens*{-Ct} \cdot \chi^2\parens*{ \nu_0 \| \mu_i } + \eps \cdot \left\| \frac{\dif \nu_0}{\dif \mu_i} - 1 \right\|_{\infty}^2 + \eps \cdot \sup_{x\in\Omega}\left\| \grad \frac{\dif \nu_0}{\dif \mu_i}(x) \right\|^2 \mper
\]

\parhead{Analyzing simulated annealing with weak \Poincare inequalities.}
To see why the above statement plays well with simulated annealing, imagine plugging in initialization $\nu_0 = \mu_{i-1}$. 
By selecting the number of annealing steps $m = \poly(N)$, we can ensure $\left\| \frac{\dif \nu_0}{\dif \mu_i} - 1 \right\|_{\infty}$ and $\sup_{x\in\Omega}\left\| \grad \frac{\dif \nu_0}{\dif \mu_i}(x) \right\|$ are $O(1)$. 
The guarantee after running the Markov chain for some sufficiently large polynomial time $T$ is then
\[
    \chi^2\parens*{ P_{i, T} \mu_{i-1} \| \mu_i } \le O(\eps)\mcom
\]
which in particular implies
\[
    \dtv{P_{i, T} \mu_{i-1}}{\mu_i} \le O(\sqrt{\eps})\mper
\]
When we combine the above with the data processing inequality, we then get the following guarantee for $\nu_{m, T} \coloneqq P_{m,T}\cdots P_{2,T} P_{1,T}\nu_0$, the distribution that simulated annealing samples from.
\begin{align*}
    \dtv{P_{m,T}\cdots P_{1,T} \mu_0}{\mu_m} &\le \dtv{P_{m,T}\cdots P_{1,T}\mu_0}{P_{m,T}\mu_{m-1}} + \dtv{P_{m,T} \mu_{m-1}}{\mu_m} \\
    &\le \dtv{P_{m-1,T}\cdots\mu_0}{\mu_{m-1}} + O(\sqrt{\eps})\mper
\end{align*}
Applying the above inequality $m$ times tells us that $\dtv{\nu_{m,T}}{\mu_m} \le O\parens*{\sqrt{\eps}\cdot m}$.

We now turn our attention to the proof technique for showing a weak \Poincare inequality.

\parhead{How to prove weak \Poincare inequalities.}
Suppose our goal is to prove a weak \Poincare inequality for a measure $\pi$.
The high-level strategy in the localization schemes approach for proving a weak \Poincare inequality is to design a \emph{measure decomposition} of $\pi$: for some mixture distribution $\rho$, express $\pi$ as $\E_{\bz\sim\rho} \pi_{\bz}$. Refer to \Cref{sec:md} for a brief review of measure decompositions.
Once we have a measure decomposition in hand, establishing the following simple set of inequalities forms the crux of the argument.
Let $f$ be a function such that $\E_{\pi} f = 1$.
\begin{enumerate}
    \item \mylabel{item:cons-dirich}{conservation of Dirichlet form} \emph{Conservation of Dirichlet form.} $$\displaystyle \calE_{\pi} (f, f) \ge \E_{\bz\sim\rho} \calE_{\pi_{\bz}}(f,f)\mper$$
    In the case of Langevin diffusion, this is an equality, and
    in the case of Glauber dynamics, the inequality is true by a generic concavity argument; see, e.g. \cite[Page 19]{AJKPV22} or \cite[Page 5]{LMRW24}.
    \item \mylabel{item:WPI-comps}{weak \Poincare inequality for good component measures.} \emph{Weak \Poincare inequality for good component measures.} 
    $$\displaystyle \E_{\bz\sim\rho} \calE_{\pi_{\bz}}(f,f) \ge C \cdot \E_{\bz\sim\rho} \Var_{\pi_{\bz}}[f]\cdot\Ind[{\bz} \text{ ``good''}] - \eps \norm*{f-1}_{\infty}^2 - \eps\norm*{\grad f}_{\infty}^2\mcom$$ where the ``good'' $\pi_{\bz}$ are those which satisfy a $(c,\eps)$-weak \Poincare inequality.
    This inequality follows from the nonnegativity of norms and Dirichlet forms.
    \item \mylabel{item:cons-var}{approximate conservation of variance} \emph{Approximate conservation of variance.}
    $$\displaystyle \E_{\bz\sim\rho} \Var_{\pi_{\bz}}[f] \ge \alpha \cdot \Var_{\pi}[f]\mper$$
    This is one of the parts that depends nontrivially on $\pi$ and the decomposition $\rho$, and we discuss the general proof strategy for this portion based on localization schemes.
    \item \mylabel{item:high-prob-goodness}{high-probability goodness of components} \emph{High-probability goodness of component measures.}
    \[
        \Pr_{\bz\sim\rho}[\bz \text{ ``good''}] \ge 1-\eps\mper
    \]
    This part also requires analyzing the measure decomposition we design.
    Ideally, the measure decomposition presents us with ``simpler'' measures than $\pi$ itself.
\end{enumerate}
Once we have the above inequalities at hand, we get a $(c\alpha,2\eps)$-weak \Poincare inequality; see \Cref{lem:measure-decomp-weak-poincare} for details.

\parhead{How to construct a good measure decomposition.}
Henceforth, we restrict our attention to the case where $\pi = \mu_H$, the Gibbs distribution for a spherical mixed $p$-spin glass model.
In the discussion below, we fix $H$ as a typical Hamiltonian, and drop the phrase ``with high probability'' for events that occur with high probability over the randomness of $H$.

To construct our measure decomposition, we rely on Eldan's stochastic localization \cite{Eld13}.
Our inspiration is the use of stochastic localization as a tool for measure decomposition for proving \Poincare inequalities in the work of Chen and Eldan \cite{CE22}.
Stochastic localization is a measure-valued random process $(\mu_t)_{t\ge 0}$ described by:
\[
    \dif\mu_t(x) \propto \exp(\angles*{\by_t, x} - \frac{t}{2}\norm*{x}^2) \dif \mu_H(x)\mcom
\]
where $\by_t = \bsigma + B_t$ where $\bsigma\sim\mu_H$ and $(B_t)_{t\ge 0}$ is a standard Brownian motion; see \cite[Theorem 2]{EAM22} for a proof of why the above description of stochastic localization is equivalent to the more traditional definition via a stochastic differential equation that $\mu_t$ obeys.

We run stochastic localization up to a stopping time $\tau$, defined as
\[
    \tau \coloneqq \min\{t:0\le t \le T,\, \| \Cov(\mu_t) \| \ge K \text{ or } t = T \}\mcom
\]
where $T$ is chosen as a sufficiently large constant, independent of $N$.
We impose the constraint on the covariance matrix as it is relevant to satisfying approximate conservation of variance:
\cite[Eq. (20)]{CE22} proves that a measure decomposition based on stochastic localization run for time at most $T$ satisfies approximate conservation of variance with parameter $\alpha = \exp(-KT)$ if $\opnorm{\Cov(\mu_t)}$ is bounded by $K$ almost surely.
Hence, by construction, we automatically ensure that our measure decomposition satisfies approximate conservation of variance.

For the measure decomposition to ultimately be useful, we also need to argue that the component measures satisfy a weak \Poincare inequality with high probability.
Building on technical results in Huang, Montanari, and Pham \cite[Section 9.2]{HMP24}, we show that the stochastic localization process run up to time $T$ starting at $\mu_H$ gives a distribution satisfying an $(\Omega(1), \exp(-\Omega(n)))$-weak \Poincare inequality with probability $1-\exp(-\Omega(n))$ over the randomness of the stochastic localization path; see \Cref{lem:pspin-weak-pi-sl-end} for details.
Unfortunately, in the situation where the stochastic localization process stops before $T$, we do not have a simple way to show a weak \Poincare inequality, and for our analysis, treat $\bz$ arising from early stopping as ``bad''.

Thus, we have: $\Pr_{\bz\sim\rho}\bracks*{ \bz \text{ ``good''} } \ge 1 - \exp(-\Omega(n)) - \Pr[\tau < T]$.
To bound $\Pr[\tau < T]$, it is sufficient to prove a high-probability covariance norm bound on the entire stochastic localization path for $0\le t\le T$.
Most of the technical work in this paper is devoted to proving this covariance norm bound.
\begin{theorem}[Informal version of \Cref{lem:pspin-cov-bound-sl-path}]
    For a typical $H$, with probability $1-e^{-\Omega(n^{1/5})}$ over the randomness of the stochastic localization path, we have $\opnorm{\Cov(\mu_t)} < K$.
\end{theorem}
The proof of the covariance bound spans \Cref{sec:sec8,sec:high-prob-cov}; we give a detailed technical overview of how it is proved in \Cref{sec:tech-overview-cov}.

\section{Preliminaries}\label{sec:prelims}

\parhead{Notation}
\begin{itemize}
    \item We use $\ScS_N$ to denote the scaled $(N-1)$-sphere, $\sqrt{N} \cdot \bbS^{N-1}$.
    \item We use $\rho$ to denote the uniform measure over $S_N$. 
    \item Given $\sigma_1,\sigma_2 \in S_N$, we use $R(\cdot,\cdot)$ to denote the normalized inner product (i.e. the \emph{overlap})
    \[ R(\sigma_1,\sigma_2) \defeq \frac{\langle \sigma_1,\sigma_2\rangle}{N}. \]
    \item For an interval $I \subseteq [-1, 1]$ and $\bx \in \ScS_N$, define $\Band(\bx, I) \defeq \{\sigma \in \ScS_N: R(\sigma, \bx) \in I\}$.
    \item We use $c$ to denote small constants whose values may change from line to line, and $C$ to denote similarly fickle large constants.
    \item Let $f: \Omega \to \R$ be any function. We define $\osc(f) \defeq \sup f - \inf f$.
    \item Let $f: \Omega \to \R$ be a smooth function. If $\Omega \subseteq \R^N$, then $\grad f$ denotes its Euclidean gradient. If $\Omega \subseteq \ScS_N$, then $\grad_{\sp} f$ denotes the Riemannian gradient on $\ScS_N$. 
    When the correct notion of gradient is clear from context, by an abuse of notation we will suppress this distinction and simply write $\grad f$.
\end{itemize}

\subsection{Measure decompositions}\label{sec:md}
Our framework for proving weak functional inequalities relies on the notion of a measure decomposition. 
\begin{definition}[Measure decomposition]
    Let $\pi$ be a distribution on $\R^N$. Let $\rho$ be a mixture distribution, also on $\R^N$, which indexes into a family of mixture components $\{\pi_z\}_{z \in \R^N}$. 
    We say that $(\rho, \pi_z)$ is a measure decomposition for $\pi$ if 
    \[
        \pi = \E_{\bz\sim\rho} \pi_{\bz}\mper
    \]
\end{definition}
One reason measure decompositions are useful is that they compose nicely with worst-case functional inequalities, as shown in the following lemma.
\begin{lemma}[\cite{BB19,AJKPV22,CE22}]\label{lem:worst-case-md}
    Let $\pi$ be a distribution over $\Omega \subseteq \R^N$, and $\pi = \E_{\bz \sim \rho} \pi_{\bz}$ a measure decomposition of $\pi$ such that
    \begin{itemize}
        \item for all functions $f$, $\E_{\bz \sim \rho} \Var_{\pi_{\bz}}[f] \ge \Cvar \Var_{\pi}[f]$, and
        \item Every $\pi_{\bz}$ satisfies a $\CPoi$-\Poincare~inequality with respect to Langevin diffusion.
    \end{itemize}
    Then, $\pi$ satisfies a $\CPoi\Cvar$-\Poincare inequality.
\end{lemma}
In \Cref{lem:measure-decomp-weak-poincare}, we will show an average-case relaxation of the above result, that $\pi$ satisfies a weak \Poincare inequality if \emph{most} measures in the decomposition satisfy weak \Poincare inequalities.
Then, in \Cref{sec:stopped-ls}, we construct explicit measure decompositions using the localization schemes framework introduced in \cite{CE22}.
This will show weak \Poincare inequalities for our measures of interest.

Besides proving functional inequalities, measure decompositions have also been directly used for sampling and inference (see, e.g., \cite{MW24,LMRRW24}).

\subsection{Langevin diffusion}\label{sec:langevin-def}
In this paper, we study Langevin diffusion on $\R^N$ and the scaled sphere $\ScS_N$. 
These definitions can be directly generalized to the setting of Riemannian manifolds, but we do not comment further on this.
\begin{definition}[Langevin diffusion on $\R^N$]
    Let $\pi$ be a distribution on $\R^N$ with density at $x$ proportional to $e^{-V(x)}$ for some function $V$. The \emph{Langevin diffusion} process with stationary distribution $\pi$ is the solution to the stochastic differential equation
    \[ \dif Z_t = - \grad V(Z_t) \dif t + \sqrt{2} \dif B_t, \]
    where $(B_t)_{t \ge 0}$ is a standard Brownian motion.
\end{definition}
\begin{definition}[Langevin diffusion on $\ScS_N$]
    Let $\pi$ be a distribution on $\ScS_N$ with $\dif \pi(x) \propto e^{-V(x)} \dif \rho(x)$, where $V: \ScS_N \to \R$. The Langevin diffusion process with stationary distribution $\pi$ is the solution to the stochastic differential equation
    \[ \dif Z_t = - \grad_{\sp} V(Z_t) \dif t + \sqrt{2} \dif B_t, \]
    where $(B_t)_{t \ge 0}$ is a standard spherical Brownian motion.
    (For a textbook introduction to spherical Brownian motion, see \cite{Hsu02}.)
\end{definition}
\begin{fact}[{\cite[Example 1.2.17]{Che23}}]
    The Langevin diffusion SDE with stationary distribution $\pi$ is reversible with respect to $\pi$. In particular, the ergodicity of the process implies that $\KL{\Law(Z_t)}{\pi} \xrightarrow{t \to \infty} 0$.
\end{fact}

Furthermore, it is well-known that Langevin diffusion on $\R^N$ with respect to a strongly log-concave stationary distribution converges rapidly.

\begin{definition}
    Let $\pi$ be a distribution over $\R^N$ with density proportional to $e^{-V}$. $\pi$ is said to be \emph{$\alpha$-strongly log-concave} if $V$ is $\alpha$-strongly convex, that is, $\grad^2 V \psdge \alpha I$.
\end{definition}

\begin{fact}[{\cite[Theorem 1.2.24]{Che23}}]\label{fact:lsi-mixing}
    Let $\pi$ be a distribution satisfying a log-Sobolev inequality with constant $\LSI$, in that for any differentiable function $f : \R^N \to \R_{> 0}$,
    \[ \E_{\pi} \|\grad \sqrt{f}\|^2 \ge \LSI \Ent_{\pi}[f]. \]
    Then, if $\pi_t$ is the distribution at time $t$ of Langevin diffusion,
    \[ \KL{\pi_t}{\pi} \le \KL{\pi_0}{\pi} e^{-\LSI \cdot t}. \]
    Furthermore, $\alpha$-strongly log-concave distributions $\pi$ satisfy a log-Sobolev inequality with constant $\alpha$.
\end{fact}

\section{Weak functional inequalities}\label{sec:weak-functional}
In this paper, we study continuous-time Markov chains.
\begin{definition}[Markov semigroup]
    Let $(X_t)_{t \ge 0}$ denote a continuous-time Markov process on \emph{state space} $\Omega$. 
    Let $(P_t)_{t \ge 0}$ be the associated \emph{Markov semigroup} operator; $P_t$ acts on functions $f: \Omega \to \R$ via $P_tf(x) = \E[f(X_t) | X_0 = x]$. 
    Throughout, we assume that the semigroup is \emph{reversible} with respect to stationary distribution $\pi$. 
    Furthermore, let $\Lap$ denote the \emph{infinitesimal generator} of $P_t$, i.e.,  $P_t = e^{-t\Lap}$. 
    For functions $f, g: \Omega \to \R$, we define the \emph{Dirichlet form} as $\calE(f, g) = \E_{\pi}[f\Lap g]$.
\end{definition}
See e.g. \cite[Section 1.2]{Che23} for a textbook treatment.
Of particular interest to us are the two settings where the semigroup corresponds to a discrete-time Markov chain or the Langevin diffusion defined in \Cref{sec:langevin-def}. 
In these cases, the Dirichlet form satisfies the following explicit identities.
\begin{fact}[Dirichlet form from discrete-time Markov chain]
    Let $P$ be the transition matrix of a reversible discrete-time Markov chain with stationary distribution $\pi$.
    We can define an associated continuous-time semigroup operator $(P_t)_{t \ge 0}$ by setting $\Lap = I-P$. 
    The Dirichlet form for the continuous-time dynamics satisfies 
    \[
        \calE(f, g) \coloneqq \E_{\bx\sim\pi} \E_{\by\sim_P\bx} (f(\bx)-f(\by))(g(\bx)-g(\by))\mper
    \]
    Here, for a probability distribution $\mu$, we say $\bx\sim\mu$ to denote a sample $\bx$ from $\mu$, and we use $\by\sim_P\bx$ for a single transition from $\bx$ according to $P$.    
\end{fact}
\begin{fact}[Dirichlet form for Langevin diffusion]
    We will need the following explicit identities for the Dirichlet form for Langevin diffusion.
    \begin{enumerate}[label=\normalfont{(\arabic*)}]
        \item When $(P_t)_{t \ge 0}$ corresponds to Langevin diffusion on $\R^N$ with stationary distribution $\pi$, the Dirichlet form is $\calE(f, g) = \E_{\pi}[\angles{\grad f, \grad g}]$.
        \item When $(P_t)_{t \ge 0}$ corresponds to Langevin diffusion on $\ScS_N$ with stationary distribution $\pi$, the Dirichlet form is $\calE(f, g) = \E_{\pi}[\angles{\grad_{\sp} f, \grad_{\sp} g}]$. 
    \end{enumerate}
\end{fact}

\begin{definition}
    We say $\pi$ satisfies a \emph{weak \Poincare inequality} if for some error functional $\Error: \R_{>0}^{\Omega} \to \R_{\ge 0}$ and $\CPoi > 0$,
    \[
        \Var_{\pi} [f] \le \frac{1}{\CPoi}\cdot\calE(f, f) + \Error(f)\mper
    \]
    Similarly, we say $\pi$ satisfies a \emph{weak modified log-Sobolev inequality} if for some error functional $\Error : \R_{> 0}^{\Omega} \to \R_{\ge 0}$ and $\MLSI \ge 0$,
    \[
        \Ent_{\pi}[f] \le \frac{1}{\MLSI} \cdot \calE(f, \log f) + \Error(f).
    \]
\end{definition}

\begin{theorem}
    \label{th:approx-lsi-main}
    Consider the trajectory $(\nu_t)_{t \ge 0}$ of a reversible continuous-time Markov chain with stationary distribution $\pi$, initialized at the distribution $\nu_0$, and suppose that $\pi$ satisfies a weak MLSI with parameters $\Error$ and $\MLSI$. Fix $T > 0$, and set $\Lambda_T$ to be the distribution on $[0,T]$ with density $\Lambda_T(s) = \frac{e^{\MLSI}}{e^{\MLSI T} - 1} \cdot e^{\MLSI s}$. Then,
    \[ \KL{\nu_T}{\pi} \le e^{-\MLSI T} \KL{\nu_0}{\pi} + \E_{\bs \sim \Lambda_T}[\Error(\tfrac{\dif \nu_{\bs}}{\dif \pi})]. \]
\end{theorem}
\begin{proof}
    Let $f_0 = \frac{\dif \nu_0}{\dif pi}$, and let $f_t = P_t f_0 = \frac{\dif \nu_t}{\dif \pi}$ (this last equality holds due to reversibility).
    For ease of notation, set $\Error_t = \Error(f_t)$ for $t \ge 0$.
    Recalling that $\calE(f_t,\log f_t) = - \frac{\dif}{\dif t} \KL{\nu_t}{\pi}$, the weak MLSI says that
    \[ -\calE(f_t,\log f_t) + \LSI \cdot \KL{\nu_t}{\pi} - \LSI \cdot \Error_t \le 0, \] 
    so
    \[ \frac{\dif}{\dif t} \left( e^{\LSI t} \cdot \KL{\nu_t}{\pi} - \LSI \int_{0}^{t} e^{\LSI s} \Error_s \dif s \right) \le 0. \]
    Therefore,
    \[ e^{\LSI T} \cdot \KL{\nu_T}{\pi} - \LSI \int_0^T e^{\LSI s} \Error_s \dif s \le \KL{\nu_0}{\pi}, \]
    and
    \[ \KL{\nu_T}{\pi} \le e^{-\LSI T} \KL{\nu_0}{\pi} + \LSI \int_0^T e^{\LSI (s-T)} \Error_s. \]
    Noting that $\Lambda_T(s) = \frac{\LSI}{e^{\LSI T} - 1} \cdot e^{\LSI s} \ge \LSI e^{\LSI (s-T)}$, the above implies that
    \[ \KL{\nu_T}{\pi} \le e^{-\LSI T} \KL{\nu_0}{\pi} + \E_{s \sim \Lambda_T}\left[ \Error_s \right], \]
    as desired.
\end{proof}

By essentially the same proof, we obtain the analogous result for weak \Poincare inequalities.
\begin{theorem}
    \label{th:approx-pi-main}
    Consider the trajectory $(\nu_t)_{t \ge 0}$ of a (continuous-time) Markov chain with stationary distribution $\pi$, initialized at the distribution $\nu_0$, and suppose that $\pi$ satisfies a weak \Poincare inequality with parameters $\Error$ and $\CPoi$. Fix $T > 0$, and set $\Lambda_T$ to be the distribution on $[0,T]$ with density $\Lambda_T(s) = \frac{e^{2\CPoi}}{e^{2\CPoi T} - 1} \cdot e^{2\CPoi s}$. Then,
    \[ \chitwo{\nu_T}{\pi} \le e^{-2\CPoi T} \chitwo{\nu_0}{\pi} + \E_{\bs \sim \Lambda_T}[\Error(\tfrac{\dif \nu_{\bs}}{\dif \pi})]. \]
\end{theorem}
For the analysis of the annealed Langevin dynamics, we will also require the following definition. 
For $f : \Omega \to \bbR$, let $\osc(f) \defeq \sup(f) - \inf(f)$, and let $\grad f$ denote the Riemannian gradient.
\begin{definition}[Weak functional inequalities for Langevin]\label{def:weak-pi-langevin}
    We say a distribution $\pi$ on $\Omega \subseteq \R^N$ or $\Omega \subseteq \ScS_N$ satisfies a $(\CPoi,\eps)$-weak \Poincare inequality if for all differentiable functions $f$,
    \[ \Var_{\pi}[f] \le \frac{1}{\CPoi} \cdot \calE(f,f) + \eps \cdot (\osc(f)^2 + \sup_{x \in \Omega} \norm{\grad f}^2). \]
    Similarly, we say $\pi$ satisfies a  $(\MLSI,\eps)$-weak modified log-Sobolev inequality if for all differentiable functions $f$,
    \[ \Ent_{\pi}[f] \le \frac{1}{\MLSI} \cdot \calE(f,\log f) + \eps \cdot (\osc(\sqrt{f})^2 + \sup_{x \in \Omega} \norm{\grad f}^2). \]
\end{definition}
\begin{remark}
    As mentioned in the beginning of this section, by replacing the Riemmanian gradient with the discrete gradient, an analogous theory can be developed for annealed Glauber dynamics; see \Cref{def:glauber-weak-poincare}. 
\end{remark}
\begin{remark}
    A weak \Poincare inequality with sufficiently good parameters implies a true \Poincare inequality. Indeed, any low conductance cut limits on the region of valid $(\CPoi, \eps)$. Hence, by Cheeger, one can conclude that Langevin satisfies a true \Poincare inequality, with some loss in parameters.
\end{remark}
We shall typically use weak \Poincare inequalities with functions $f$ that have expectation $1$, where we bound $\osc(f) \le 2 \|f-1\|_{\infty}$.


\subsection{Properties of weak functional inequalities}
In this section, we state some crucial properties of weak functional inequalities for Langevin diffusion on $\R^N$ or $\ScS_N$. 
With minor modifications, the same results hold for Glauber dynamics on finite state spaces; see \Cref{sec:glauber-app} for formal details.
\begin{lemma}
    \label{lem:langevin-close-to-poincare-weak-poincare}
    Let $\pi$ be a distribution on $\R^N$ or $\ScS_N$ satisfying a $\CPoi$-\Poincare inequality for Langevin diffusion, and $\pi'$ a distribution such that $\dtv{\pi}{\pi'} \le \delta$. Then, $\pi'$ satisfies a  $\left(\CPoi, \delta \max(\CPoi^{-1},1)\right)$-weak \Poincare inequality for Langevin diffusion.
\end{lemma}
\begin{proof}
    There exists a coupling $\mathcal{C}$ of $(\pi,\pi')$ such that for $(x,x') \sim \mathcal{C}$, $\Pr[x\neq x'] \le \delta$.
    Thus,
    \begin{align*}
        \calE_{\pi'}(f,f) 
        &= \E_{\pi'} \norm{\grad f}^2 \\
        &\ge \E_{\pi} \norm{\grad f}^2 - \delta \sup \norm{\grad f}^2 \\
        &\ge \CPoi \Var_{\pi}[f] - \delta \sup \norm{\grad f}^2.
    \end{align*}
    Let $I = [\inf f, \sup f]$.
    Note that $\Var_{\pi}[f] = \inf_{a \in I} \E_{\pi}[(f-a)^2]$.
    For each $a\in I$, 
    \[
        \E_{\pi}[(f-a)^2]
        \ge \E_{\pi'}[(f-a)^2] - \delta \cdot \osc(f)^2,
    \]
    and therefore 
    \begin{equation}
        \label{eq:variance-tv-approximation}
        \Var_{\pi}[f]
        \ge \Var_{\pi'}[f] - \delta \cdot \osc(f)^2.
    \end{equation}
    Combining with the above shows
    \[
        \calE_{\pi'}(f,f) 
        \ge \CPoi \Var_{\pi'}[f]
        - \delta \lt(\CPoi \cdot \osc(f)^2 + \sup \norm{\grad f}^2\rt). \qedhere
    \]
\end{proof}

As foreshadowed previously, measure decompositions compose well with weak functional inequalities. Indeed, the following lemma can be viewed as a relaxation of the setup to prove genuine functional inequalities (cf. \Cref{lem:worst-case-md}).
\begin{lemma}
    \label{lem:measure-decomp-weak-poincare}
    Let $\pi$ be a distribution over $\R^N$ or $\ScS_N$, and $\pi = \E_{\bz \sim \rho} \pi_{\bz}$ a measure decomposition of $\pi$ such that
    \begin{itemize}
        \item for all functions $f$, $\E_{\bz \sim \rho} \Var_{\pi_{\bz}}[f] \ge \Cvar \Var_{\pi}[f]$, and
        \item with probability $1-\eta$ over $\bz \sim \rho$, $\pi_{\bz}$ satisfies a  $(\CPoi,\delta)$-weak \Poincare~inequality with respect to Langevin diffusion.
    \end{itemize}
    Then, $\pi$ satisfies a $\left( \CPoi\Cvar, \frac{\delta+\eta}{\Cvar} \right)$-weak \Poincare inequality.
\end{lemma}
\begin{proof}
    Let us say that $\bz$ is \emph{good} if $\pi_{\bz}$ satisfies a weak \Poincare~inequality, and $f$ be a function. 
    Then,
    \begin{align*}
        \calE_{\pi}(f,f) &= \E_{\bz \sim \rho} \calE_{\pi_{\bz}}(f,f) \\
            &\ge \E_{\bz \sim \rho} \calE_{\pi_{\bz}}(f,f) \bone_{\bz\text{ is good}} \\
            &\ge \E_{\bz \sim \rho} \CPoi \Var_{\pi_{\bz}}[f]  \bone_{\bz\text{ is good}} - \delta\CPoi \cdot (\osc(f)^2 + \sup \norm{\grad f}^2)\\
            &= \E_{\bz \sim \rho} \CPoi \Var_{\pi_{\bz}}[f] - \delta \CPoi \cdot (\osc(f)^2 + \sup \norm{\grad f}^2) - \E_{\bz \sim \rho} \CPoi \Var_{\pi_{\bz}}[f] \bone_{\bz\text{ is not good}} \\
            &\ge \CPoi \E_{\bz \sim \rho} \Var_{\pi_{\bz}}[f] -(\delta\CPoi + \eta\CPoi) \cdot (\osc(f)^2 + \sup \norm{\grad f}^2) \\
            &\ge \Cvar\CPoi \Var_{\pi}[f] - (\delta\CPoi + \eta\CPoi) \cdot (\osc(f)^2 + \sup \norm{\grad f}^2)\mper
    \end{align*}
    The desired follows.
\end{proof}

\subsection{Weak \Poincare inequalities and annealed Markov chains}\label{sec:annealed-mc}
The notion of weak functional inequalities defined in \Cref{def:weak-pi-langevin} can be naturally applied in the context of simulated annealing, which we now define.
\begin{definition}[Annealing scheme]
    Let $H$ be a Hamiltonian over $\Omega$, and $(\mu_\beta)_{\beta \ge 0}$ the class of distributions over $\Omega$ with $\mu_\beta(\sigma) \propto e^{\beta H(\sigma)}$. For each $\beta \le \beta_0$, let $P = P_\beta$ be a (reversible and ergodic) Markov chain with stationary distribution $\mu_\beta$. 

    An \emph{(inverse) temperature schedule} is any function $\beta: \R_{\ge 0} \to \R_{\ge 0}$. 
    An \emph{annealing scheme} $\calA$ is the time-inhomogeneous Markov chain such that at time $t$, one applies the Markov chain $P_{\beta(t)}$.
\end{definition}

Of interest is the temperature schedule of the form $t \mapsto \delta \cdot \left\lfloor \frac{t}{T} \right\rfloor$, with the chain being run for time $T \cdot \left(\frac{\beta_0}{\delta} + 1\right)$.

\begin{theorem}
    \label{th:annealing}
     Let $T,\delta > 0$ such that $k_0 \defeq \frac{\beta_0}{\delta}$ is an integer. Suppose that for each $\beta = k\delta$ for $0 \le k \le k_0$, $\mu_{\beta}$ satisfies a  $\left(\CPoi,\eps\right)$-weak \Poincare~inequality for $P_{\beta}$. Consider the annealing scheme given by schedule $t \mapsto \delta \cdot \left\lfloor \frac{t}{T} \right\rfloor$, run for total time $T \cdot \left(\frac{\beta_0}{\delta} + 1\right)$. Let $\nu$ be the output distribution of this annealing scheme. Then,
    \[ \dtv{\nu}{\mu_{\beta_0}} \le \frac{\beta_0}{\delta} \cdot \left[(1 + \delta \sup \norm{\grad H})e^{2\delta\|H\|_{\infty}} - 1\right] \cdot O \left( e^{-2\CPoi T} + \eps \right)^{1/2}. \]
\end{theorem}

\begin{remark}
    Setting $\eps = 0$ and $\delta = \beta_0$ matches the guarantees of \cite{CE22} (after applying Pinsker's inequality).
\end{remark}

\begin{proof}
    We shall prove the above using a simple inductive argument -- our goal will be to show that initialized at $\mu_\beta$, the $P_{\beta+\delta}$ Markov chain run for time $T$ yields a distribution sufficiently close (in total variation distance) to $\mu_{\beta+\delta}$. The total variation distance between the distribution that the annealed Markov chain outputs and the true distribution $\mu_{\beta_0}$ is then upper bounded by the sum of these total variation errors.

    Let $\nu^{(r,k)}$ be the distribution obtained by running the annealed Markov chain initialized with $\mu_{r\delta}$ until inverse temperature $k\delta$. In particular, $\nu^{(r,k)}$ corresponds to the result of running our annealed Markov chain for $T(k-r)$ time, and $\nu^{(k.k)} = \mu_{k\delta}$.
    We are interested in bounding $\dtv{\nu^{(0,k_0)}}{\mu_{\beta_0}}$. We have
    \begin{align*}
        \dtv{\nu^{(0,k_0)}}{\mu_{\beta_0}} &= \dtv{\nu^{(k_0-1,k_0)}}{\mu_{\beta_0}} + \sum_{1 \le r \le k_0 - 1} \left( \dtv{\nu^{(r-1,k_0)}}{\mu_{\beta_0}} - \dtv{\nu^{(r,k_0)}}{\mu_{\beta_0}}\right)\\
            &\le \dtv{\nu^{(k_0-1,k_0)}}{\mu_{\beta_0}} + \sum_{1 \le r \le k_0-1} \dtv{\nu^{(r-1,k_0)}}{\nu^{(r,k_0)}} \tag{Triangle inequality} \\
            &\le \dtv{\nu^{(k_0-1,k_0)}}{\mu_{\beta_0}} + \sum_{1 \le r \le k_0-1} \dtv{\nu^{(r-1,r)}}{\nu^{(r,r)}} \tag{Data processing} \\
            &= \sum_{1 \le r \le k_0} \dtv{\nu^{(r-1,r)}}{\mu_{r\delta}}.
    \end{align*}
    We now turn to controlling the error functional $\osc(f)^2 + \sup \norm{\grad f}^2$. 
    Fix an arbitrary $\beta$, and set $f$ to be the likelihood ratio $\frac{\dif \mu_{\beta}}{\dif \mu_{\beta+\delta}}$. Then, 
    \begin{align*}
        \|f - 1\|_{\infty} &\le \left\| \frac{e^{-\delta H}}{\E_{\mu_{\beta+\delta}} e^{-\delta H}} - 1 \right\|_{\infty} \\
        &\le \left\| \frac{e^{-\delta H}-1}{\E_{\mu_{\beta+\delta}} e^{-\delta H}} \right\|_{\infty} + \left| \frac{1}{\E_{\mu_{\beta+\delta}} e^{-\delta H}} - 1 \right| \\
        &\le \frac{e^{\delta \|H\|_{\infty}} - 1}{e^{-\delta\|H\|_{\infty}}} + \frac{e^{\delta\|H\|_{\infty}} - 1}{e^{-\delta\|H\|_{\infty}}} \le 2 \cdot (e^{2\delta\|H\|_{\infty}} - 1).
    \end{align*}
    Hence, $\osc(f) \le 4 \cdot (e^{2\delta\|H\|_{\infty}} - 1)$. 
    Next, a simple computation yields
    \begin{align*}
        \norm{\grad f} &= \frac{\delta e^{-\delta H}}{\E_{\mu_{\beta + \delta}} e^{-\delta H}} \norm{\grad H}\\
        &\le 2\delta \cdot e^{2\delta \norm{H}_{\infty}} \norm{\grad H}\mcom
    \end{align*}
    so we have $\sup \norm{\grad f} \le 2\delta \cdot e^{2\delta \norm{H}_{\infty}} \sup \norm{\grad H}$.
    
    Since each $\mu_{r\delta}$ satisfies a  $(\CPoi, \eps)$-weak \Poincare~inequality, \Cref{th:approx-pi-main} with the above calculation implies that
    \begin{align*}
        \dtv{\nu^{(r-1,r)}}{\mu_{r\delta}}^2 &\le \chitwo{\nu^{(r-1,r)}}{\mu_{r\delta}} \\
            &\le e^{-2\CPoi T} \cdot \chitwo{\mu_{(r-1)\delta}}{\mu_{r\delta}} + \eps \cdot (16(e^{2\delta\|H\|_{\infty}}-1)^2 + 4(\delta e^{2\delta \norm{H}_{\infty}} \sup \norm{\grad H})^2) \\
            &\le (16(e^{2\delta\|H\|_{\infty}}-1)^2 + 4(\delta e^{2\delta \norm{H}_{\infty}} \sup \norm{\grad H})^2) \left( e^{-2\CPoi T}  + \eps \right).
    \end{align*}
    Plugging this back into the earlier sequence of equations completes the proof.
\end{proof}

\begin{remark}
    While the proof above has been stated for the annealing scheme where at time $t$ the Hamiltonian is of the form $\sigma \mapsto \beta(t) \cdot H(\sigma)$, the proof immediately extends to essentially any annealing scheme that changes the Hamiltonian ``slowly'', in that if $H_t$ is the Hamiltonian at time $t$, $\|H_{t+T}-H_{t}\|_{\infty} \le \delta$ for all $t$. A concrete example of such a scheme that might work better than the vanilla annealing is that which at time $t$ has as Hamiltonian $\sigma \mapsto H(\beta(t) \cdot \sigma)$.
\end{remark}


\section{Vignette: Sampling from mixture models with advice}\label{sec:mixture-advice}

We are interested in the following question.

\begin{displayquote}
    Let $\pi$ be a distribution over $\R^N$ with density proportional to $e^{-V}$. Given oracle access to the gradient $\grad V$, when is it possible to efficiently produce samples that are close (in total variation distance) to $\pi$?
\end{displayquote}

We begin with an overview of existing results towards the above question.
Recall from \Cref{fact:lsi-mixing} that for distributions satisfying a \Poincare inequality, such as strongly log-concave distributions, Langevin diffusion enjoys rapid mixing.
Beyond this setting, however, very little is known. \cite{BCESZ22,CWZZ24} prove certain ``local mixing'' guarantees for Langevin diffusion on non-log-concave distributions, but these do immediately not translate to sampling guarantees. The works \cite{GLR18,LRG18,GTC24} use Langevin diffusion-based algorithms to sample from mixtures of log-concave distributions. Furthermore, the first of these papers proves that it is hard to sample from a mixture of two Gaussian distributions with distinct covariance matrices given access to just the gradient $\grad V$. 

In \cite{KV23}, the first theoretical guarantees are provided for a new model designed to circumvent this issue, where in addition to being given access to the gradient $\grad V$, we are also given ``advice'' in the form of $m$ samples from the distribution (also see \cite{NHHZW20} and \cite{Hin10,GLZZW18,XLZW16} for related discussion). In particular, they show that when the stationary distribution is a mixture of constantly many strongly log-concave distributions, Langevin diffusion initialized at the empirical measure on the advice gets close to the stationary distribution. However, their dependence on the number of components $K$ is doubly exponential. The main result in this section improves the doubly exponential dependence to a polynomial one for any mixture of distributions satisfying \Poincare inequalities.
Similar results are obtained by Koehler, Lee, \& Vuong \cite{KLV24}.

\begin{theorem}
    \label{th:langevin-fast}
    Let $\eps,\delta \in (0, 1)$, and let $\pi$ a mixture
    \[ \pi = \sum_{i=1}^{K} p_i \pi_i \]
    of distributions $(\pi_i)_{i=1}^K$, where each $\pi_i$ satisfies a \Poincare inequality with constant (at least) $\CPoi$. Further assume that $p_i \ge p_*$ for all $i$. Let $\nu_0$ be a random distribution over $\R^N$ such that $\E \nu_0 = \pi$, in that for any measurable subset $A$ of $\R^N$, $\E \nu_0(A) = \pi(A)$. Set
    \[ m = \Omega\left(\frac{\log(1/\delta)}{p_*\eps^2}\right). \]
    Let $\nu_1,\ldots,\nu_m$ be iid draws from $\nu_0$, and $\nu$ the uniform mixture over the $(\nu_i)_{i=1}^m$. Further 
    suppose that with probability at least $1-\delta$, $\chitwo{\nu_i}{\pi} \le M$. Denoting by $\mu_T$ the distribution attained by running Langevin diffusion for time $T$ initialized $\nu$, it holds that
    \[ \Pr\left[\chitwo{\mu_T}{\pi} \le \eps \right] \ge 1-O(\delta), \]
    for $T = \Omega\left(\frac{1}{\CPoi} \log\left( \frac{M}{\eps} \right)\right)$,
    where the probability is over the draws of $\nu_i$.
\end{theorem}

\begin{remark}
    One should think of $\nu_0$ as being the point mass distribution supported on a (random) sample drawn from $\pi$. Alternatively, one can think of $\nu_0$ as being the distribution obtained by drawing a sample $x_0$ according to $\pi$, then running Langevin diffusion for a short amount of time --- doing this would make the $\chi^2$-divergence $\chitwo{\nu_0}{\pi}$ finite. We also remark that a version of this proof goes through if we have that each $\pi_i$ satisfies a log-Sobolev inequality instead of a \Poincare inequality, working with KL divergences instead.
\end{remark}

\begin{proof}[Proof of \Cref{th:langevin-fast}]
    The idea of the proof will be to show that up to some additive error depending on the samples, $\pi$ does satisfy a \Poincare inequality with respect to the distributions along the path of Langevin diffusion initialized at the empirical distribution. This error corresponds to how imbalanced the samples are in terms of the mixture weights --- a straightforward concentration argument using Bernstein's inequality then shows that this error is small, so the $\chi^2$ divergence essentially decays exponentially fast, as if $\pi$ satisfied a true \Poincare inequality. 
    
    Let $f_t$ be the Radon-Nikodym derivative of $\mu_t$ (obtained by running Langevin diffusion initialized at $\nu$) with respect to $\pi$. By definition, we have
    \begin{align}
        \chitwo{\mu_t}{\pi} &= \E_{\pi}[f_t^2] - 1 \nonumber \\
            &= \sum_{i=1}^K p_i \left(\E_{\pi_i}[f_t^2] - 1 \right) \nonumber \\
            &= \sum_{i=1}^K p_i \Var_{\pi_i}[f_t] + \sum_{i=1}^{K} p_i \left( \E_{\pi_i}[f_t]^2 - 1 \right). \nonumber
    \end{align}
    Because each $\pi_i$ satisfies a \Poincare inequality, the first term is bounded as
    \[ \sum_{i=1}^{K} p_i \Var_{\pi_i}[f_t] \le \frac{1}{\CPoi} \sum_{i=1} p_i \E_{\pi_i} \|\grad f_t\|^2 = \frac{1}{\CPoi} \E_{\pi} \|\grad f_t\|^2. \]
    Consequently,
    \begin{equation}
        \chitwo{\mu_t}{\pi} \le \frac{1}{\CPoi} \cdot \E_{\pi} \|\grad f_t\|^2 + \sum_{i=1}^{K} p_i \left(\E_{\pi_i}[f_t]^2 - 1 \right).
    \end{equation}
    \Cref{th:approx-pi-main} then yields that
    \begin{align*}
        \chitwo{\mu_T}{\pi} &\le \chitwo{\mu_0}{\pi} \cdot e^{-\CPoi T} + \E_{s \sim \Lambda_T} \left[ \sum_{i=1}^{K} p_i \left( \E_{\pi_i}[f_t]^2 - 1 \right) \right] \\
            &\le M e^{-\LSI T} + \E_{s \sim \Lambda_T} \left[ \sum_{i=1}^{K} p_i \left( \E_{\pi_i}[f_t]^2 - 1 \right) \right].
    \end{align*}
    Above, we use that because the KL divergence to $\pi$ of each of the $\nu_i$ is at most $M$, so is that of the mixture $\mu_0 = \nu$.
    
    To conclude, we shall establish tail bounds on
    \[ \E_{s \sim \Lambda_T} \left[ \sum_{i=1}^K p_i \left( \E_{\pi_i}[f_s]^2 - 1 \right) \right]. \]
    For $1 \le j \le m$, let $f_s^{(j)}$ be the Radon-Nikodym derivative of $\mu_s^{(j)}$ with respect to $\pi$, where $\mu_s^{(j)}$ is the distribution obtained by running Langevin diffusion for time $s$ initialized at $\nu_j$. It is not difficult to see that $f_s = \frac{1}{m} \sum_{j=1}^{m} f_s^{(j)}$.\\
    First, for fixed $s$ and $j$, we use the fact that the $(\E_{\pi_i} [f_s^{(j)}])_j$ are independent mean $1$ random variables, with Hoeffding's inequality, to get tail bounds for $\E_{\pi_i}[f_s]^2 - 1$. We may use this to bound a certain Orlicz norm of this random variable --- this bound on the norm also transfers to $\E_{s \sim \Lambda_T} \left[ \sum_{i=1}^{K} p_i \left( \E_{\pi_i}[f_s]^2 - 1 \right) \right]$ as it is a convex combination of random variables with bounded Orlicz norm.  This immediately yields the desired tail bound.
    
    Fix $s$ and $i$. To start, we have the almost sure bounds
    \[ \frac{1}{p_*} = \frac{1}{p_*} \E_{\pi}[f^{(j)}_s] = \frac{1}{p_*} \sum_{r=1}^{K} p_r \E_{\pi_r}[f^{(j)}_s] \ge \E_{\pi_i}[f^{(j)}_s] \ge 0. \]
    Note that because the expected $\nu_j$ is equal to $\pi$, $\E_{\nu_j} \E_{\pi_i}[f^{(j)}_s] = 1$ for any $j$. Furthermore, because $\E_{\pi_i}\left[f_s^{(j)}\right]$ is a mean $1$ random variable which is bounded in $\left[0,\frac{1}{p_*}\right]$, its variance is at most $\frac{1}{p_*}$ (see e.g. \cite{BD00}).
    Bernstein's inequality implies that
    \[ \Pr\left[ \left| \E_{\pi_i} f_s - 1 \right| > t \right] = \Pr\left[ \left| \frac{1}{m} \sum_{i=1}^{m} \E_{\pi_i} \left[ f^{(j)}_s \right] - 1 \right| > t \right] \le 2 \exp\left( - \frac{mp_*}{2} \cdot \frac{t^2 }{1+t} \right). \]
    Thus, for any $t > 0$,
    \begin{align*}
        \Pr\left[ \left|\E_{\pi_i}[f_s]^2 - 1\right| > t \right] &\le \Pr\left[ \left| \E_{\pi_i} f_s - 1 \right| > \frac{t}{2(1+\sqrt{t})} \right] \\
            &\le 2 \exp\left( - \frac{mp_*}{8} \cdot \frac{\left(\frac{t}{1+\sqrt{t}}\right)^2}{1 + \frac{t}{1+\sqrt{t}}} \right).
    \end{align*}
    Now, consider the Orlicz norm $\|\cdot\|_{\psi}$ associated to the above family of tail bounds. As mentioned earlier, standard machinery may be used to go from the above tail bounds to a bound on the norm $\left\| \E_{\pi_i}[f_s]^2 - 1 \right\|_\psi$. Convexity of the norm yields the same bound on $\left\| \E_{s \sim \Lambda_T} \sum_{i=1}^{K} p_i \left( \E_{\pi_i}[f_s]^2 - 1 \right) \right\|_\psi$. Translating this back to a tail bound, we get that
    \[ \Pr\left[ \left|\E_{s \sim \Lambda_T} \sum_{i=1}^{K} p_i \left( \E_{\pi_i}[f_s]^2 - 1 \right) \right| > \frac{\eps}{2} \right] \le 2 \exp \left( -\frac{mp_*\eps^2}{10} \right) \le \delta. \]
    Conditioning on the above event not happening, we get that
    \[ \chitwo{\mu_T}{\pi} \le \chitwo{\mu_0}{\pi} \cdot e^{-\CPoi \cdot T} + \frac{\eps}{2} \le \eps, \]
    as desired.
\end{proof}

\section{Stopped localization schemes}\label{sec:stopped-ls}
\subsection{Localization schemes}\label{sec:loc-schemes}
We review some basic notions for the localization schemes framework introduced in \cite{CE22}. 
\begin{definition}[Linear-tilt localization scheme]
    Let $\mu = \mu_0$ be a probability measure,  $(\mu_t)_{t \in \Z_{\ge 0}}$ be a localization process.
    A linear-tilt localization scheme is one where $\mu_t$ is  defined by
    \[ \mu_{t+1}(x) = \mu_t(x) \left( 1 + \langle x-\mean(\mu_t) , Z_t\rangle \right) \]
    where $Z_t$ is a random variable with $\E[Z_t | \mu_t] = 0$ and $\mean(\mu_t)$ denotes the mean of $\mu_t$.
\end{definition}
For our main application to $p$-spin models, we will focus on a continuous-time version of  linear-tilt localization known as stochastic localization \cite{Eld13}.
\begin{definition}[Stochastic localization]\label{def:sl}
    Let $\mu$ be a probability measure on $\Omega \subseteq \R^N$, $(B_t)_{t \ge 0}$ be a standard Brownian motion on $\R^N$. The stochastic localization process with driving matrix $(C_t)_{t \ge 0}$ is a localization process $(\mu_t)_{t \ge 0}$ with $\mu_0 = \mu$ and 
    \begin{align*}
        \mu_t(x) \propto \mu_0(x) \exp(-\tfrac{1}{2}\angles{x, \Sigma_t x} + \angles{y_t, x}),
    \end{align*}
    where $\Sigma_t = \int_0^t C_s^2 \dif s$ and $y_t = \int_{0}^{t} C_s^2 \mean(\mu_s) \dif s + C_s \dif B_s$.
\end{definition}
A crucial property of these localization schemes is that establishing (approximate) conservation of variance reduces to bounding the covariance matrices of the intermediate distributions $\mu_t$.
\begin{lemma}[{Conservation of variance for linear-tilt \cite[Claim 22]{CE22}}]\label{lem:cvar-linear-tilt}
    Let $(\mu_t)_{t \in \Z_{\ge 0}}$ be a linear-tilt localization process.
    Suppose that for all $t \le T$ we have
    \begin{align*}
        \opnorm{\Cov(Z_t | \mu_t)^{1/2} \cdot \Cov(\mu_t) \cdot \Cov(Z_t | \mu_t)^{1/2}} \le K_t,
    \end{align*}
    where $K_t \in [0, 1]$.
    Then for any function $\varphi$,
    \[ \frac{\E \Var_{\mu_{T}}[\varphi]}{\Var_{\mu}[\varphi]} \ge \prod_{t = 0}^{T-1} (1 - K_t). \]
\end{lemma}
\begin{lemma}[Conservation of variance for stochastic localization]\label{lem:cvar-sl}
    Let $(\mu_t)_{t \ge 0}$ be a stochastic localization process with driving matrix $(C_t)_{t \ge 0}$. 
    Suppose that for all $t \le T$ we have 
    \begin{align*}
       \opnorm{C_t^{1/2} \cdot \Cov(\mu_t) \cdot C_t^{1/2}} \le K_t
    \end{align*}
    where $K_t \in [0, 1]$. 
    Then for any function $\varphi$,
    \[ \frac{\E \Var_{\mu_{T}}[\varphi]}{\Var_{\mu}[\varphi]} \ge e^{-\int_{0}^T K_t \dif t}. \]
\end{lemma}
\subsection{Proving weak \Poincare inequalities using stopped localization schemes}
To apply \Cref{th:annealing}, we required weak \Poincare~inequalities for the measures of interest. To show these, we next introduce a generic tool to prove these using \Cref{lem:measure-decomp-weak-poincare}, building on the localization schemes framework introduced in \Cref{sec:loc-schemes}. 
Let $\mu$ be a distribution. Using a localization scheme, we would like to design a measure decomposition $\mu = \E_{\bz \sim \rho} \mu_{\bz}$ such that
\begin{itemize}
    \item for all functions $f$, $\Var_{\pi}[f] \le \Cvar \E_{\bz \sim \rho} \Var_{\pi_{\bz}}[f]$, and
    \item with probability $1-\eta$ over $\bz \sim \rho$, $\pi_{\bz}$ satisfies a  $(\CPoi,\delta)$-weak \Poincare~inequality.
\end{itemize}

One way to ensure the first condition --- approximate conservation of variance --- is to simply stop the localization scheme whenever it fails to hold.
Indeed, the following lemma immediately follows from \Cref{lem:cvar-linear-tilt}.
\begin{lemma}
    Let $\mu = \mu_0$ be a measure, and let $(\mu_t)_{t \in \Z_{\ge 0}}$ be a linear-tilt localization process defined by
    \[ \mu_{t+1}(x) = \mu_t(x) \left( 1 + \langle x-\mean(\mu_t) , Z_t\rangle \right) \]
    for some random variable $Z_t$ with $\E[Z_t | \mu_t] = 0$.
    Let $T > 0$ be an arbitrary stopping time and $0 \le K_t < 1$ for each $t \ge 0$, and consider the stopping time
    \[ \tau = T \land \inf_{t \ge 0} \left\{ \opnorm{\Cov(Z_t | \mu_t)^{1/2} \cdot \Cov(\mu_t) \cdot \Cov(Z_t | \mu_t)^{1/2}} \ge K_t \right\}. \]
    Then, for any function $\varphi$,
    \[ \frac{\E \Var_{\mu_{\tau}}[\varphi]}{\Var_{\mu}[\varphi]} \ge \prod_{t \ge 0} (1 - K_t). \]
\end{lemma}

Similarly, we have the following lemma for stochastic localization, which follows from \Cref{lem:cvar-sl}.

\begin{lemma}
    \label{lem:stopped-sl}
    Let $\mu = \mu_0$ be a measure, and $(\mu_t)_{t\ge 0}$ be a stochastic localization process with driving matrix $(C_t)_{t\ge 0}$. Let $T,K > 0$ be constant parameters, and consider the stopping time
    \[ \tau = T \wedge \inf_{t \ge 0} \left\{ \opnorm{C_t^{1/2} \cdot \Cov(\mu_t) \cdot C_t^{1/2}} \ge K \right\}. \]
    Then, 
    \[ \frac{\E \Var_{\mu_{\tau}}[\varphi]}{\Var_{\mu}[\varphi]} \ge e^{-TK}. \]
\end{lemma}


\begin{remark}
    \label{rem:stopping-beyond-sl}
    The localization process in the above lemmas can depend on $\varphi$, and need not be a linear-tilt localization. The more general requirement is that
    \[ \frac{\E\left[\Var_{\mu_{t+1}}[\varphi] | \mu_t\right]}{\Var_{\mu_t}[\varphi]} \ge K_t \quad\text{ or }\quad \frac{1}{\Var_{\mu_t}[\varphi]} \cdot \frac{\dif}{\dif s} \E \left[\Var_{\mu_s}[\varphi] \mid \mu_t\right] \bigg\rvert_{s = t} \ge K. \]
    This can always be achieved by stopping the localization process whenever these conditions fail to hold.
\end{remark}

With these elements in hand, we now show how to prove a weak \Poincare inequality using stopped localization schemes. 
\begin{lemma}
    \label{lem:stopped-scheme-to-weak-poi}
    Let $\mu = \mu_0$ be a measure, and $(\mu_t)_{t\ge 0}$ be a stochastic localization process with driving matrix $(C_t)_{t\ge 0}$. 
    Let $T,K>0$ be constant parameters. 
    Suppose that with probability $1-\eta_1$, it holds that $\opnorm{C_t^{1/2} \cdot \Cov(\mu_t) \cdot C_t^{1/2}} < K$ for all $t\in [0,T]$. Further suppose that with probability $1-\eta_2$, $\mu_T$ satisfies a $(\CPoi,\delta)$-weak \Poincare~inequality. Then, $\mu$ satisfies a $\left( \CPoi e^{-TK}, e^{TK} \left(\delta + \eta_1 + \eta_2\right)\right)$-weak \Poincare~inequality.
\end{lemma}
\begin{proof}
    As in \Cref{lem:stopped-sl}, define the stopping time
    \[ \tau = T \land \inf_{t \ge 0} \left\{ \opnorm{C_t^{1/2} \cdot \Cov(\mu_t) \cdot C_t^{1/2}} \ge K \right\}. \]
    Consider the measure decomposition $\mu = \E \mu_{\tau}$. By \Cref{lem:stopped-sl}, this decomposition is variance-conserving with parameter $e^{-K}$. By the hypothesis of the lemma, $\tau = T$ with probability $1-\eta_1$, and $\mu_1$ satisfies a weak \Poincare~inequality with probability $1-\eta_2$. Consequently, $\mu_\tau$ satisfies a weak \Poincare~inequality with probability at least $1-\eta_1-\eta_2$. \Cref{lem:measure-decomp-weak-poincare} completes the proof.
\end{proof}

\begin{remark}  \label{rem:other-ls-to-weak-poi}
    An analogous lemma to the above holds if $(\mu_t)_{t\in\Z_{\ge 0}}$ is any linear-tilt localization process.
\end{remark}

While it will not be used in this paper, we note that a similar method proves a weak \Poincare inequality for a natural Markov chain associated to a localization scheme.
This includes for example the restricted Gaussian dynamics; see \cite{CE22} for several other examples.

\begin{lemma}
    Let $\mu = \mu_0$ be a measure, and $(\mu_t)_{t\ge 0}$ be a stochastic localization process with driving matrix $(C_t)_{t\ge 0}$. 
    Let $T,K > 0$ be constant parameters.
    Consider the Markov chain $P$ given by $P_{x \to y} = \E \left[ \frac{\mu_T(x)\mu_T(y)}{\mu_0(x)} \right]$.
    Define the stopping time
    \[ \tau = T \land \inf_{t \ge 0} \left\{ \opnorm{ C_t^{1/2} \cdot \Cov(\mu_t) \cdot C_t^{1/2} } \ge K \right\}. \]
    If $\tau = T$ with probability at least $1-\delta$, then $P$ satisfies a $(e^{-TK} , \delta e^{TK})$-weak \Poincare inequality.
\end{lemma}
\begin{proof}
    For the Markov chain $P$, the Dirichlet form is given by $\calE_P(f,f) = \E \Var_{\mu_T}[f]$ (see, e.g., \cite[Proposition 19]{CE22}). We then have the chain of inequalities
    \begin{align*}
        \E \Var_{\mu_T}[f] &\ge \E \Var_{\mu_\tau}[f] - \delta \osc(f)^2 \\
            &\ge e^{-TK} \Var_{\mu_0}[f] - \delta \osc(f)^2.
    \end{align*}
    The first inequality here is immediate since
    \[ \E \Var_{\mu_\tau}[f] - \E \Var_{\mu_T}[f] = \E \Var_{\mu_\tau}[f] \bone_{\tau \ne T} \le \osc(f)^2 \Pr[\tau \ne T]. \]
    The second inequality follows from \Cref{lem:cvar-sl}.
\end{proof}

\begin{remark}
    As in \Cref{rem:stopping-beyond-sl}, the above lemma can be generalized to localization schemes other than stochastic localization.
\end{remark}


\section{Sampling from spherical $p$-spin models}
\label{sec:sec8}

In this section, we prove that simulated annealing samples from spherical spin glass models for models satisfying \eqref{eq:SL-condition}.
Recall that $\ScS_N = \sqrt{N}\cdot\bbS^{N-1}$.
For $\gamma_2,\gamma_3,\dots,\gamma_{p_*} \ge 0$, the mixed $p$-spin Hamiltonian $H_N:\ScS_N\to\R$ is defined by 
\begin{equation}
    \label{eq:p-spin-ham}
    H_N(\sigma) \coloneqq \sum_{p\ge 2} \frac{\gamma_p}{N^{(p-1)/2}} \sum_{i_1,\dots,i_p = 1}^N \bg_{i_1,\dots,i_p} \sigma_{i_1}\cdots\sigma_{i_p},
\end{equation}
for i.i.d. samples $\bg_{i_1,\dots,i_p}$ from $\calN(0,1)$.
This is the gaussian process on $\bbR^N$ with covariance
\[
    \E H_N(\sigma^1) \cdot H_N(\sigma^2) = N\cdot\xi\parens*{R(\sigma^1,\sigma^2)},
\]
where we recall the mixture function $\xi$ is defined by $\xi(s) = \sum_{p=2}^{p_*} \gamma_p^2 s^p$.
The algorithm we will study is the following simple annealing scheme for Langevin diffusion.
\begin{definition}[Annealed Langevin diffusion]
    \label{def:annealed-langevin-diffusion}
    Let $\delta_N, T_N > 0$ be parameters possibly depending on $N$. 
    For any $\beta \ge 0$, let $\mu_\beta \defeq \mu_{\beta H_N}$, where $H_N$ is the $p$-spin Hamiltonian.
    Annealed Langevin diffusion is the annealing scheme $\calA$ where $\beta(t) = \delta_N \floor{t/T_N}$ and $P_{\beta}$ is the Langevin semigroup operator for Gibbs distribution $\mu_\beta$. 
    In words, $\calA$ keeps $\beta$ constant for time $T_N$ and then increments $\beta$ by $\delta_N$.
\end{definition}

\begin{theorem} \label{thm:main-p-spin}
    Let $H_N$ be a mixed $p$-spin Hamiltonian whose mixture function $\xi$ satisfies \eqref{eq:SL-condition}, which we recall below: 
    \[
        \xi''(q) < \frac{1}{(1-q)^2} \text{ for all } q \in [0,1). 
    \]
    Let $\mu$ be the associated Gibbs measure over the scaled sphere $\ScS_N$, with
    \[ \dif \mu(\sigma) \propto \exp(H_N(\sigma)) \dif\rho(\sigma). \]
    With probability $1-e^{-cN^{1/5}}$ over the randomness of $H_N$, the following holds.
    For some parameters $\delta_N = O(N^{-4/5})$, $T_N = \Omega(N^{1/5})$, the output measure $\nu$ of the the annealed Langevin diffusion scheme with these parameters satisfies
    \[ \dtv{\nu}{\mu} \le e^{-cN^{1/5}}. \] 
\end{theorem}

\begin{remark}
    We expect that the error $e^{-cN^{1/5}}$ can be improved to $e^{-cN}$, matching the fact that a $e^{-O(N)}$ fraction of the Gibbs measure is typically trapped in metastable states between the uniqueness and shattering thresholds \cite{AJ24a}.
    However, we will not pursue this improvement in this paper.
\end{remark}

\begin{remark}
    \label{rmk:SL-fundamental-barrier}
    The condition \eqref{eq:SL-condition} is a fundamental barrier for stochastic localization, both as an algorithm and a proof technique.
    As was essentially shown in \cite[Section 10]{HMP24}, for models satisfying \eqref{eq:strict-RS-condition} but not \eqref{eq:SL-condition}, the means $\mean(\mu_t)$ along the localization process do not move stably, in the sense that there exist time intervals of width $o(1)$ in which $\mean(\mu_t)$ moves by $\Omega(N^{1/2})$.
    (The condition \eqref{eq:strict-RS-condition} is an artifact of the proof, and it is expected that the mean continues to move non-stably beyond the regime \eqref{eq:strict-RS-condition}).
    In the setting of \cite{HMP24}, this implies that their algorithmic simulation of the localization process fails, because approximate message passing will not estimate the mean at some times.
    In our setting, this implies that the covariance $\Cov(\mu_t)$, which arises as the derivative of $\mean(\mu_t)$, is genuinely not bounded in operator norm at some times, and thus the main input to our framework does not hold.
\end{remark}

\begin{remark}
    While the result above is stated for the continuous time Langevin \emph{diffusion}, the results therein can be adapted to the discretized Langevin Monte Carlo algorithm using standard tools, \`{a} la \cite[Part II]{Che23}, to obtain a polynomial time sampling algorithm.
\end{remark}

To prove the above, we shall use \Cref{th:annealing} in conjunction with \Cref{lem:stopped-scheme-to-weak-poi}.
For the remainder of this section, let $(\gamma_p)_{p \ge 2}$ be a sequence of weights such that the associated mixture function $\xi$ satisfies the condition \eqref{eq:SL-condition}.

\begin{notation*}[Measure decomposition for $p$-spin models]
    Let $\mu_{H_N} = \mu_0$. For a large constant time $T$, let $(\mu_t)_{0 \le t \le T}$ be the stochastic localization process with driving matrix $\Id$ (see \Cref{def:sl}).
\end{notation*}

\begin{lemma}[Covariance bound on stochastic localization path]
    \label{lem:pspin-cov-bound-sl-path}
    There exist constants $c,K$, depending only on $\xi$, such that for any constant $T>0$ the following holds with probability at least $1 - e^{-cN^{1/5}}$ over the randomness of $H_N$. If $(\mu_t)_{0 \le t \le T}$ is the (random) trajectory of stochastic localization initialized at $\mu_0 = \mu_{H_N}$, with probability $1-e^{-cN^{1/5}}$, $\opnorm{\Cov(\mu_t)} < K$ for all $0 \le t \le T$. In other words,
    \[ \Pr_{H_N} \left[ \Pr_{(\mu_t) \mid H_N}\left[\opnorm{\Cov(\mu_t)} < K \text{ for all } 0 \le t \le T\right] \ge 1-e^{-cN^{1/5}} \right] \ge 1-e^{-cN^{1/5}}. \]
\end{lemma}

\begin{lemma}[Weak \Poincare inequality for endpoint distributions]
    \label{lem:pspin-weak-pi-sl-end}
    There exists a constant $T$ depending only on $\xi$ such that the following holds with probability at least $1-e^{-cN}$ over the randomness of $H_N$. With probability at least $1-e^{-cN}$, the (random) measure $\mu_T$ satisfies a $(c, e^{-cN})$-weak \Poincare inequality. In other words,
    \[ \Pr_{H_N}\left[ \Pr_{\mu_T \mid H_N}\left[ \mu_T \text{ satisfies a $(c, e^{-cN})$-weak \Poincare inequality} \right] \ge 1-e^{-cN} \right] \ge 1 - e^{-cN}. \]
\end{lemma}

Let us first see how these two lemmas imply the main theorem.

\begin{proof}[Proof of \Cref{thm:main-p-spin}]
    Fix some $0 \le \beta \le 1$. Note that the Hamiltonian $\beta H_N$ has mixture function $\xi_\beta(s) = \xi(\beta^2 s)$, and if $\xi$ satisfies \eqref{eq:SL-condition} then $\xi_\beta$ does as well. Plugging in \Cref{lem:pspin-cov-bound-sl-path,lem:pspin-weak-pi-sl-end} into \Cref{lem:stopped-scheme-to-weak-poi} implies that with probability at least $1-e^{-cN^{1/5}}$, $\mu_{\beta H_N}$ satisfies a $(c, e^{-cN^{1/5}})$-weak \Poincare inequality. A union bound implies that with probability $1-e^{-cN^{1/5}}$, for all $\beta$ encountered along the annealing schedule, $\mu_{\beta H_N}$ satisfies a $(c, e^{-cN^{1/5}})$-weak \Poincare inequality.
    
    By \cite[Proposition 2.3]{HS22}, with probability $1-e^{-cN}$, $\|\grad H_N\|_{\infty} = O(\sqrt{N})$.
    The same argument implies that with probability $1-e^{-cN}$, $\|H_N\|_\infty = O(N)$.  
    With probability $1-e^{-cN^{1/5}}$, all three of these events occur, and \Cref{th:annealing} completes the proof.
\end{proof} 
We conclude this subsection by proving \Cref{lem:pspin-weak-pi-sl-end}.
\begin{proof}[Proof of \Cref{lem:pspin-weak-pi-sl-end}]
	Let $\dif\mu_t(\sigma) \propto e^{H_{N,T}(\sigma)}\,\dif\sigma$ for $H_{N,T}(\sigma) = H_N(\sigma) + \la \by, \sigma \ra$.
	Let 
   \[\ScS_N(\by) = \{\sigma \in \ScS_N : R(\by,\sigma) > 0\}.\]
	Let $\bU \in \R^{N\times (N-1)}$ be a matrix whose columns are an orthonormal basis of the orthogonal complement of $\by$.
	Let $\hby = \sqrt{N} \by / \|\by\|$ be $\by$ (which is a.s. nonzero) scaled to length $\sqrt{N}$, and define the map $\sigma_{\by}(\rho) : \R^{N-1} \to \ScS_N(\by)$ by
	\[
		\sigma_{\by}(\rho) = \frac{\hby + \bU \rho}{\sqrt{1 + R(\rho,\rho)}}.
	\]
	This is the inverse of the map that first stereographically projects $\ScS_N(\by)$ from the origin to $\hby + \bU \R^{N-1}$, the plane tangent to $\ScS_N$ at $\hby$, and then maps the resulting point to coordinates given by $\bU$.
	Let $\eps_0 = 0.1$, and $A = \{\rho \in \R^{N-1} : \|\rho\|^2 \le \eps_0 N\}$, and note that
	\[
		\sigma_{\by}(A) = \{\sigma \in \ScS_N : R(\sigma,\hby) \ge (1 + \eps_0)^{-1/2}\}
	\]
    is a spherical cap around $\hby$.
    Let $A ' \coloneqq \sigma_{\by}(A)$.
	By arguments in \cite[Subsection 9.2]{HMP24}, there exists a measure $\nu$ (denoted $\tilde \nu^{{\mathsf{proj}}}_{H_N,\by}$, see Eq. 2.10 therein) such that the following holds with probability $1-e^{-cN}$.
	\begin{itemize}
		\item The push-forward of $\nu_{|A}$ through $\sigma_{\by}$ coincides with $(\mu_T)_{|A'}$. (Lemma 9.5 therein.)
		\item $\mu_T(A') = 1-e^{-cN}$. (Lemma 9.6 therein states this with $1-o_N(1)$ in place of $1-e^{-cN}$, but the proof implies bound $1-e^{-cN}$, as this is the bound given by Proposition 5.12 used therein.)
		\item $\nu(A) = 1-e^{-cN}$. (Corollary 9.7 therein, modulo the same issue of $1-o_N(1)$ versus $1-e^{-cN}$, which is addressed the same way.)
		\item $\nu$ is $\Omega(1)$-strongly log-concave. (Proposition 9.8 therein.)
	\end{itemize}
	By the well-known Bakry-\'{E}mery condition (see, e.g., \cite[Section 1.2.3]{Che23}), on this event $\nu$ satisfies a $\CPoi$-\Poincare inequality for some $\CPoi = \Omega(1)$.
	We will transfer this inequality to a $(\CPoi,e^{-cN})$-weak \Poincare inequality for $\mu_T$.
	Consider a smooth test function $f : \ScS_N \to \R$ and let $\wt{f} : \R^{N-1} \to \R$ be defined by $\wt{f} = f \circ \sigma_{\by}$.
	Since $\dtv{\mu_T}{(\mu_T)_{|A'}} = e^{-cN}$ and $\dtv{\nu}{\nu_{|A}} = e^{-cN}$, and $\osc(f') \le \osc(f)$, arguing as in \eqref{eq:variance-tv-approximation} shows
	\begin{align*}
		\Var_{\mu_T}(f) 
		&\le \Var_{(\mu_T)_{|A'}}(f) + e^{-cN} \osc(f) \\
		&= \Var_{\nu_{|A}}(\wt{f}) + e^{-cN} \osc(f) \\
		&\le \Var_{\nu}(\wt{f}) + 2e^{-cN} \osc(f).
	\end{align*}
	By the \Poincare inequality for $\nu$ and the definition of the Dirichlet form for Langevin diffusion,
	\[
		\Var_{\nu}(\wt{f})
		\le \frac{1}{\CPoi} \cdot \calE_\nu(\wt{f},\wt{f})
		= \frac{1}{\CPoi} \cdot \E_{\nu} [\|\grad \wt{f}\|^2]
	\]
	By \cite[Proof of Lemma 9.5]{HMP24}, the map $\sigma_{\by}$ has Jacobian $J_{\sigma_{\by}}$ satisfying $\|J_{\sigma_{\by}}\|_{\op} \le 1$, and thus for all $\rho \in \R^{N-1}$, 
	\[
		\|\grad \wt{f}(\rho)\|
		= \|\grad (f \circ \sigma_{\by}) (\rho)\|
		\le \|\grad f(\sigma_{\by}(\rho))\|.
	\]
	It follows that
	\begin{align*}
		\E_{\nu} [\|\grad \wt{f}\|^2]
		&\le \E_{\nu_{|A}} [\|\grad \wt{f}\|^2]
		+ e^{-cN} \sup \|\grad \wt{f}\|^2 \\
		&\le \E_{(\mu_T)_{|A'}} [\|\grad f\|^2]
		+ e^{-cN} \sup \|\grad f\|^2 \\
		&\le \E_{\mu_T} [\|\grad f\|^2]
		+ 2e^{-cN} \sup \|\grad f\|^2
	\end{align*}
	Combining the above shows
	\[
		\Var_{\mu_T}(f) 
		\le \frac{1}{\CPoi} \calE_{\mu_T}(f,f)
		+ 2e^{-cN}\lt(\osc(f) + \frac{1}{\CPoi} \sup \|\grad f\|^2 \rt).
	\]
	The result follows by adjusting $c$.
\end{proof}

\subsection{Technical overview for covariance bounds}\label{sec:tech-overview-cov}
The proof of the main theorem has boiled down to \Cref{lem:pspin-cov-bound-sl-path} --- we now give a high-level overview of our proof strategy for this. We wish to show that with very high probability ($1-e^{-\Omega(N^{1/5})}$), the covariance is bounded along the entire path $(\mu_t)_{0 \le t \le T}$ of stochastic localization. By performing a union bound over time and a standard perturbation argument, it suffices to show that for a fixed time $t \in [0, T]$, $\mu_t$ has bounded covariance with very high probability.

To do this, we recall an alternate view of stochastic localization \cite{AM22}. The measure at time $t$ of stochastic localization (with the identity driving matrix) is given as follows. First, draw $\bsigma \sim \mu_{H_N}$, and independently $\bg \sim \mathcal{N}(0,I_N)$. Then, $\mu_t$ has the same law as $\mu_{H_N , t\bsigma + \sqrt{t}\bg}$, in that
\[ \mu_{H_N , t\bsigma+\sqrt{t}\bg}(\wt{\sigma}) \propto \exp\left( H_N(\wt{\sigma}) + \langle t\bsigma + \sqrt{t}\bg , \wt{\sigma} \rangle \right). \]
As written, the covariance of this distribution is difficult to analyze --- the sample $\bsigma$ has very complicated correlations with the disorder of the Hamiltonian $H_N$, making it intractable.

\parhead{The planting trick.}
To deal with this, we will use the \emph{planting trick} introduced by Achlioptas and Coja-Oghlan \cite{AC08}.
The application of this method in the context of stochastic localization is by now standard \cite{AMS22,AMS23,HMP24}, and we review the main ideas for the reader's convenience. 
\begin{restatable}[Planted $p$-spin model]{definition}{planted}\label{def:planted}
    The planted measure $\mu_{\pl}$ is a joint law over a Hamiltonian $H_N$ and a \emph{spike} $\bx \in S_N$ given by
    \[ \dif \mu_{\pl}(H_N,x) \propto \exp\left( H_N(x) \right) \cdot \dif \rho(x) \cdot \dif \mu_{\nullmodel}(H_N), \]
    where $\rho$ is the uniform measure over $S_N$ and $\mu_{\nullmodel}$ is the law over $p$-spin Hamiltonians with mixture function $\xi$. We frequently abuse notation to let $\mu_{\pl}(H_N)$ denote the marginal of $\mu_{\pl}$ on $H_N$.
\end{restatable}
To provide further intuition for the above definition, consider the following alternate sampling interpretation of the planted model, which describes the distribution of $\bx$ conditioned on $H_N$.

\begin{fact}
    Consider the following inference problem. We start by sampling the spike $\bx \sim S_N$, sample $G^{(p)}$ as a rank-$p$ tensor with iid $\calN(0, 1)$ entries for $p \ge 2$, and for each $p$ let $M^{(p)} = -G^{(p)} + \frac{\gamma_p}{N^{(p-1)/2}} \bx^{\otimes p}$.

    Then, the posterior on $\bx$ after observing the tensors $(M^{(p)})_{p \ge 2}$ is of the form $\mu(\bx = \sigma \mid (M^{(p)})) \propto \exp(H_N(\sigma))$, where
    \[ H_N(\sigma) = \sum_{p \ge 2} \frac{\gamma_p}{N^{(p-1)/2}} \langle M^{(p)} , \sigma^{\otimes p} \rangle. \]
    Then, the joint law of $(H_N,\bx)$ is $\mu_{\pl}$.
\end{fact}

The above says that conditioned on $H_N$, the distribution of $\bx$ (according to $\mu_{\pl}$) is simply distributed as a sample according to $\mu_{H_N}$.
That is, the spike $\bx$ resulting in a Hamiltonian $H_N \sim \mu_{\pl}$ is exchangeable with a sample from $\mu_{H_N}$.

The latter of these interpretations will be very useful for us. When dealing with the measure at time $t$ of stochastic localization applied to the $p$-spin model, the primary issue was that it was unclear how to deal with the sample $\bsigma$ drawn from the Gibbs distribution.
However, if we could work with the planted $p$-spin model, this issue would be absent. Indeed, the exchangability of the spike and a sample implies that the law of $\mu_t$ applied to the planted model is given by
\begin{align*}
    \mu_{H_N,t\bx + \sqrt{t}\bg}(\wt{\sigma}) &\propto \exp\left( H_N(\wt{\sigma}) + \langle t\bx + \sqrt{t}\bg , \wt{\sigma} \rangle \right),
\end{align*}
where $\bx$ is the spike hidden in $H_N$. This decouples the randomness of the external field $t\bx + \sqrt{t}\bg$ and the disorder of the Hamiltonian $H_N$ that arises from the Gaussians $(G^{(p)})_{p \ge 2}$.

As was shown in \cite[Corollary 3.5]{HMP24} and recalled just below, the planted and null models are mutually contiguous.
Thus high-probability statements from one model transfer to the other, and it suffices to study the planted model.
\begin{displayquote}
    For all models satisfying \eqref{eq:strict-RS-condition}, the measures $\mu_{\nullmodel}(H_N)$ and $\mu_{\pl}(H_N)$ from \Cref{def:planted} are mutually contiguous, i.e., for any sequence of events $\calE_{N}$, $\mu_{\nullmodel}(\calE_N) \to 0$ whenever $\mu_{\pl}(\calE_N) \to 0$.
\end{displayquote}
The transfer from the $p$-spin model to the planted model may then be carried out by setting
\[ \mathcal{E}_N = \left\{ H_N  :  \Pr_{\substack{\bsigma \sim \mu_{H_N} \\ \bg \sim \mathcal{N}(0,I_N)}}\left[ \norm*{ \Cov\left( \mu_{H_N , t\bsigma + \sqrt{t}\bg} \right) } > K \right] < e^{-cN^{1/5}} \right\}. \]
This event is very complicated in the null model, but exchangeability makes it tractable in the planted model.
In the actual proof, we will require a stronger (quantitative) version of mutual contiguity; see \Cref{prop:quantiguity} for details.

Now, we must understand what the Hamiltonian in the planted model looks like conditioned on the spike.

\begin{fact}
    \label{fact:planting-trick-alternate-1}
    Consider the following process: sample $\bx \sim S_N$, $\wt{H}_N \sim \mu_{\nullmodel}$, and define $H_N$ by $H_N(\sigma) = \wt{H}_N(\sigma) + N \cdot \xi(R(\bx,\sigma))$. Then, the joint law of $(H_N,\bx)$ is $\mu_{\pl}$.
\end{fact}

Consequently, our goal is to bound the covariance of the distribution
\[ \mu_{t}(\sigma) \propto \exp\left( \wt{H}_N(\sigma) + N \cdot \xi \left( R\left( \bx , \sigma \right) \right) + \langle t \bx + \sqrt{t} \bg , \sigma \rangle \right). \]
for $\wt{H}_N \sim \mu_{\nullmodel}$ with mixture function $\xi$. Now, define $\xi_t$ by $\xi_t(s) = \xi(s) + ts$, and extend the definition of the $p$-spin model \eqref{eq:p-spin-ham} to allow a random linear term. Then,
\[ \mu_t(\sigma) \propto \exp\left( \underbrace{\wt{H}_{N,t}(\sigma) + N \cdot \xi_t(R(\bx,\sigma))}_{H_{N,t}(\sigma)} \right), \]
where $\wt{H}_{N, t} \sim \mu_{\nullmodel}$ with mixture function $\xi_t$.

\parhead{The TAP planted model.} We now turn to controlling the covariance matrix of these models.
As we will see below, it is relatively easier to bound the covariance matrix (in fact, the second moment matrix) of a model with zero or small external field. 
However, for any time $t>0$, $H_{N,t}$ has an external field.
We will use a method developed in \cite{HMP24} to reduce to the case of a model with zero or small external field.

Let $\boldm^{\true} = \mean(\mu_t)$.
The main intuition of this reduction is that the Gibbs measure concentrates near a codimension-$2$ band passing through $\boldm^{\true}$ and orthogonal to $\boldm^{\true}$ and $\bx$, and furthermore \emph{the model on this band is essentially a replica symmetric model with no external field}.
Moreover, one expects that both $R(\boldm^{\true},\boldm^{\true})$ and $R(\boldm^{\true},\bx)$ concentrate near a value $q_* = q_*(t)$ defined by $\xi_t'(q_*) = \frac{q_*}{1-q_*}$.

However, $\boldm^{\true}$ is a complicated function of $H_{N,t}$, so it is a priori difficult to reason about the joint distribution of $(\boldm^{\true},H_{N,t})$.
Thus, this reduction is formally carried out by conditioning on a \emph{TAP fixed point} $\boldm^\TAP$, which will serve as a proxy for $\boldm^{\true}$.
Define the TAP free energy
\[ \FTAP(\boldm) = H_{N,t}(\boldm) + \frac{N}{2} \cdot \theta(R(\boldm,\boldm)) + \frac{N}{2} \log (1 - R(\boldm,\boldm)), \]
where 
\[ \theta(s) = \xi(1) - \xi(s) - (1-s) \xi'(s). \]
As shown in \cite{HMP24}, for sufficiently small constant $\iota > 0$, with probability $1-e^{-cN}$ $\FTAP$ has a unique critical point $\boldm^\TAP$ in the region $\mathcal{S}_\iota$ defined by $R(\boldm,\boldm), R(\boldm,\bx) \in [q_*-\iota,q_*+\iota]$.
Due to the existence and uniqueness of $\boldm^\TAP$, it becomes possible to relate $H_{N,t}$ to a ``TAP-planted model" where one samples $\boldm^\TAP$ \emph{first}, and then samples $H_{N,t}$ conditional on $\nabla \FTAP(\boldm) = 0$: 
\begin{lemma}[See \Cref{lem:planted-to-TAP}; essentially due to \cite{HMP24}]
  For any small constant $\iota > 0$, the following holds.
  For any $H_{N,t}$-measurable event $\calE$, if 
  \[
    \sup_{\boldm^\TAP \in \mathcal{S}_\iota} \Pr(\calE | \nabla \FTAP(\boldm^\TAP) = 0) \to 0,
  \]
  then $\Pr(\calE) \to 0$.
\end{lemma}
Crucially, the conditional law of $H_{N,t}$ in the TAP-planted model is very tractable, as (for a fixed $\boldm^\TAP$) $\nabla \FTAP(\boldm^\TAP) = 0$ amounts to a linear constraint on the Gaussian process $H_{N,t}$.
The resulting explicit conditional law of $H_{N,t}$ is described in \Cref{lem:TAP-hamiltonian}.

\begin{remark}
  While it will not be relevant to our purposes, \cite{AMS22,AMS23,HMP24} have shown that $\boldm^{\TAP}$ typically approximates $\boldm^\true$ well, in the sense that $\|\boldm^\true - \boldm^\TAP\|^2 = O(1)$, thereby justifying the heuristic that $\boldm^\TAP$ is a proxy for $\boldm^\true$.
\end{remark}
\begin{remark}
    The idea of reducing to a TAP-planted model has also been used beyond the setting of sampling from spherical spin glasses.
    In the recent work \cite{Hua24}, an analogous reduction is used to obtain the capacity of the Ising perceptron.
    In this application, passage to the TAP-planted model is used to tightly control a partition function rather than to bound a covariance matrix.
\end{remark}
Consequently, we can now work within the TAP-planted model.
Let $\HTAP$ denote the Hamiltonian $H_{N,t}$ after conditioning on $\bx$ and $\nabla \FTAP(\boldm^\TAP) = 0$. 
As 
\[
  \Cov(\mu_{\HTAP}) \preceq \E_{\bsigma \sim \mu_{\HTAP}} (\bsigma - \mathbf{v})(\bsigma - \mathbf{v})^\top
\]
for any $\mathbf{v} \in \R^N$ (with equality at $\mathbf{v} = \boldm^\true$), it suffices to control the operator norm of 
\[
  \E_{\bsigma \sim \mu_{\HTAP}} (\bsigma - \boldm^\TAP)(\bsigma - \boldm^\TAP)^\top.
\]

\parhead{Reduction to slices of the sphere.}
Next, to control the covariance, let us decompose the sphere into codimension-$2$ slices
\[ \ScS(a,b) \coloneqq \left\{ \sigma \in S_N : R(\sigma, \boldm) = \parens*{1 + \frac{a}{\sqrt{N}}} R(\boldm, \boldm) ,\, R(\sigma, \bx) = \parens*{ 1 + \frac{b}{\sqrt{N}} } R(\bx,\boldm) \right\}, \]
with the central slice centered at $\boldm = \boldm^\TAP$.
Let $\mu_{a,b}$ be the measure $\mu_{H_{\TAP}}$ conditioned to lie in the codimension-$2$ slice $\ScS(a,b)$.

The concentration of the Gibbs measure described in the previous section implies that, viewed as random variables of a sample $\sigma \sim \mu_{\HTAP}$, $a$ and $b$ are well-concentrated around $0$. 
Let $v_{a,b}$ be the center of $\ScS(a,b)$. The covariance of the distribution may be bounded as
\begin{align}   
    \notag
    \Cov(\mu) &\psdle \E\parens*{ \bsigma - \boldm } \parens*{ \bsigma - \boldm }^{\top}\\
    \notag
    &= \E_{(a,b)} {\E}_{\bsigma\sim\mu_{a,b}} \parens*{ \bsigma - v_{a,b} + v_{a,b} - \boldm } \parens*{ \bsigma - v_{a,b} + v_{a,b} - \boldm }^{\top} \\
    \notag
    &\psdle 2 \E_{(a,b)}\E_{\bsigma\sim\mu_{a,b}} \parens*{\bsigma - v_{a,b}}\parens*{\bsigma - v_{a,b}}^{\top} + 2\E_{(a,b)} (v_{a,b}-\boldm)(v_{a,b}-\boldm)^{\top} \\
    \label{eq:covariance-decomposition}
    &\psdle 2 \E_{(a,b)}\E_{\bsigma\sim\mu_{a,b}} \parens*{\bsigma - v_{a,b}}\parens*{\bsigma - v_{a,b}}^{\top} + 2 \E_{(a,b)} O \left(a^2 + b^2\right).
\end{align}

One can interpret $v_{a,b}$ as explaining the variation within the slice originating from the $\boldm$ and $\bx$ directions. 
Hence, as alluded to in the previous discussion about the TAP planted model, the key fact is that under $\mu_{a,b}$, the recentered sample $\bsigma - v_{a,b}$ is a sample from a spherical spin glass in two lower dimensions, as can be shown by calculating the covariance of the (conditioned) Gaussian process $H_\TAP$ restricted to this slice.
This verification is carried out in \Cref{cor:slice-dist}.

These codimension-$2$ models have the crucial property that the spherical spin glass on the slice $a=b=0$ is a model satisfying \eqref{eq:strict-RS-condition} with no external field (i.e. degree-$1$ term), while nearby slices have a small (random) external field of magnitude $\sqrt{a^2+b^2}$.
In particular, the first term of \eqref{eq:covariance-decomposition} requires bounding the second moment of a Gibbs sample from a strictly RS model with small (random) external field. 
As a result, \eqref{eq:covariance-decomposition} would be bounded if we proved the following.

\begin{enumerate}
    \item Let $H_N$ be the Hamiltonian of a slightly generalized mixed $p$-spin model, where we allow the mixture function $\xi$ to have a small linear term $\gamma_1 q$ (in our proofs we allow $\gamma_1^2 \le N^{-4/5}$), such that the non-degree-$1$ part $\xi_{\sim1}$ of $\xi$ satisfies \eqref{eq:strict-RS-condition}. Then, with high probability,
    \[ \norm*{\E_{\mu_{H_N}} \sigma\sigma^\top}_{\op} = O \left( 1 + \gamma_1^2 N \right). \]
    Much of \Cref{sec:high-prob-cov} is dedicated to showing this.
    \item The second moments of $a$ and $b$ are $O(1)$. In fact, we will show in \Cref{lem:nu-ab-subgaussianity} that they are essentially $O(1)$-subgaussian.
\end{enumerate}

Let us start by explaining how to show subgaussianity.

\parhead{Subgaussianity of $a,b$.}
The distribution $\nu$ of $(a,b)$ is given by
\[ \nu(a,b) \propto \exp\left( \log \wh{Z}_{a,b} + \frac{N-4}{2} \log \left( 1 - \frac{\|v_{a,b}\|^2}{N} \right) + \HTAP( v_{a,b} ) \right). \]
Here, the first term $\log \wh{Z}_{a,b}$ is the free energy of the $(N-2)$-dimensional $p$-spin model $\mu_{a,b}$, obtained by restricting $\mu_{\HTAP}$ to the slice $\ScS(a,b)$ and rescaling the distribution to lie on $\ScS_{N-2}$. The second term is an effective decrement in the free energy caused by the radius of the sphere $S(a,b)$ shrinking for larger values of $a$ and $b$. The third term is an effective increment in the free energy coming from the energy of $\HTAP$ at the center of the slice $\ScS(a,b)$.

For a fixed $(a,b)$, the only random quantities in the definition of $\nu$ are the first and third terms.
In \Cref{thm:partition-fn-and-covariance}, proved in \Cref{sec:high-prob-cov}, we show that the first term may essentially be approximated by a deterministic function of the mixture function of $\mu_{a,b}$, at the cost of incurring a small $O(1)$ error. We do not elaborate on the details of this proof in the technical overview; it is similar to that used to bound the covariance (which we explain shortly).
The third term is similar, and is a deterministic function plus a small Gaussian, whose variance is $O(a^2+b^2)$.

Given these bounds, we may show that the distribution $\nu$ is strongly log-concave at $0$  with high probability over the randomness of $\HTAP$. A simple perturbation argument then implies that $\nu$ is strongly log-concave in a macroscopic neighborhood of $0$, implying subgaussianity.

\parhead{Covariance bound for strictly RS models with small external fields.}
The covariance bound has now boiled down to bounding $\|M\|_{\op}$, for $M = \E_{\mu_{H_N}} [\sigma \sigma^\top]$ the second moment matrix of model satisfying \Cref{eq:strict-RS-condition} with small external field. 
Note that $M$ is a $H_N$-measurable random variable.

The proof proceeds in two high level steps, which we carry out in \Cref{sec:high-prob-cov}.
\begin{enumerate}
    \item We show using the second moment method that with positive probability over $H_N$, $\norm{M}_{\op}$ is bounded.
    \item Using a much simpler argument, we can show that $\norm{M}_{\op}$ is essentially $O(N^{-1/10})$-Lipschitz in the disorder. Hence, by gaussian concentration, it concentrates very well around its expectation (which is $O(1)$ by the positive probability bound).
\end{enumerate}

Let us elaborate a bit more on the proof of the first point above.
It turns out that, under the condition \eqref{eq:strict-RS-condition} with small external field, the leading order contribution to $M$ comes from the degree-$2$ part of the Hamiltonian $H_{N, 2}(\sigma) = \frac{\gamma_2}{N^{1/2}} \sum_{i, j} \bg_{i,j} \sigma_i \sigma_j$.
We will ultimately reduce the study of the covariance matrix of $\mu_{H_N}$ to that of $\mu_{H_{N,2}}$, and then show boundedness of $\Cov(\mu_{H_{N,2}})$ using random matrix theory.
A similar strategy of isolating the degree-$2$ component of $H_N$ was used to study the partition function and magnetization of strictly RS models in \cite{HMP24}.

\parhead{Degree-$2$ behavior.}
Let us discuss the typical behavior of the covariance of $\mu_{H_{N,2}}$. 
Define the degree-2 Gibbs measure 
\[
    \dif \mu_{H_{N,2}}(\sigma) \propto \exp(H_{N, 2}(\sigma)) \dif \rho(\sigma),
\]
with corresponding partition function $Z_{N,2} = \int \exp(H_{N, 2}(\sigma)) \dif \rho(\sigma)$. 
This is the spherical Sherrington-Kirkpatrick model with interaction matrix $A = \grad^2 H_N(0)$; note that $A$ is a scaled GOE matrix.
Observe that if we shift $A$ by a constant multiple of the identity $\gamma \Id_N$, the measure does not change, as it is supported on $\ScS_N$. 
The crucial observation is the following:
\begin{displayquote}
    For a careful choice of $\gamma$, the measure $\dif \mu_{H_{N,2}}(\sigma) \propto \exp(-\frac{1}{2} \angles{\sigma, (\gamma \Id_N - A) \sigma})\dif \rho(\sigma)$ looks like a Gaussian with covariance $(\gamma \Id_N - A)^{-1}$.
\end{displayquote}
In fact, we will see that it suffices to pick $\gamma = 1 + \xi''(0)$. 
The typical value of $\norm{x}_2^2$, where $x \sim \calN(0, \gamma \Id_N - A)^{-1}$, is equal to $\Tr(\gamma \Id_N - A)^{-1}$. 
By approximating this trace using the semicircle law for the eigenvalues of $A$ and the explicit choice of $\gamma$, we see that $\norm{x}_2^2 \approx N$, which justifies the heuristic that this Gaussian approximates the spherical distribution $\mu_{H_{N,2}}$.

For the above discussion to be well-defined, we require that $\gamma \Id_N - A$ is positive definite, which can only occur if the maximum eigenvalue of $A$ is bounded above by $\gamma = 1 + \xi''(0)$. 
By standard concentration inequalities about the maximum eigenvalue of a GOE matrix, this holds with a constant margin with exponentially good probability. 
Thus, at least for typical realizations of $H_{N,2}$, the covariance will have bounded operator norm. 
To make this rigorous, we will use the Laplace transform to precisely control the moments of the overlaps, as was previously done in \cite{BL16,HMP24}.

\parhead{Reduction to degree-$2$.}
Below, we give some justification for why one should expect to be able to reduce to the degree-$2$ behavior. We will heuristically argue this by showing that the partition function $Z_N$ is essentially controlled by the degree-$2$ portion.

To simplify the discussion, let us assume that we are in a $2+p$ spin model, so that $H_{N}(\sigma) = H_{N,2}(\sigma) + H_{N, p}(\sigma)$, where $H_{N, p}(\sigma) = \frac{\gamma_p}{N^{(p-1)/2}} \sum_{i_1,\dots,i_p = 1}^N \bg_{i_1,\dots,i_p} \sigma_{i_1}\cdots\sigma_{i_p}$. 
The corresponding mixture function decomposes as $\xi(q) = \gamma_2^2 q^2 + \xi_{\sim2}(q)$, so that $\xi_{\sim2}(q) = \gamma_p^2 q^p$ corresponds to the non degree-$2$ part of the mixture function. 
It turns out that, once we condition on $H_{N,2}$ (and hence the value of $Z_{N,2}$), the full partition function $Z_{N}$ is essentially deterministic. 
Indeed, we will show in \cref{ppn:nondeg2} that with very high probability, 
\begin{align*}
    Z_N \approx Z_{N,2} e^{N\xi_{\sim2}(1)/2}
\end{align*}
To see why this is reasonable, let us consider the first two moments of $Z_N$ conditioned on the degree-2 Hamiltonian $H_{N,2}$.
Indeed, let $\E_{\sim2}$ denote expectation with respect to $H_{N, p}$ conditioned on $H_{N,2}$.
Standard gaussian MGF calculations yield $\E_{\sim2} Z_N = e^{N\xi_{\sim2}(1)/2} Z_{N,2}$ and 
\begin{align*}
     \E_{\sim2} [Z_N^2] &= Z_{N,2}^2 e^{N\xi_{\sim2}(1)} \int \exp(N\xi_{\sim2}(R(\sigma^1, \sigma^2)))\dif \rho^{\otimes 2}(\sigma^1, \sigma^2) \\
     &= (\E_{\sim2} Z_N)^2 \int \exp(N\gamma_p^2 R(\sigma^1, \sigma^2)^p) \dif \rho^{\otimes 2}(\sigma^1, \sigma^2).
\end{align*}
At sufficiently high temperatures, the typical overlap behavior $R(\sigma^1, \sigma^2) \asymp N^{-1/2}$, where $\sigma^1, \sigma^2$ are iid draws from $\mu_{H_N}$. 
This matches the overlap behavior at infinite temperature, where the Gibbs distribution is uniform.
Then, pretending that $R(\sigma^1, \sigma^2) = cN^{-1/2}$ for all $\sigma^1, \sigma^2$, we obtain that 
\begin{align*}
    \E_{\sim2}[Z_N^2] \approx (\E_{\sim2} Z_N)^2 \exp(c^p\gamma_p^2 N^{1-p/2}).
\end{align*}
Since $p \ge 3$, it follows that, conditional on $H_{N,2}$, the conditional variance of $Z_N$ is tiny compared to its conditional expectation.
In summary, we see that the higher degree portions of the partition function have negligible contributions to the fluctuations of $Z_N$, so that the typical behavior of $Z_N$ is controlled by $Z_{N,2}$. 

Turning now to the covariance bound, we will control the $(i,j)$th covariance entry $M_{i,j} \defeq \int \sigma_i \sigma_j e^{H_N(\sigma)} \dif \rho(\sigma)$. 
A crucial fact is that, by rotational invariance of the sphere and gaussians, we can rotate to the eigenbasis of $A = \grad^2 H_N(0)$ so that $A$ becomes diagonal.
When $A$ is diagonal, one can in fact show that
\begin{align*}
    \E_{\sim2} [M_{i,j}^2] \lesssim \frac{1}{N^2} (\E_{\sim2} M_{i,j})^2,
\end{align*}
where $\E_{\sim2} M_{i,j}$ can be interpreted as (up to normalization) the predicted $(i,j)$th covariance entry by just looking at the degree-2 randomness; see \Cref{ppn:hX-ii,ppn:hX-ij} for details.
It follows that the Frobenius norm error of the true covariance compared to the degree-2 covariance is $O(1)$. 
Combined with the typical behavior of the degree-2 covariance being essentially the diagonal matrix $((1+\xi''(0)) \Id_N - A)^{-1}$, we conclude an $O(1)$ covariance bound for $\mu_{H_N}$.

Although this direct moment approach can be made rigorous at sufficiently high temperature, it will not cover the entire regime \eqref{eq:strict-RS-condition} of our main theorem. 
To deal with this, we will use the \emph{free energy typical truncation} recently introduced by \cite{HS23}.
The main idea is that, while pairs $\sigma^1,\sigma^2$ with overlap $R(\sigma^1, \sigma^2) \asymp N^{-1/2}$ do not necessarily dominate the second moment $\E [Z_N^2]$ throughout the regime \eqref{eq:strict-RS-condition}, there is a truncation $\wt{Z}_N$ accounting for nearly all of $Z_N$, whose second moment is dominated by such pairs. 
We defer the details to the following sections.

\subsection{Null models, planted models, and contiguity}
\label{sec:planting}
As described in the technical overview, we will need a quantitative strengthening of contiguity between the null and planted models. 
For convenience, let us restate the definition of the planted model.
\planted*

\begin{remark}[Interpretation of planted model]
    \label{rmk:planted-interpretation}
    Equivalently, the planted measure $\mu_{\pl}$ can be described as follows.
    \begin{itemize}
        \item Sample $\bx\sim\ScS_N$.
        \item Sample $\wt{H}_N \sim \HamDist_{\nullmodel}$.
        \item Define $H_N$ by $H_N(\sigma) = \wt{H}_N(\sigma) + N \cdot \xi(R(\bx,\sigma))$.
    \end{itemize}
    The following Bayesian interpretation of $\mu_{\pl}$ will make the planted model amenable to explicit calculation.
    For $(\bx,H_N)$ sampled from $\mu_{\pl}$, the posterior distribution $\bx|H_N$ is described by the density:
    \[
        \dif\HamDist_{\bx|H_N}(\sigma) \propto \exp(H_N(\sigma)) \dif\Unif(\sigma)\mper
    \]
    Therefore, the distribution of $(H_N,\sigma)$ for $\sigma \sim \mu_{H_N}$ is identical to that of $(H_N,\bx)$.
\end{remark}

In order to show the probability bound of $1-e^{-cN^{1/5}}$ in Lemma~\ref{lem:pspin-cov-bound-sl-path}, we will prove the following quantitative strengthening of mutual contiguity, under the following quantitative strict RS condition.
Note that, since the proof of Theorem~\ref{thm:main-p-spin} union bounds over $\poly(N)$ many values of $\beta$, quantitative control of the error in Lemma~\ref{lem:pspin-cov-bound-sl-path} is needed to carry out the proof.
\begin{con}[$\eps$-strict replica symmetry]
    \label{con:strict-rs}
    We say $\xi$ is $\eps$-strictly replica symmetric if for all $q\in (0,1)$,
    \begin{equation}
        \label{eq:strict-rs}
        \frac{1}{q^2} \cdot (\xi(q) + q + \log(1-q)) \le -\eps/2.
    \end{equation}
\end{con}
Under this assumption, we prove the following quantitative contiguity result in \Cref{sec:high-prob-cov}.
\begin{restatable}[Quantitative contiguity]{proposition}{quantcon}  \label{prop:quantiguity}
    Under \Cref{con:strict-rs}, 
    there exists $c = c(\eps) > 0$ such that for any event $\calE$, if $\HamDist_{\pl}(\calE) = p$, then $\HamDist_{\nullmodel}(\calE) \le e^{-cN^{1/5}} + e^{\frac{1}{c}\sqrt{\log \frac{2}{p}}} p$.
\end{restatable}
Thus, from now on, we work under the planted model. 
One reason the planted model is easier to work with is because of the following lemma, which provides a simple description of the distribution of $\mu_t$ (by describing the distribution of the external field at time $t$) in the planted model. 

\begin{lemma}
    Let $\mu_t$ be the distribution after running stochastic localization with the $\Id$ driving matrix for time $t$ initialized at $\mu_{H_N}$.
    Then $\mu_t$ arises as the Gibbs distribution of the Hamiltonian $H_{N,t}(\sigma)$:
    \[
        H_{N,t}(\sigma) = H_N(\sigma) + \angles*{\by_t, \sigma}\mcom
    \]
    where 
    \[
        (H_N, \by_t) \stacker{law}{=} (H_N, t\bx + \sqrt{t} \bg)\mcom
    \]
    where $\bx \sim \ScS_N$, $H_N \sim \mu_{\pl}(\cdot | \bx)$, and $\bg\sim\calN(0,\Id_N)$.
\end{lemma}

\begin{notation*}[$\HamDist_{\pl,t}$, $\xi_t(q)$, $\gamma(q)$]
    We will use $\HamDist_{\pl, t}$ to denote the distribution of the pair $\parens*{H_{N,t}, \bx}$, $\xi_t(q) = \xi(q) + tq$ to refer to the mixture function of $H_{N,t}$, and $\gamma(q)$ to refer to the function $q\xi_t'(q)$.
\end{notation*}

In the subsequent sections, we will prove a high probability covariance bound for $\mu_t$ at a \emph{fixed} time $t$ under the planted model.
\begin{restatable}{lemma}{lemcovbound}
    \label{lem:planted-pspin-cov-bound}
    There exist universal constants $c, T, K$, such that for any $t\in[0,T]$, with probability at least $1-e^{-cN^{1/5}}$ over the randomness of $H_N$ drawn from $\HamDist_{\pl, t}$, we have $\norm*{\Cov(\mu_t)} \le K$.
\end{restatable}

We now have all the necessary ingredients to prove the covariance bound along the entire localization path for the null model.

\begin{proof}[Proof of \Cref{lem:pspin-cov-bound-sl-path}]
    Define $\calT$ as the discrete set $\{iT/\delta:1\le i \le 1/\delta, i\in\bbZ\}$ for $\delta = N^{-100}$.
    We will prove:
    \[
        \Pr_{H_N} \bracks*{ \Pr_{(\mu_t)|H_N} \bracks*{ \norm*{\Cov(\mu_t)}_{\op} < K\text{ for all }t\in\calT } \ge 1-e^{-cN^{1/5}} } \ge 1 - e^{-cN^{1/5}}\mper
    \]
    A simple continuity argument can be used to derive the desired statement from the above.
    By taking a union bound over all elements of $\calT$, along with \Cref{prop:quantiguity,lem:planted-pspin-cov-bound}, we can conclude:
    \[
        \E_{H_N} \Pr_{(\mu_t)|H_N} \bracks*{ \norm*{\Cov\parens*{\mu_t}}_{\op} > K\text{ for some }t\in\calT } \le e^{-2cN^{-1/5}}.
    \]
    The resulting statement then follows from Markov's inequality on the random variable
    \[
        \Pr_{(\mu_t)|H_N} \bracks*{ \norm*{\Cov\parens*{\mu_t}}_{\op} > K\text{ for some }t\in\calT }\mper  \qedhere
    \]
\end{proof}

\subsection{TAP planted models}
\label{sec:tap-planted}
In this section, we formally introduce the TAP planted model and relate it to the planted model from the previous section. 
\begin{definition}  \label{def:TAP-free-energy}
    Let $H_N$ be a planted Hamiltonian with mixture function $\xi$, and define
    \[ \theta(s) = \xi(1) - \xi(s) - (1-s)\xi'(s). \]
    The associated \emph{TAP free energy} is defined by
    \[ \FTAP(\boldm) = H_N(\boldm) + \frac{N}{2} \cdot \theta\left(R(\boldm,\boldm)\right) + \frac{N}{2} \cdot \log \left( 1 - R(\boldm,\boldm) \right). \]
\end{definition}

While the TAP free energy is interesting for a multitude of reasons, we will be interested in it because its fixed points provide a good proxy for the mean. Furthermore, the linearity of the TAP free energy in the Gaussian coefficients of the Hamiltonian provides certain desirable properties (that using the true mean would not allow).

\begin{fact}[{\cite[Fact 4.2]{HMP24}}]
    Let $\xi$ be a mixture function satisfying the condition \eqref{eq:SL-condition}.
    For any $t \in [0,\infty)$, let $\xi_t(q) = \xi(q) + tq$. 
    Then there is a unique solution in $[0,1)$, which we denote $q_* = q_*(t)$, to
    \[
        \xi'_t(q_*) = \frac{q_*}{1-q_*}.
    \]
\end{fact}

\begin{lemma}   \label{lem:planted-to-TAP}
    For any $K > 0$, sufficiently small (constant) $\iota > 0$ and $\bx\in\ScS_N$:
    \begin{align*}
        \Pr_{\HamDist_{\pl, t}}&\left[ \opnorm{\Cov(\mu_{H_{N, t}})} \ge K \right]\\
        &\le C\cdot \sup_{\boldm \in \calS_{\iota}} \Pr_{\HamDist_{\pl,t}|\bx}\left[ \opnorm{\Cov(\mu_{H_{N,t}})} \ge K \land \calE_{\iota} \mid \grad\FTAP(\boldm) = 0 \right]^{1/2} + 2e^{-cN},
    \end{align*}
    where
    \begin{align*}
        \calS_{\iota} = \calS_{\iota}(\bx) &\coloneqq \left\{ \boldm \in \R^N : |R(\boldm,\boldm) - q_*| , |R(\boldm,\bx) - q_*| < \iota \right\},
    \end{align*}
    and $\calE_{\iota}$ is the event that $\FTAP$ has a unique critical point $\mTAP$ in $\calS_{\iota}$, and that
    \[
        \Pr_{\bsigma\sim\mu_{H_{N,t}}}\bracks*{ R(\bsigma, \mTAP), R(\bsigma, \bx) \in [q_*-\iota, q_*+\iota] } \ge 1-e^{-cN}\mper
    \]
\end{lemma}
\begin{proof}
    The above statement is effectively due to \cite[Propositions 4.4(d) and 4.5(a)]{HMP24}.
    For the reader's convenience, we include the steps to arriving at the above statement.
    For any event $\calE$ (and in particular, for the event $\calE$ defined in \cite[Proposition 4.4]{HMP24}), we have:
    \begin{align*}
        \Pr\left[ \opnorm{\Cov(\mu_{H_{N, t}})} \ge K \right] \le \Pr\left[ \opnorm{\Cov(\mu_{H_{N, t}})} \ge K \land \calE_{\iota} \land \calE \right] + \Pr\bracks*{\ol{\calE}} + \Pr\bracks*{\ol{\calE_{\iota}}} \mper
    \end{align*}
    The desired statement follows by observing that $\Pr\bracks*{\ol{\calE_{\iota}}} \le e^{-cN}$ by \cite[Proposition 4.5(a)]{HMP24}, $\Pr\bracks*{\ol{\calE}} \le e^{-cN}$ by \cite[Proposition 4.4]{HMP24}, and applying \cite[Proposition 4.4(d)]{HMP24} with $X = \Ind\bracks*{\opnorm{\Cov(\mu_{H_{N, t}})} \ge K \land \calE_{\iota}}$.
\end{proof}

\Cref{lem:planted-to-TAP} reduces our task to studying the covariance matrix in a \emph{conditional} planted model.

\begin{notation*}[$\muTAP$, $\HTAP$, $q_{\boldm}$, $q_{\bx}$]
    For $\bx\sim\ScS_N$ and $\boldm\in\R^N$, we consider the distribution $\HamDist_{\TAP,\bx,\boldm}$ of $\HTAP$ for $\HTAP\sim(\HamDist_{\pl,t}|\bx,\grad\FTAP(\boldm) = 0)$.
    We use $q_{\boldm}$ and $q_{\bx}$ to refer to $R(\boldm,\boldm)$ and $R(\boldm,\bx)$ respectively.    
\end{notation*}

\begin{restatable}{lemma}{plantedcovbound}   \label{lem:tap-planted-cov-bound}
    Let $\bx\in\ScS_N$, let $\calS_{\iota}$ be as in \Cref{lem:planted-to-TAP}, and let $\boldm\in\calS_{\iota}$.
    Then for an absolute constant $K > 0$,
    \[
        \Pr_{\HTAP\sim\HamDist_{\TAP,\bx,\boldm}}\bracks*{\norm*{\Cov\parens*{\mu_{\HTAP}}} \ge K \land \calE_{\iota}} \le e^{-c N^{1/5}}\mper
    \]
\end{restatable}

\begin{proof}[Proof of \Cref{lem:planted-pspin-cov-bound}]
    The statement is immediate from \Cref{lem:planted-to-TAP,lem:tap-planted-cov-bound}.
\end{proof}

The rest of this section is dedicated to proving \Cref{lem:tap-planted-cov-bound}.
As a first step, we determine the law of the typical Hamiltonian sampled from $\muTAP$. We prove the following lemma in \Cref{app:spinglass-calcs} --- it follows by routine calculations, using the form of the law of a Gaussian process conditioned on the value of a linear function of it.
Recall that $\xi_t(q) = \xi(q) + tq$.

\begin{restatable}{lemma}{lemTAPhamiltonian}
    \label{lem:TAP-hamiltonian}
    The law of Hamiltonian $\HTAP\sim\HamDist_{\TAP,\bx,\boldm}$ is described by a Gaussian process $\parens*{\HTAP(\sigma)}_{\sigma\in\ScS_N}$ defined by
    \begin{multline*}
        \E\, \HTAP(\sigma) = N\xi_t\parens*{R(x,\sigma)} -\angles*{\bx, v(\sigma)}\cdot\xi_t'(q_{\bx}) - \frac{\xi_t'(R(\boldm,\sigma))}{\gamma'(q_{\boldm})} \cdot \angles*{\boldm,\sigma}\cdot\parens*{\theta'\parens*{q_{\boldm}} - \frac{1}{1-q_{\boldm}}} \\
        \frac{1}{N}\Cov\parens*{\HTAP\parens*{\sigma}, \HTAP\parens*{\sigma'}}
        = \xi_t\parens*{R(\sigma,\sigma')} -
        R(\sigma,\sigma')\frac{\xi_t'(R(\boldm,\sigma))\xi_t'(R(\boldm,\sigma'))}{\xi_t'(q_{\boldm})} \\
        + \frac{\xi_t''(q_{\boldm})}{\gamma'(q_{\boldm})\xi_t'(q_{\boldm})}\gamma(R(\boldm,\sigma))\gamma(R(\boldm,\sigma')),
    \end{multline*}
    where
    \begin{align*}
        v(\sigma) &\coloneqq \frac{\xi_t'(R(\boldm,\sigma))}{\xi_t'(q_{\boldm})}\bracks*{I - \frac{\xi_t''(q_{\boldm})}{\gamma'(q_{\boldm}) }\cdot \frac{\boldm\boldm^{\top}}{N}  }\sigma \\
        \gamma(q) &\coloneqq q\cdot\xi_t'(q)\mper 
    \end{align*}
\end{restatable}

For the proofs below, it will also be helpful to consider Hamiltonians with a linear term representing an external field. 
For a sequence $\gamma_1,\gamma_2,\ldots,\gamma_{p_*}$, consider the following generalization of $H_N$ from \eqref{eq:p-spin-ham}:
\begin{equation}
    \label{eq:pspin-with-external-field}
    H_N(\sigma) \coloneqq \sum_{p\ge 1} \frac{\gamma_p}{N^{(p-1)/2}} \sum_{i_1,\dots,i_p = 1}^N \bg_{i_1,\dots,i_p} \sigma_{i_1}\cdots\sigma_{i_p}.
\end{equation}
This has mixture function
\[
    \xi(s) = \sum_{p\ge 1} \gamma_p^2 q^p.
\]
We will write $\xi_{\sim1}(s) = \sum_{p\ge 2} \gamma_p^2 q^p$ for the part of $\xi$ with degree at least $2$, and extend \Cref{con:strict-rs} to such $\xi$ as follows.
\begin{restatable}[$\eps$-strict replica symmetry]{con}{strictrs}
    \label{con:strict-rs-2}
    We say $\xi$ is $\eps$-strictly replica symmetric if $\gamma_1^2 \le N^{-4/5}$ and $\xi_{\sim1}$ satisfies Condition~\ref{con:strict-rs}.
\end{restatable}

\subsection{Slices in TAP planted models}\label{sec:tap-slices}
For succinctness, we shall fix $\bx \in S_N$ and $\boldm \in \calS_\iota(\bx)$, and use $\muTAP$ to refer to the distribution $\HamDist_{\TAP,\bx,\boldm}$.
For $\HTAP\sim\muTAP$, we are interested in bounding the covariance of $\mu_{\HTAP}$.
To reason about $\mu_{\HTAP}$, we write it as a mixture of distributions over $(N-2)$-dimensional slices of $\ScS_{N}$.
For $a, b\in\R$, we define
\[
    \ScS(a,b) \coloneqq \left\{ \sigma \in S_N : R(\sigma, \boldm) = \parens*{1 + \frac{a}{\sqrt{N}}} q_{\boldm} ,\, R(\sigma, \bx) = \parens*{ 1 + \frac{b}{\sqrt{N}} }q_{\bx} \right\}\mper
\]
Let $r_{a,b}$ refer to the radius of this slice, which is equal to
\[
    r_{a,b} = \sqrt{1 - q_{\boldm} \left( 1 + \frac{a}{\sqrt{N}} \right)^2 - \frac{q_{\boldm} q_{\bx}^2}{q_{\boldm} - q_{\bx}^2} \left(\frac{a-b}{\sqrt{N}} \right)^2  }. 
\]
Note in particular that
\begin{equation}
    \label{eq:rab-bound-sos}
    q_{\boldm} \left( 1 + \frac{a}{\sqrt{N}} \right)^2 + \frac{q_{\boldm} q_{\bx}^2}{q_{\boldm} - q_{\bx}^2} \left(\frac{a-b}{\sqrt{N}}\right)^2
        \ge q_{\boldm} \left( 1 + \frac{a}{\sqrt{N}} \right)^2.
\end{equation}

We refer to the uniform distribution on this slice as $\Unif_{a,b}$, and the partition function on the slice as
\[
    Z_{a,b} \coloneqq \E_{\sigma\sim \Unif_{a,b}}\exp(\HTAP(\sigma))\mper
\]
With this definition, the partition function of the original Hamiltonian is given by
\[ Z = \Lambda_N \int Z_{a,b} r_{a,b}^{N-4} \dif (a,b) \]
for some fixed number $\Lambda_N$ depending only on $N$.
\begin{remark}
    To see why we scale by $r_{a,b}^{N-4}$, observe that when $H_{\TAP}$ is the constant-$0$ Hamiltonian, the resulting distribution on the sphere should be uniform.
    The distribution restricted to each slice must also be uniform.
    However, not all slices are weighted equally --- slice that have smaller radii must be downweighted accordingly, with this weighting proportional to $r_{a,b}^{N-4}$ for $S(a,b)$.\footnote{The constant of proportionality here is something depending only on $N$.}
\end{remark}
Use $\nu$ to refer to the distribution over $(a,b)$ where $\dif\nu(a,b) \propto Z_{a,b} r_{a,b}^{N-4} \dif (a,b)$, and $\mu_{a,b}$ to refer to the distribution $\mu_{\HTAP}$ restricted to $S({a,b})$.
Now, we can write $\mu_{\HTAP}$ as the following mixture:
\begin{align*}
    \mu_{\HTAP} = \E_{(a,b)\sim\nu} \mu_{a,b}\mper
\end{align*}

We will need coarse understanding of the tails of $\nu$, and fine understanding of the distribution $\mu_{a,b}$ for small $a$ and $b$.

Now, let us probe the distribution $\mu_{a,b}$.
\[
    \frac{\dif\mu_{a,b}}{\dif \Unif_{a,b}}(\sigma) = \frac{\exp(\HTAP(\sigma))}{\E_{\sigma\sim\Unif_{a,b}}\exp(\HTAP(\sigma))}
\]
Since $\ScS(a,b)$ can be naturally identified with $\ScS_{N-2}$, the first step to understanding $\mu_{a,b}$ is to express it as a $p$-spin model on $\ScS_{N-2}$.
To do so, we will verify that \emph{some} Hamiltonian that gives rise to $\mu_{a,b}$ has a mixture function that is given by a polynomial in the overlap.
We can write any $\sigma\in\ScS(a,b)$ as:
\[
    \sigma = \sqrt{N} \cdot v(a,b) + \underbrace{\sqrt{1-\norm*{v(a,b)}^2}}_{r_{a,b}} \sigma_{\perp}
\]
for $\sigma_{\perp} \in S_N$ orthogonal to $\boldm$ and $\bx$, and for $v(a,b)$ in the span of $\boldm$ and $\bx$.
Let $Q$ be an isometric linear transformation that maps $\ScS_{N-2}$ to $\ScS_N\perp\{\boldm,\bx\}$.
We can write $\HTAP(\sigma) = \HTAP(v(a,b) + r_{a,b} Q\tau)$ (where $\tau \in S_{N-2}$).
The following is a consequence of \Cref{lem:TAP-hamiltonian}, and is proved in \Cref{app:spinglass-calcs}. 
\begin{restatable}{corollary}{corslicedist}   \label{cor:slice-dist}
    For a fixed choice of $a$ and $b$, the Gaussian process $\parens*{\HTAP(v(a,b) + r_{a,b}Q\tau)}_{\tau\in\ScS_{N-2}}$ is described by the following law.
    \begin{itemize}
        \item Let $H_{a,b}$ be a spherical $p$-spin Hamiltonian with mixture function $\xi_{a,b}$ given by:
        \[
            \xi_{a,b}(s) \coloneqq \xi_t\parens*{\norm*{v(a,b)}^2 + r_{a,b}^2 s  } - \xi_t\parens*{\norm*{v(a,b)}^2}
            - s \cdot \frac{r_{a,b}^2\xi'_t\parens*{q_{\boldm}\cdot\parens*{1+\frac{a}{\sqrt{N}}}}^2}{\xi_t'(q_{\boldm})} \mper
        \]
        \item Let $V(a,b) \coloneqq \xi_t\parens*{\norm*{v(a,b)}^2} - \norm*{v(a,b)}^2 \cdot \frac{\xi_t'\parens*{\parens*{1+\frac{a}{\sqrt{N}}}q_{\boldm}}^2}{\xi_t'(q_{\boldm})} + \frac{\xi_t''(q_{\boldm})}{\gamma'(q_{\boldm})\xi_t'(q_{\boldm})} \cdot \gamma \left( \left( 1 + \frac{a}{\sqrt{N}}\right) q_{\boldm} \right)^2$.
    \end{itemize}
    The law of $\HTAP\parens*{v(a,b)+r_{a,b}Q\tau}$ is the same as that of $H_{a,b}(\tau) + \sqrt{N} \cdot g_{a,b} + \E_{\muTAP} \HTAP\parens*{v(a,b)+r_{a,b}Q\tau} $ where $g_{a,b}$ is a centered Gaussian of variance $V(a,b)$ independent of $H_{a,b}$.
\end{restatable}

Now, $\HTAP$ is described by the collection $\parens*{H_{a,b}, g_{a,b}}_{a,b}$.
This is \emph{not} an independent collection of random variables.
The only structural properties of this collection we will use are:
\begin{itemize}
    \item For each $a,b\in\R$, we have $H_{a,b}$ and $g_{a,b}$ are independent.
    \item For any $a,b$, we have $g_{a,b} = g_{0,0} + \wh{g}_{a,b}$, where $\wh{g}_{a,b}$ is a centered Gaussian of variance $O\left(\frac{a^4+b^4}{N^2}\right)$.
\end{itemize}

We will first give an explicit form for $v(a,b)$.
\begin{fact}    \label{fact:vab-formula}
    We have 
    \begin{align*}
        \sqrt{N} \cdot v(a,b) = \boldm\cdot\parens*{1 + \frac{aq_{\boldm}-bq_{\bx}^2}{\sqrt{N}(q_{\boldm}-q_{\bx}^2)} } +
        \frac{q_{\boldm} q_{\bx} }{ q_{\boldm} - q_{\bx}^2 } \bx\cdot\parens*{ \frac{b-a}{\sqrt{N}} }\mper
    \end{align*}
\end{fact}

\begin{lemma}
    \label{lem:band-models-eps-strictly-rs}
    There exists $\eps = \eps(\xi) > 0$ such that for all $t\ge 0$ and all $|a|,|b| \le \eps N^{1/10}$, the mixture function $\xi_{a,b}$ defined in \Cref{cor:slice-dist} (recall this implicitly depends on $t$) is $\eps$-strictly replica symmetric (\Cref{con:strict-rs-2}).
\end{lemma}
\begin{proof}
    We will first show $\xi'_{a,b}(0) \le N^{-4/5}$.
    We calculate:
    \begin{align*}
        \xi_{a,b}'(0) &= r_{a,b}^2 \xi_t' \left( \|v(a,b)\|^2 \right) - \frac{r_{a,b}^2 \xi_t'\left( q_{\boldm} \left( 1 + \frac{a}{\sqrt{N}} \right) \right)^2}{\xi_t'(q_{\boldm})} \\
            &\le \xi_t' \left( q_{\boldm} + \frac{2aq_{\boldm}}{\sqrt{N}} + O\left( \frac{a^2+b^2}{N} \right) \right) - \frac{\xi_t'\left( q_{\boldm} \left( 1 + \frac{a}{\sqrt{N}} \right) \right)^2}{\xi_t'(q_{\boldm})}.
    \end{align*}
    Here, the inequality follows from the fact that $\|v(a,b)\|^2 = q_{\boldm}\left( 1 + \frac{a}{\sqrt{N}} \right)^2 + O\left( \frac{a^2+b^2}{N} \right)$, $\xi_t'$ is non-decreasing, and $r_{a,b}^2 \le 1$. Now, we have, using the $O(1)$-Lipschitzness of $\xi_t'$,
    \[ \xi_t'\left( \|v(a,b)\|^2 \right) = \xi_t'(q_{\boldm}) + \xi_t''(q_{\boldm}) \cdot \frac{2aq_{\boldm}}{\sqrt{N}} + O\left( \frac{a^2+b^2}{N} \right) \]
    and
    \[ \xi_t'\left( q_{\boldm} \left( 1 + \frac{a}{\sqrt{N}} \right) \right)^2 = \xi_t'(q_{\boldm})^2 + \frac{2aq_{\boldm}}{\sqrt{N}} \cdot \xi_t''(q_{\boldm}) \cdot \xi_t'(q_{\boldm}) + O\left( \xi_t'(q_{\boldm}) \cdot \frac{a^2 + b^2}{N} \right). \]
    Thus $\xi'_{a,b}(0) = O(\fr{a^2+b^2}{N})$.
    Setting $\eps$ sufficiently small ensures $\xi'_{a,b}(0) \le N^{-4/5}$.
    Next, we show $(\xi_{a,b})_{\sim1}$ satisfies \Cref{con:strict-rs}.
    Since $\xi$ satisfies \eqref{eq:SL-condition}, there exists sufficiently small $\eps = \eps(\xi)$ such that $\xi''(q) \le \fr{1-\eps}{(1-q)^2}$ for all $q\in [0,1)$.
    Then,
    \begin{align*}
        \xi_{a,b}''(q) 
        &= \parens*{1-\norm*{v(a,b)}^2}^2 \xi_t''\parens*{\norm*{v(a,b)}^2 +  \parens*{1-\norm*{v(a,b)}^2}q } \\
        &\le \parens*{1-\norm*{v(a,b)}^2}^2 \cdot \frac{1-\eps}{\parens*{1-\norm*{v(a,b)}^2 - \parens*{1-\norm*{v(a,b)}^2}q}^2} \\
        &= \fr{1-\eps}{(1-q)^2}
        \le \fr{1}{(1-q)^2} - \eps.
    \end{align*}
    Integrating twice 
    shows
    \[
        (\xi_{a,b})_{\sim1}(q) + q + \log(1-q)
        \le \fr12 \eps q^2. \qedhere
    \]
\end{proof}

We are now ready to bound $\Cov\parens*{ \mu_{\HTAP} }$.
First, recall that for any distribution $\mu$ over $\R^N$ and any vector $v\in\R^N$, we have $\Cov(\mu) \psdle \E_{\sigma\sim\mu} (\sigma - v)(\sigma-v)^{\top}$.
Thus it suffices to bound
\begin{align*}
    {\E}_{\sigma\sim\mu_{\HTAP}} &\parens*{ \sigma - \boldm } \parens*{ \sigma - \boldm }^{\top}\\
    &= \E_{(a,b)\sim\nu}{\E}_{\sigma\sim\mu_{a,b}} \parens*{ \sigma - v(a,b) + v(a,b) - \boldm } \parens*{ \sigma - v(a,b) + v(a,b) - \boldm }^{\top} \\
    &\psdle 2\E_{(a,b)\sim\nu}\E_{\sigma\sim\mu_{a,b}} \parens*{\sigma - v(a,b)}\parens*{\sigma - v(a,b)}^{\top} + 2\E_{(a,b)\sim\nu} (v(a,b)-\boldm)(v(a,b)-\boldm)^{\top}\mper    \numberthis \label{eq:cov-bound}
\end{align*}
We will bound the spectral norm of each of the above terms below.
We employ the following statement for proving the desired bounds, proved in \Cref{sec:high-prob-cov}.

\begin{restatable}{theorem}{partitionfn}
    \label{thm:partition-fn-and-covariance}
    Suppose $\eps > 0$, and $\xi$ is $\eps$-strictly replica symmetric (\Cref{con:strict-rs-2}).
    Then, there exist $c = c(\eps)$ and $C = C(\eps)$ such that the following hold with probability $1-e^{-cN^{1/5}}$.
    \begin{enumerate}[label=(\alph*)]
        \item We have $\nabla^2 H_N(0) \preceq (1 + \xi''(0) - \eps^2/8) I_N$ and
        \[
            \lt|\log Z_N - \fr{N\xi(1)}{2} - \fr{N\xi''(0)}{4} - \fr{\log(1-\xi''(0))}{2}
            + \fr12 \log \det\lt((1 + \xi''(0)) I_N - \nabla^2 H_N(0)\rt) \rt| \le 1.
        \]
        \item The Gibbs measure satisfies $\|\E_{\mu_{H_N}}\sigma\sigma^{\top}\|_{\op} \le C(1 + \gamma_1^2 N)$.
    \end{enumerate}
\end{restatable}

\parhead{Bounding the first term.}
Observe that:
\[
     \E_{\sigma\sim\mu_{a,b}} \parens*{\sigma - v(a,b)}\parens*{\sigma - v(a,b)}^{\top} = r_{a,b}^2 Q_{a,b} \E_{  \tau \sim \mu_{ H_{a,b} }  } \tau\tau^{\top} Q_{a,b}^{\top}\mper
\]
Thus,
\[
    \Norm{ \E_{\sigma\sim\mu_{a,b}} \parens*{\sigma - v(a,b)}\parens*{\sigma - v(a,b)}^{\top} } =
    r_{a,b}^2 \Norm{ \E_{  \tau \sim \mu_{ H_{a,b} }  } \tau \tau^{\top} } \le Cr_{a,b}^2 (1 + a^2 + b^2)\mper  \numberthis  \label{eq:first-term-cov-bound}
\]
where the inequality follows from \Cref{lem:band-models-eps-strictly-rs,thm:partition-fn-and-covariance}.

Our next goal is to control the fluctuations of $(a,b)$. 
By definition, if $\wh{Z}_{a,b}$ is the partition function of the distribution with Hamiltonian $H_{a,b}$ defined by $H_{a,b}(\tau) = \HTAP(v(a,b) + r_{a,b}Q\tau)$, we have
\[ \nu(a,b) \propto \exp\left( \log \wh{Z}_{a,b} + (N-4) \log r_{a,b} + \E_{\mu_{\TAP}} \HTAP(v(a,b)) + \sqrt{N} g_{a,b} \right). \numberthis \label{eq:density-of-nu} \]
We will show that $(a,b) \sim \nu$, conditioned on $|a|,|b| \le \eps N^{1/10}$ for $\eps$ as in \Cref{lem:band-models-eps-strictly-rs}, is subgaussian with variance $O(1)$.
We will also show in \Cref{lem:nu-ab-concentration} that $\nu$ places very little mass outside the set $|a|,|b| \le \eps N^{1/10}$.
\begin{lemma}
    \label{lem:formula-for-nu}
    Let $\eps$ be as in Lemma~\ref{lem:band-models-eps-strictly-rs}. 
    On an event with probability $1-e^{-cN^{1/5}}$, the following holds.
    The density of $(a,b)$ under $\nu$, conditioned on $|a|,|b| \le \eps N^{1/10}$, is given by
    \[ \nu(a,b) \propto \exp\left( N \wh{E}_{a,b} + \sqrt{N} g_{a,b} + \Error^{(1)}_{a,b} + \Error^{(2)}_{a,b} + \Delta_{a,b}\right), \]
    where $|\Delta_{a,b}| \le 1$ and
    \begin{align*}
        \wh{E}_{a,b} &= \frac{\xi_{a,b}(1)}{2} + \log r_{a,b} + \frac{1}{N} \E_{\mu_\TAP} \HTAP(v(a,b)) - \frac{\xi_t(1)}{2} \\
            &= \frac{1}{2} \left( \log r_{a,b}^2 - \xi_t(\|v(a,b)\|^2) - r_{a,b}^2 \cdot \frac{\xi_t'\left( q_{\boldm} \left( 1 + \frac{a}{\sqrt{N}} \right) \right)^2}{\xi_t'(q_{\boldm})} \right) \\
            &\qquad - \gamma(q_{\bx}) \cdot \frac{\xi_t'\left( q_{\boldm} \left( 1 + \frac{a}{\sqrt{N}} \right) \right)}{\xi_t'(q_{\boldm})} \cdot \left( \left( 1 + \frac{b}{\sqrt{N}} \right) -  q_{\boldm} \cdot \frac{\xi_t''(q_{\boldm})}{\gamma'(q_{\boldm})} \cdot \left( 1 + \frac{a}{\sqrt{N}} \right) \right) \\
            &\qquad + \xi_t\left( q_{\bx} \left( 1 + \frac{b}{\sqrt{N}} \right) \right) + \frac{\gamma\left( q_{\boldm} \left( 1 + \frac{a}{\sqrt{N}} \right) \right)}{\gamma'(q_{\boldm})} \cdot \left( (1-q_{\boldm}) \xi_t''(q_{\boldm}) + \frac{1}{1-q_{\boldm}} \right).
    \end{align*}
    for the error terms
    \begin{align*}
        \Error^{(1)}_{a,b} &= \left( \log \wh{Z}_{a,b} - \frac{N\xi_{a,b}(1)}{2} - \frac{N \xi''_{a,b}(0)}{4} + \frac{1}{2} \log \det \left( (1+\xi''_{a,b}(0)) \Id - \grad^2 H_N(v(a,b)) \right) \right) - 4 \log r_{a,b}
    \end{align*}
    and
    \begin{align*}
        \Error^{(2)}_{a,b} &= \frac{N\xi''_{a,b}(0)}{4} - \frac{1}{2} \log \det \left( (1+\xi''_{a,b}(0)) \Id - \grad^2 H_N(v(a,b)) \right)
    \end{align*}
\end{lemma}
The above follows by expanding out all the terms in the expression \eqref{eq:density-of-nu} for the density of $\nu$ and evaluating the term $\log \wh{Z}_{a,b}$ using \Cref{thm:partition-fn-and-covariance}, which applies because the models $\xi_{a,b}$ for $|a|,|b| \le \eps N^{1/10}$ are $\eps$-strictly replica symmetric by \Cref{lem:band-models-eps-strictly-rs}.
Because the conclusion of \Cref{thm:partition-fn-and-covariance} holds with probability $1-e^{-cN^{1/5}}$, we may evaluate $\log \hat Z_{a,b}$ over a $1/\poly(N)$-net of such $(a,b)$ via a union bound, and then infer the estimate for all such $a,b$ by a standard continuity argument.
\begin{restatable}{lemma}{nuenergygradient}
    \label{lem:nu-energy-gradient}
    $\restr{\grad \wh{E}_{a,b}}{(a,b)=(0,0)} = 0$.
\end{restatable}

\begin{restatable}{lemma}{lemenergyconcavity}
    \label{lem:energy-ab-concavity}
    There exist constants $\eta,\eps > 0$ such that for all $|a|,|b| \le \eps \sqrt{N}$, $N\grad^2 \wh{E}_{a,b} \psdle -\eta \Id$.
\end{restatable}

The above follow from routine calculations, which we defer to \Cref{app:spinglass-calcs}.

\begin{restatable}{lemma}{gabconcentration}   \label{lem:gab-concentration}
    For every constant $\iota > 0$, there is a constant $c$ such that with probability $1-e^{-cN}$, for all $a,b$, we have $|g_{a,b} - g_{0,0}| \le  \iota\frac{a^2+b^2}{\sqrt{N}} $.
\end{restatable}

\begin{lemma}
    \label{lem:nu-energy-bounded-error}
    With probability $1-e^{-cN^{1/5}}$, $|\Error^{(1)}_{a,b}| = O(1)$ uniformly for all $|a|,|b| < \eps N^{1/10}$, for $\eps$ as in \Cref{lem:band-models-eps-strictly-rs}.
\end{lemma}
\begin{proof}
    We shall show this very high probability bound for a fixed $a,b$. Constructing a net over the relevant $a,b$ and performing a union bound over this net allows us to extend this to a uniform bound for all $a,b$; we omit the details. We may write the error term as
    \begin{align*}
        \Error^{(1)}_{a,b} &= \log \wh{Z}_{a,b} - \frac{N\xi_{a,b}(1)}{2} - \frac{\log \left( 1 - \xi_{a,b}''(0) \right)}{2}  \\
        &\qquad\qquad  - \frac{N \xi_{a,b}''(0)}{4} + \frac{1}{2} \log\det\left( (1+\xi_{a,b}''(0)) \Id - \grad^2 H_{a,b}(0) \right) + O(1).
    \end{align*}
    Due to the bound on $a$ and $b$, the above is $O(1)$ with very high probability by \Cref{thm:partition-fn-and-covariance}(a). The desideratum follows.
\end{proof}

\begin{restatable}{lemma}{smallerrorinnu}
    \label{lem:nu-energy-bounded-error-2}
    For any sufficiently small $\iota > 0$, with probability at least $1-e^{-cN}$, $\left| \Error^{(2)}_{a,b} - \Error^{(2)}_{0,0} \right| = O\left( 1 \right)$ for all $a,b < \iota N^{1/4}$.
\end{restatable}

We relegate the proof of the above to the appendix \Cref{app:spinglass-calcs}. The idea of the proof is that the Hessian $\grad^2 H_{a,b}(0)$ does not deviate too much for small variations in $a,b$ -- the first order terms in the deviation end up being cancelled by the $\xi''_{a,b}(0)/2$ term, while the second order terms are $O(1)$.

\begin{lemma}
    \label{lem:nu-ab-concentration}
    Let $\calE_\iota$ be as in \Cref{lem:planted-to-TAP}.
    With probability $1-e^{-cN^{1/5}}$, either $\calE_\iota$ does not hold, or the following holds.
    For $\eps$ as in \Cref{lem:band-models-eps-strictly-rs},
    \[ \Pr_{(a,b) \sim \nu} \left[ |a| \le \eps N^{1/10} \text{ and } |b| \le \eps N^{1/10} \right] \ge 1 - e^{-cN^{1/5}}. \]
\end{lemma}
\begin{proof}
    On the event $\calE_{\iota}$, we have
    \[ 
        \Pr_{(a,b) \sim \nu} \left[ |a| > \iota N^{1/2} \text{ or } |b| > \iota N^{1/2} \right] \le e^{-cN}. 
    \]
    Thus, let 
    \[
        T = \left\{ (a,b) \in \R^2 : |a| \in [\eps N^{1/10}, \iota N^{1/2}] \text{ or } |b| \in [\eps N^{1/10}, \iota N^{1/2}] \right\}.
    \]
    It suffices to show that $\Pr_{(a,b) \sim \nu} [(a,b) \in T] \le e^{-cN^{1/5}}$ with probability $1-e^{-cN^{1/5}}$.
    Recall the density of $(a,b) \sim \nu$ is given by \eqref{eq:density-of-nu}, and that $\E \wh{Z}_{a,b} = e^{N\xi_{a,b}(1)/2}$.
    Thus, for $\wh{\E}$ denoting expectation with respect to the $\wh{Z}_{a,b}$ alone, 
    \begin{align*}
        &\wh{\E} \int_T
        \exp\left( \log \wh{Z}_{a,b} + (N-4) \log r_{a,b} + \E_{\mu_{\TAP}} \HTAP(v(a,b)) + \sqrt{N} g_{a,b} \right) \dif(a,b) \\
        &= \int_T \exp\lt(N\wh{E}_{a,b} + \fr{N\xi_t(1)}{2} + \sqrt{N} g_{a,b} - 4 \log r_{a,b}\rt) \dif(a,b).
    \end{align*}
    On the event in \Cref{lem:gab-concentration}, we have, for any constant $\iota > 0$, 
    \[
        \sqrt{N} g_{a,b}
        \le \sqrt{N} g_{0,0} + \iota (a^2+b^2).
    \]
    By \Cref{lem:energy-ab-concavity}, 
    \[
        N \wh{E}_{a,b} \le N \wh{E}_{0,0} - \eta (a^2+b^2).
    \]
    Combining shows that 
    \[
        N\wh{E}_{a,b} + \sqrt{N} g_{a,b} - 4 \log r_{a,b}
        \le N\wh{E}_{0,0} + \sqrt{N} g_{0,0} - \fr{\eta}{2} (a^2+b^2)+O(1).
    \]
    Combining shows
    \begin{align*}
        &\wh{\E} \int_T 
        \exp\left( \log \wh{Z}_{a,b} + (N-4) \log r_{a,b} + \E_{\mu_{\TAP}} \HTAP(v(a,b)) + \sqrt{N} g_{a,b} \right) \dif(a,b) \\
        &\le e^{-cN^{1/5}} \exp\lt(N\wh{E}_{0,0} + \fr{N\xi_t(1)}{2} + \sqrt{N} g_{0,0}\rt)
    \end{align*}
    and therefore with probability $1-e^{-cN^{1/5}/2}$ over the $\wh{Z}_{a,b}$,
    \begin{align}   
        \notag
        &\int_T 
        \exp\left( \log \wh{Z}_{a,b} + (N-4) \log r_{a,b} + \E_{\mu_{\TAP}} \HTAP(v(a,b)) + \sqrt{N} g_{a,b} \right) \dif(a,b) \\
        \label{eq:partition-fn-over-bad-ab}
        &\le e^{-cN^{1/5}/2} \exp\lt(N\wh{E}_{0,0} + \fr{N\xi_t(1)}{2} + \sqrt{N} g_{0,0}\rt)
    \end{align}
    On the other hand, \Cref{lem:formula-for-nu} implies that with probability $1-e^{-cN^{1/5}}$, 
    \begin{align*}
        &\log \wh{Z}_{0,0} + (N-4) \log r_{0,0} + \E_{\mu_{\TAP}} \HTAP(v(0,0)) + \sqrt{N} g_{0,0} \\
        &= N\wh{E}_{0,0} + \fr{N\xi_t(1)}{2} + \sqrt{N} g_{0,0} + \Error^{(1)}_{0,0} + \Error^{(2)}_{0,0} + O(1),
    \end{align*}
    and \Cref{lem:nu-energy-bounded-error} implies $|\Error^{(1)}_{0,0}| = O(1)$ with probability $1-e^{-cN^{1/5}}$.
    Furthermore, \Cref{lem:logdet-subgaussian} below implies that $|\Error^{(2)}_{0,0}| \le N^{1/10}$ with probability $1-e^{-cN^{1/5}}$.
    Thus
    \begin{align*}
        &\log \wh{Z}_{0,0} + (N-4) \log r_{0,0} + \E_{\mu_{\TAP}} \HTAP(v(0,0)) + \sqrt{N} g_{0,0} \\
        &\ge N\wh{E}_{0,0} + \fr{N\xi_t(1)}{2} + \sqrt{N} g_{0,0} - 2N^{1/10},
    \end{align*}
    and standard continuity arguments imply that for $T' = \{(a,b) : |a|,|b| \le N^{-10}\}$, 
    \begin{align*}
        &\int_{T'}
        \exp\left( \log \wh{Z}_{a,b} + (N-4) \log r_{a,b} + \E_{\mu_{\TAP}} \HTAP(v(a,b)) + \sqrt{N} g_{a,b} \right) \dif(a,b) \\
        &\ge e^{-3N^{1/10}} \exp\lt(
            N\wh{E}_{0,0} + \fr{N\xi_t(1)}{2} + \sqrt{N} g_{0,0}
        \rt).
    \end{align*}
    Comparing with \eqref{eq:partition-fn-over-bad-ab} implies the conclusion, after adjusting $c$.
\end{proof}

\begin{lemma}
    \label{lem:nu-ab-subgaussianity}
    With probability $1-e^{-cN^{1/5}}$, either $\calE_\iota$ does not hold or the following holds. 
    There exists a random variable $X$ over $\R^2$ (coupled with $\nu$) such that the following holds for $(a,b) \sim \nu$.
    \begin{enumerate}[label=(\alph*)]
        \item With probability at least $1-e^{-cN^{1/5}}$, $X = (a,b)$.
        \item $X$ has mean $O(1)$ and is $O(1)$-subgaussian.
    \end{enumerate}
\end{lemma}
\begin{proof}
    This is an immediate corollary of \Cref{lem:nu-energy-gradient,lem:energy-ab-concavity,lem:gab-concentration,lem:nu-energy-bounded-error,lem:nu-energy-bounded-error-2,lem:nu-ab-concentration}, setting $X$ to be the random variable that is equal to $(a,b)$ if $|a|,|b| < \eps N^{1/10}$, and $0$ otherwise.
\end{proof}

We are now finally prepared to bound $\Cov(\mu)$.
\plantedcovbound*
\begin{proof}
Note that $|a|,|b| \le 2\sqrt{N}$ almost surely.
Let $X$ be as in \Cref{lem:nu-ab-subgaussianity}.
This lemma implies that with probability at least $1 - e^{-cN^{1/5}}$ over the randomness of the Hamiltonian, 
\begin{align*}
    \E_{(a,b)\sim\nu}\bracks*{a^2 + b^2} 
    &= \E[\|X\|^2] + \E_{(a,b)\sim\nu}\bracks*{\bone[X\neq (a,b)](a^2 + b^2)} \\
    &\le O(1) + \Pr(X \neq (a,b)) \cdot 8N = O(1).
\end{align*}
Thus, by plugging in \eqref{eq:first-term-cov-bound} along with this observation into \eqref{eq:cov-bound}, we get that the following holds with probability at least $1-e^{-cN^{1/5}}$
\begin{align*}
    \norm*{\Cov(\mu)} &\le 2C \E_{(a,b)\sim\nu} (1+a^2+b^2) + \norm*{2\E_{(a,b)\sim\nu} (v(a,b)-\boldm)(v(a,b) - \boldm)^{\top}} \\
    &\le O(1) + 2 \E_{(a,b)\sim\nu} \norm*{v(a,b) - \boldm}^2 \\
    &\le O(1) + 2 \E_{(a,b)\sim\nu} \bracks*{O(a^2 + b^2)} \\
    &\le O(1)\mper  \qedhere
\end{align*}
\end{proof}

\def\tgam{{\widetilde \gamma}}
\def\tG{{\widetilde G}}
\def\tQ{{\widetilde Q}}
\def\tT{{\widetilde T}}
\def\tX{{\widetilde X}}
\def\hX{{\widehat X}}
\def\GOOD{{\mathsf{good}}}
\def\Band{{\mathsf{Band}}}
\def\RomI{{\mathrm{I}}}
\def\RomII{{\mathrm{II}}}
\def\RomIII{{\mathrm{III}}}
\def\RomIV{{\mathrm{IV}}}

\section{High-probability covariance bound of replica symmetric spherical spin glass}\label{sec:high-prob-cov}

In this section we prove the main technical input to the proofs in \Cref{sec:tap-slices}.
This takes the form of a high-probability bound on the partition function and covariance matrix (in fact, second moment matrix) of a spherical spin glass in the replica symmetric phase.

In this section, we let $H_N$ be defined as in \eqref{eq:pspin-with-external-field}, with a linear term corresponding to an external field: 
\[
    H_N(\sigma) \coloneqq \sum_{p \ge 1} \fr{\gamma_p}{N^{(p-1)/2}} \sum_{i_1,\dots,i_p = 1}^N \bg_{i_1,\dots,i_p} \sigma_{i_1}\cdots\sigma_{i_p}.
\]
We recall $\xi_{\sim1}(q) = \sum_{p\ge 2} \gamma_p^2 q^p$ denotes the part of $\xi$ without the linear term, and let
\[
    \xi_{\sim2}(q) = \gamma_1^2 q + \sum_{p\ge 3} \gamma_p^2 q^p
\]
denote the part of $\xi$ excluding the degree $2$ term.

The results in this section hold under the following condition, which we restate for reference.
\strictrs*
Throughout this section, we treat $\eps > 0$ as a constant and let $O_\eps(1)$ denote a quantity bounded depending on $\eps$.
\partitionfn*

In the below proofs, we allow the constants $c$ and $C$ to change from line to line, but they will always be uniform in $\eps$.
We always set $C$ sufficiently large depending on $\eps$, and then $c$ sufficiently small depending on $\eps,C$.

\Cref{thm:partition-fn-and-covariance} will be proved through the following pair of propositions.
We introduce the degree-$2$ Hamiltonian
\begin{equation}
    \label{eq:HN2}
    H_{N,2}(\sigma) \coloneqq \fr{\gamma_2}{N^{1/2}} \sum_{i_1,i_2 = 1}^N \bg_{i_1,i_2} \sigma_{i_1}\sigma_{i_2}
    = \fr12 \la \nabla^2 H_N(0) \sigma, \sigma \ra.
\end{equation}
Similarly let $H_{N,\sim2}(\sigma) = H_N(\sigma) - H_{N,2}(\sigma)$ be the non degree-$2$ part of $H_N(\sigma)$.
Define the degree-$2$ Gibbs measure and partition function by
\[
    \dif \mu_{H_{N,2}}(\sigma) = \fr{\exp(H_{N,2}(\sigma))}{Z_{N,2}} \dif\rho(\sigma), \qquad
    Z_{N,2} = \int_{S_N} \exp(H_{N,2}(\sigma)) \dif\rho(\sigma).
\]
Throughout this section, we will let $\E$ denote expectation with respect to the disorder coefficients $\bg_{i_1,\dots,i_p}$, while $\la \cdot \ra$ denotes averaging with respect to $\sigma \sim \mu_{H_N}$ (or several i.i.d. samples $\sigma^1,\sigma^2,\ldots$ from this measure).
Similarly, let $\la \cdot \ra_2$ denote Gibbs average with respect to $\mu_{H_{N,2}}$.

Note that $\nabla^2 H_N(0)$ depends on $H_N$ only through $H_{N,2}$.
\begin{proposition}[Concentration of degree-$2$ partition function; proved in \Cref{ss:deg2-partition-fn}]
    \label{ppn:deg2-partition-fn}
    With probability $1-e^{-cN}$ over $H_{N,2}$, we have $\nabla^2 H_N(0) \preceq (1 + \xi''(0) - \eps^2/8) I_N$ and
    \[
        \lt|
            \log Z_{N,2} - \fr{N\xi''(0)}{2} - \fr{\log(1-\xi''(0))}{2}
            + \fr12 \log \det \lt((1+\xi''(0)) I - \nabla^2 H_N(0)\rt)
        \rt| \le 1/2.
    \]
\end{proposition}

The following is proved in \Cref{ss:nondeg2-positive-prob,ss:nondeg2-high-prob}.

\begin{restatable}{proposition}{propfreeenergycovconcentration}
    \label{ppn:nondeg2}
    There is a $H_{N,2}$-measurable event with probability $1-e^{-cN^{1/5}}$ on which the following holds with probability $1-e^{-cN^{1/5}}$ over $H_{N,\sim2}$.
    \begin{enumerate}
        \item The partition functions $Z_N$, $Z_{N,2}$ satisfy
        \[
            \lt|\log \fr{Z_N}{Z_{N,2}} - \fr{N\xi_{\sim2}(1)}{2} \rt| \le 1/2.
        \]
        \item The Gibbs measure satisfies $\|\la \sigma\sigma^{\top} \ra\|_{\op} \le C(1 + \gamma_1^2 N)$.
    \end{enumerate}
\end{restatable}

\begin{proof}[Proof of \Cref{thm:partition-fn-and-covariance}]
    Immediate from \Cref{ppn:deg2-partition-fn,ppn:nondeg2}, since $\xi(1) = \xi_{\sim2}(1) + \fr12 \xi''(0)$.
\end{proof}
We also show the following concentration of the log determinant in \Cref{thm:partition-fn-and-covariance}.
\begin{lemma}[Proved in \Cref{ss:deg2-partition-fn}]
    \label{lem:logdet-subgaussian}
    There exists a $H_{N,2}$-measurable random variable $X$ that the following holds.
    \begin{enumerate}
        \item With probability $1-e^{-cN}$, $X = \log \det ((1 + \xi''(0)) I - \nabla^2 H_N(0)) - N \xi''(0) / 2$.
        \item $X$ has mean $O_\eps(1)$ and is $O_\eps(1)$-subgaussian.
    \end{enumerate}
\end{lemma}
This implies the quantitative contiguity between the planted and null models, which we restate below for convenience.

\quantcon*
\begin{proof}
    Let $\calE_\GOOD$ be intersection of the event in \Cref{thm:partition-fn-and-covariance}, the event
    \[
        X = \log \det ((1 + \xi''(0)) I - \nabla^2 H_N(0)) - N \xi''(0) / 2
    \]
    from \Cref{lem:logdet-subgaussian}, and the event $X \le t$, for some $t>0$ to be determined.
    Then, after adjusting $c = c(\eps)$ as necessary,
    \[
        \HamDist_{\nullmodel} (\ol{\calE_\GOOD})
        \le e^{-cN^{1/5}} + \bbP(X > t)
        \le e^{-cN^{1/5}} + e^{-c(t - \fr{1}{c})_+^2}.
    \]
    Note that $\log \E Z_N = N\xi(1)/2$, while on the event $\calE_\GOOD$,
    \[
        \log Z_N
        = \fr{N\xi(1) - X + O_\eps(1)}{2}
        \ge \fr{N\xi(1) - t - \fr{1}{c}}{2}.
    \]
    Thus $\fr{\E Z_N}{Z_N} \le e^{\fr{1}{2}(t + \fr{1}{c})}$.
    So,
    \begin{align*}
        \HamDist_{\nullmodel} (\calE)
        &\le \HamDist_{\nullmodel}(\ol{\calE_\GOOD})
        + \int \fr{\E Z_N}{Z_N} \Ind[H_N\in \calE \cap \calE_\GOOD] \dif \HamDist_{\pl}(H_N) \\
        &\le e^{-cN^{1/5}} + e^{-c(t - \fr{1}{c})_+^2}
        + e^{\fr{1}{2}(t + \fr{1}{c})} p.
    \end{align*}
    We then take $t = \fr{1}{c} + \sqrt{\fr{1}{c} \log \fr{1}{p}}$, so that this is bounded by
    \[
        e^{-cN^{1/5}} + \lt(1 + e^{\fr{1}{c} + \sqrt{\fr{1}{c} \log \fr{1}{p}}} \rt) p\mper
    \]
    Further adjusting $c$ proves the desired bound.
\end{proof}

\subsection{Concentration of degree-$2$ partition function}
\label{ss:deg2-partition-fn}

We write $A = \fr12 \nabla^2 H_N(0) = \fr{\sqrt{\xi''(0)}}{2} M$.
It is straightforward to check that $M$ is distributed as a sample from $\GOE(N)$.
\begin{fact}
    \label{fac:xi2nd-at-0}
    We have $\xi''(0) \le 1-\eps$.
\end{fact}
\begin{proof}
    Writing \eqref{eq:strict-rs} as
    \[
        \xi_{\sim1}(q) + q + \log(1-q) \le -\eps q^2 /2
    \]
    and Taylor expanding around $q=0$ implies the result.
\end{proof}
\begin{fact}
    \label{fac:nabla2HN2-psd-domination}
    With probability $1-e^{-cN}$, $\nabla^2 H_N(0) \preceq (1 + \xi''(0) - \eps^2/8) I$.
\end{fact}
\begin{proof}
    With probability $1-e^{-cN}$, we have $\lambda_{\max}(M) \le 2+\eps^2/8$.
    Then,
    \begin{align*}
        \lambda_{\max} \lt( (1 + \xi''(0) - \eps^2/8) I - \nabla^2 H_N(0) \rt)
        &\ge 1 + \xi''(0) - \eps^2/8 - \sqrt{\xi''(0)} (2+\eps^2/8) \\
        &= (1 - \sqrt{\xi''(0)})^2 - \eps^2 (1 + \sqrt{\xi''(0)}) / 8 \\
        &> \eps^2 / 4 - \eps^2 / 4 = 0
    \end{align*}
    by \Cref{fac:xi2nd-at-0}.
\end{proof}
For $\gamma \in (\lambda_{\max}(A),+\infty)$, define
\begin{equation}
    \label{eq:def-G}
    G(\gamma) = \gamma - \fr{1}{2N} \log \det (\gamma I - A).
\end{equation}
Note that
\[
    G'(\gamma) = 1 - \fr{1}{2N} \Tr (\gamma I - A)^{-1}
\]
is continuous and increasing, with $\lim_{\gamma \downarrow \lambda_{\max}(A)} G'(\gamma) = -\infty$ and $\lim_{\gamma \uparrow +\infty} G'(\gamma) = 1$.
Thus $G'$ has a unique root $\gamma_*$ in $(\lambda_{\max}(A),+\infty)$.
The following lemma is a consequence of \cite[Lemma 7.3]{HMP24}, which is proved by an analysis of a Laplace transform of the free energy also used in \cite{BL16}.
\begin{lemma}
    \label{lem:ZN2-laplace}
    With probability $1-e^{-cN}$ over $H_{N,2}$,
    \begin{equation}
        \label{eq:ZN2-laplace}
        Z_{N,2} = (1 + O(N^{-c})) \sqrt{\fr{2}{G''(\gamma_*)}} (2e)^{-N/2} \exp(NG(\gamma_*)).
    \end{equation}
\end{lemma}
\begin{proof}
    Recalling \eqref{eq:HN2}, we have
    \[
        H_{N,2}(\sigma) = \fr{\sqrt{\xi''(0)}}{2} \la M\sigma, \sigma\ra,
    \]
    and \Cref{fac:xi2nd-at-0} implies the factor $\sqrt{\xi''(0)} / 2$ is bounded away from $1/2$.
    Then \cite[Lemma 7.3]{HMP24} (with $u = 0$) implies the result.
\end{proof}
Define $\gamma_0 = (1 + \xi''(0))/2$.
The next lemma shows that, although the variable $\gamma_*$ in \eqref{eq:ZN2-laplace} is random, we may approximate it deterministically by $\gamma_0$.
\begin{lemma}
    \label{lem:gz00-spam}
    For sufficiently large $C$ depending on $\eps$, and sufficiently small $c$ depending on $\eps,C$, with probability $1-e^{-cN}$ the following holds for all $\gamma \in [\gamma_0 - N^{-1/2}, \gamma_0 + N^{-1/2}]$.
    \begin{enumerate}
        \item \label{it:gz00-1} $|G'(\gamma_0)| \le 1 / (C\sqrt{N})$.
        \item \label{it:gz00-2} $|\fr{G''(\gamma)}{2/(1-\xi''(0))} - 1| \le N^{-1/3}$.
    \end{enumerate}
\end{lemma}
\begin{proof}
    Let $\dif\rho_{\smc}(x) = \fr{1}{2\pi} \Ind[|x|\le 2]\sqrt{4-x^2} ~\dif x$ denote Wigner's semicircle law, and
    \[
        f_1(x) = 1 - \fr{1}{1 + \xi''(0) - 2\sqrt{\xi''(0)}x}, \qquad
        f_2(x) = \fr{2}{(1 + \xi''(0) - 2\sqrt{\xi''(0)}x)^2}.
    \]
    For $k \in [2]$, let
    \[
        L_k = \int f_k(x) ~\dif\rho_{\smc}(x).
    \]
    We will show that with probability $1-e^{-cN}$, for each $k \in [2]$, 
    \begin{equation}
        \label{eq:GZ00-BY05-goal}
        \lt|G^{(k)}(\gamma_0) - L_k \rt| \le \fr{1}{C\sqrt{N}}.
    \end{equation}
    Recall that $M \sim \GOE(N)$.
    For $f : \bbR \to \bbR$, define the spectral trace
    \[
        \Tr f(M) = \sum_{i=1}^N f(\lambda_i(M)).
    \]
    Note that $G^{(k)}(\gamma_0) = N^{-1} \cdot \Tr f_k(M)$.
    Define
    \[
        \wt{f}_k(x) = f_k(\min(x,2+\eps^2/8)).
    \]
    By the proof of \Cref{fac:nabla2HN2-psd-domination}, $1 + \xi''(0) - 2\sqrt{\xi''(0)}x \ge \eps^2/8$ for $x\le 2 + \eps^2/8$, so $\wt{f}_k$ is $O_\eps(1)$-Lipschitz.
    Moreover, $\lambda_{\max}(M) \le 2 + \eps^2/8$ with probability $1-e^{-cN}$, and on this event $\Tr f_k(M) = \Tr \wt{f}_k(M)$.

    By \cite[Lemma 1.2(b)]{GZ00}, if we write $M_{i,i} = \sqrt{2/N} Z_{i,i}$, $M_{i,j} = \sqrt{1/N} Z_{i,j}$, then $\Tr \wt{f}_k(M)$ is a $O_\eps(1)$-Lipschitz function of the standard gaussians $(Z_{i,j})_{1\le i\le j\le N}$.
    Thus $\Tr \wt{f}_k(M)$ is $O_\eps(1)$-subgaussian, i.e.
    \[
        \bbP(|\Tr \wt{f}_k(M) - \E \Tr \wt{f}_k(M)| \ge t) \le 2e^{-t^2/C}
    \]
    for some $C = O_\eps(1)$.
    By \cite[Theorem 1.1]{BY05},
    \[
        \Tr \wt{f}_k(M) - N L_k
    \]
    converges in distribution to a gaussian with mean and variance $O_\eps(1)$.
    Combined with subgaussianity of $\Tr \wt{f}_k(M)$, this implies
    \[
        |\E \Tr \wt{f}_k(M) - NL_k| = O_\eps(1).
    \]
    It follows that (after possibly increasing $C = O_\eps(1)$),
    \[
        \bbP(|\Tr \wt{f}_k(M) - NL_k| \ge t)
        \le 2e^{-(t - C)_+^2/C}.
    \]
    Thus
    \begin{align*}
        \bbP(|G^{(k)} - L_k| \ge t)
        &\le \bbP(\Tr f_k(M) \neq \Tr \wt{f}_k(M))
        + \bbP(|\Tr \wt{f}_k(M) - NL_k| \ge Nt) \\
        &\le e^{-cN} + 2e^{-(Nt - C)_+^2/C}.
    \end{align*}
    Plugging in $t = 1/(C\sqrt{N})$ proves \eqref{eq:GZ00-BY05-goal}.
    Next, direct calculations show $L_1 = 0$, $L_2 = \fr{2}{1-\xi''(0)}$.
    The former directly implies conclusion \eqref{it:gz00-1}, and the latter implies
    \[
        \lt|G''(\gamma_0) - \fr{2}{1-\xi''(0)}\rt| \le \fr{1}{C\sqrt{N}}.
    \]
    Moreover, on the probability $1-e^{-cN}$ event that $\lambda_{\max}(M) \le 2 + \eps^2/8$, $G^{(3)}(\gamma) = O_\eps(1)$ for all $\gamma \in [\gamma_0 - N^{-1/2}, \gamma_0 + N^{-1/2}]$.
    This implies the conclusion \eqref{it:gz00-2}.
\end{proof}
\Cref{lem:logdet-subgaussian} is proved by the same method, and we present the proof here.
\begin{proof}[Proof of \Cref{lem:logdet-subgaussian}]
    Let
    \[
        f_0(x) = \log(1 + \xi''(0) - \sqrt{\xi''(0)} x).
    \]
    An elementary calculation shows that
    \[
        L_0 \coloneqq \int f_0(x) ~\dif \rho_{\smc}(X) = \xi''(0)/2.
    \]
    Proceeding as in the above proof, we have
    \[
        \log\det\lt((1+\xi''(0)) I - \sqrt{\xi''(0)} \nabla^2 H_N(0) \rt)
        = \Tr f_0(M).
    \]
    If we take $\wt{f}_0(x) = f_0(\min(x,2+\eps^2/8))$, then $\Tr f_0(M) = \Tr \wt{f}_0(M)$ with probability $1-e^{-cN}$.
    The same proof shows $\Tr \wt{f}_0(M)$ is $O_\eps(1)$-subgaussian, and
    \[
        |\E \Tr \wt{f}_0(M) - NL_0| = O_\eps(1).
    \]
    Thus we may take $X = \Tr \wt{f}_0(M) - NL_0 = \Tr \wt{f}_0(M) - N\xi''(0)/2$.
\end{proof}
\begin{proof}[Proof of \Cref{ppn:deg2-partition-fn}]
    The assertion $\nabla^2 H_N(0) \preceq (1 + \xi''(0) - \eps^2/8) I_N$ is proved in \Cref{fac:nabla2HN2-psd-domination}.
    Suppose the events in \Cref{lem:ZN2-laplace,lem:gz00-spam} occur.
    Since $\gamma_*$ is the solution to $G'(\gamma_*) = 0$, we have
    \[
        |\gamma_0 - \gamma_*|
        \le |G'(\gamma_0)| \cdot \fr{1-\xi''(0)}{2(1 - N^{-1/3})} \le \fr{1}{C\sqrt{N}}.
    \]
    So,
    \[
        N|G(\gamma_0) - G(\gamma_*)|
        \le \fr{N}{2} |\gamma_0 - \gamma_*|^2 \sup_{\gamma \in [\gamma_0 - N^{-1/2}, \gamma_0 + N^{-1/2}]} G''(\gamma)
        \le \fr{1}{C^2} \cdot \fr{2(1+N^{-1/3})}{1-\xi''(0)}
        \le \fr{3}{C^2\eps}.
    \]
    Moreover, $\xi''(\gamma_0) / \xi''(\gamma_*) = 1 + O(N^{-1/3})$.
    Combining with \Cref{lem:ZN2-laplace} shows that, for some $\Delta$ satisfying $|\Delta| \le \fr{3}{C^2\eps}$,
    \begin{align*}
        Z_{N,2} &= (1 + O(N^{-c})) e^\Delta
        \sqrt{\fr{2}{G''(\gamma_0)}} (2e)^{-N/2} \exp(NG(\gamma_0)) \\
        &= (1 + O(N^{-c})) e^\Delta \sqrt{1-\xi''(0)} (2e)^{-N/2}
        \exp(N\gamma_0)
        \det\lt(\gamma_0 I - \fr12 \nabla^2 H_N(0)\rt)^{-1/2} \\
        &= (1 + O(N^{-c})) e^{\Delta} \sqrt{1-\xi''(0)}
        \exp(N\xi''(0)/2)
        \det\lt((1+\xi''(0)) I - \nabla^2 H_N(0)\rt)^{-1/2}.
    \end{align*}
    Taking a logarithm and setting $C$ sufficiently large concludes the proof.
\end{proof}

Finally, the following concentration estimates for samples from $\mu_{H_{N,2}}$ will be useful in the sequel.
This is proved similarly to \cite[Lemma 7.5]{HMP24}, and we defer the proof to \Cref{ss:cov-bd-deferred-proofs}.

Let $v_1,\ldots,v_N$ denote the (unit) eigenvectors of $\nabla^2 H_N(0)$.
These are well defined on the almost sure event that all eigenvalues of $\nabla^2 H_N(0)$ have multiplicity $1$.
\begin{proposition}
    \label{ppn:deg2-subgaussianity}
    With probability $1-e^{-cN}$ over $H_{N,2}$, the following holds.
    Let $\sigma,\sigma^1,\sigma^2 \sim \mu_{H_{N,2}}$, and let $W = \la \sigma,v_i \ra$ for any $i\in [N]$, or $W = \la \sigma^1,\sigma^2\ra / \sqrt{N}$.
    Then:
    \begin{enumerate}
        \item \label{it:deg2-subgaussianity-tail} For any $0\le t\le N^{1/5}$, $\bbP(|W| \ge t) \le 3e^{-ct^2}$.
        \item \label{it:deg2-subgaussianity-mmt} For any $k\in 2\mathbb{N}$, there exists $C_k > 0$ independent of $N$ such that $\la W^k \ra_2 \le C_k$.
    \end{enumerate}
    In particular, part \eqref{it:deg2-subgaussianity-mmt} implies $\|\la \sigma \sigma^\top \ra_2\|_{\op} \le C$.
\end{proposition}

\subsection{Conditional positive probability bounds for non degree-$2$ part}
\label{ss:nondeg2-positive-prob}

In this subsection, we prove the following propositions, which establish a weaker version of \Cref{ppn:nondeg2} with positive instead of high probability.
\begin{proposition}
    \label{ppn:nondeg2-partition-fn}
    There is a $H_{N,2}$-measurable event with probability $1-e^{-cN^{1/5}}$ on which, with probability $1-N^{-1/15}$ over $H_{N,\sim2}$,
    \[
        \lt|\log \fr{Z_N}{Z_{N,2}} - \fr{N\xi_{\sim2}(1)}{2} \rt| = O(N^{-1/15}).
    \]
\end{proposition}
\begin{proposition}
    \label{ppn:nondeg2-covariance}
    There is a $H_{N,2}$-measurable event with probability $1-e^{-cN^{1/5}}$ on which, with probability $1/2$ over $H_{N,\sim2}$,
    \[
        \lt\|
            \fr{Z_N}{Z_{N,2} e^{N\xi_{\sim2}(1)/2}}
            \la \sigma\sigma^{\top} \ra
            - \la \sigma\sigma^{\top} \ra_2
        \rt\|_F^2
        \le C(1 + \gamma_1^4 N^2).
    \]
\end{proposition}

In conjunction with \Cref{ppn:deg2-subgaussianity,ppn:nondeg2-partition-fn}, the above immediately implies a positive probability bound on the second moment matrix $\la \sigma \sigma^\top \ra$.

Both propositions rely on the following truncation to $Z_N$ developed in \cite{HS23}, which allows one to estimate $Z_N$ via the second moment method throughout the strictly RS regime.
As shown in the following lemma, this truncation does not significantly affect the first moment; at the same time, it will force the second moment to be dominated by pairs of nearly-orthogonal points.
\begin{lemma}
    \label{lem:free-energy-typical-truncation}
    The following holds for sufficiently small $c>0$ depending on $\eps$.
    Let
    \[
        T = T(H_N) \coloneqq \lt\{
            \sigma \in S_N :
            \int_{S_N}
            \Ind[|R(\sigma,\tau)| \ge N^{-2/5}]
            e^{H_N(\tau)}
            \dif\rho(\tau)
            \le e^{N\xi(1)/2 - cN^{1/5}}
        \rt\}.
    \]
    Then, we have:
    \begin{align}
        \label{eq:free-energy-typical-1}
        \E \int_{S_N} \Ind[\sigma \not\in T] e^{H_N(\sigma)} \dif\rho(\sigma)
        &\le e^{N\xi(1)/2 - cN^{1/5}}, \\
        \label{eq:free-energy-typical-2}
        \E \int_{S_N} \Ind[\sigma \not\in T] e^{H_{N,2}(\sigma)} \dif\rho(\sigma)
        &\le e^{N\xi''(0)/4 - cN^{1/5}}, \\
        \label{eq:free-energy-typical-3}
        \E \int_{S_N} \Ind[\sigma^1 \not\in T, |R(\sigma^1,\sigma^2)| \le 3N^{-2/5}]
        e^{H_N(\sigma^1)+H_N(\sigma^2)} \dif\rho^{\otimes 2}(\sigma^1,\sigma^2)
        &\le e^{N\xi(1) - cN^{1/5}}, \\
        \label{eq:free-energy-typical-4}
        \E \int_{S_N} \Ind[\sigma^1 \not\in T, |R(\sigma^1,\sigma^2)| \le 3N^{-2/5}]
        e^{H_{N,2}(\sigma^1)+H_N(\sigma^2)} \dif\rho^{\otimes 2}(\sigma^1,\sigma^2)
        &\le e^{N\xi(1)/2 + N\xi''(0)/4 - cN^{1/5}}.
    \end{align}
\end{lemma}
The proof of the above lemma is very similar to \cite[Proposition 3.1]{HS23} and \cite[Lemma 7.9]{HMP24}, and we defer it to \Cref{ss:cov-bd-deferred-proofs}.
As a corollary, we can get control on the first two moments of $Z_N$ with respect to the randomness in $H_{N, \sim2}$.
To this end, let $\E_{\sim2}$ denote expectation with respect to $H_{N,\sim2}$.
\begin{corollary}
    \label{cor:free-energy-typical-truncation}
    There is a $H_{N,2}$-measurable event with probability $1-e^{-cN^{1/5}}$ on which the following holds.
    For
    \begin{equation}
        \label{eq:tT}
        \tT = \tT(H_{N,\sim2}) \coloneqq \lt\{
            \sigma \in S_N :
            \la
                \Ind[|R(\sigma,\tau)| \ge N^{-2/5}]
                e^{H_{N,\sim2}(\tau)}
            \ra_2
            \le e^{N\xi_{\sim2}(1)/2 - cN^{1/5}}
        \rt\},
    \end{equation}
    where the Gibbs average is with respect to $\tau \sim \la \cdot \ra_2$, we have
    \begin{align*}
        \E_{\sim2} \lt\la
            \Ind[\sigma \not\in \tT] e^{H_{N,\sim2}(\sigma)}
        \rt\ra_2
        &\le e^{N\xi_{\sim2}(1)/2 - cN^{1/5}}, \\
        \E_{\sim2} \lt\la
            \Ind[\sigma \not\in \tT]
        \rt\ra_2
        &\le e^{-cN^{1/5}}, \\
        \E_{\sim2} \lt\la
            \Ind[\sigma^1 \not\in \tT, |R(\sigma^1,\sigma^2)| \le 3N^{-2/5}]
            e^{H_{N,\sim2}(\sigma^1) + H_{N,\sim2}(\sigma^2)}
        \rt\ra_2
        &\le e^{N\xi_{\sim2}(1) - cN^{1/5}}, \\
        \E_{\sim2} \lt\la
            \Ind[\sigma^1 \not\in \tT, |R(\sigma^1,\sigma^2)| \le 3N^{-2/5}]
            e^{H_{N,\sim2}(\sigma^2)}
        \rt\ra_2
        &\le e^{N\xi_{\sim2}(1)/2 - cN^{1/5}}.
    \end{align*}
\end{corollary}
\begin{proof}
    By \Cref{ppn:deg2-partition-fn,lem:logdet-subgaussian}, with probability $1-e^{-cN^{2/5}}$ over $H_{N,2}$,
    \[
        Z_{N,2} \ge e^{N\xi''(0)/4 - cN^{1/5}/2}.
    \]
    On this event, for $\sigma \in T$ where $T$ is as in \Cref{lem:free-energy-typical-truncation},
    \begin{align*}
        \lt\la
            \Ind[|R(\sigma,\tau)| \ge N^{-2/5}]
            e^{H_{N,\sim2}(\tau)}
        \rt\ra_2
        &= \fr{1}{Z_{N,2}}
        \int_{S_N}
        \Ind[|R(\sigma,\tau)| \ge N^{-2/5}]
        e^{H_N(\tau)}
        \dif\rho(\tau) \\
        &\le e^{N\xi_{\sim2}(1)/2 - cN^{1/5}/2}.
    \end{align*}
    Here we recall $\xi(1)/2 - \xi''(0)/4 = \xi_{\sim2}(1)/2$.
    So, $\sigma \in \tT(H_{N,\sim2},c/2)$, where this denotes $\tT$ defined with $c/2$ in place of $c$.
    Therefore $T \subseteq \tT(H_{N,\sim2},c/2)$.

    By Markov's inequality and \Cref{lem:free-energy-typical-truncation}, with probability $1-e^{-cN^{1/5}/4}$ over $H_{N,2}$,
    \[
        \E_{\sim2} \int_{S_N} \Ind[\sigma \not\in T] e^{H_N(\sigma)} \dif\rho(\sigma)
        \le e^{N\xi(1)/2 - 3cN^{1/5}/4}
    \]
    On the intersection of these events,
    \begin{align*}
        \E_{\sim2} \lt\la
            \Ind[\sigma \not\in \tT(H_{N,\sim2},c/2)] e^{H_{N,\sim2}(\sigma)}
        \rt\ra_2
        &\le \E_{\sim2} \lt\la
            \Ind[\sigma \not\in T] e^{H_{N,\sim2}(\sigma)}
        \rt\ra_2 \\
        &= \fr{1}{Z_{N,2}} \E_{\sim2} \int_{S_N} \Ind[\sigma \not\in T] e^{H_N(\sigma)} \dif\rho(\sigma) \\
        &\le e^{N\xi_{\sim2}(1)/2 - cN^{1/5}/4}.
    \end{align*}
    The first conclusion follows by adjusting $c$, and the other two conclusions follow similarly.
\end{proof}
For the rest of this subsection, we condition on a realization of $H_{N,2}$ satisfying the following good event.
\begin{definition}
    \label{dfn:E2}
    Let $E_2$ denote the $H_{N,2}$-measurable event that the events in \Cref{ppn:deg2-subgaussianity,cor:free-energy-typical-truncation} hold.
    This occurs with probability $1-e^{-cN^{1/5}}$.
\end{definition}
We can now prove \Cref{ppn:nondeg2-partition-fn}.
\begin{proof}[Proof of \Cref{ppn:nondeg2-partition-fn}]
    Let $\tT$ be as in \Cref{cor:free-energy-typical-truncation}.
    We can write
    \begin{equation}
        \label{eq:ZN-ratio-decomp}
        \fr{Z_N}{Z_{N,2}}
        = \la e^{H_{N,\sim2}(\sigma)} \ra_2
        = X_1 + X_2,
    \end{equation}
    where
    \begin{align*}
        X_1 &= \lt\la \Ind[\sigma \in \tT] e^{H_{N,\sim2}(\sigma)} \rt\ra_2, &
        X_2 &= \lt\la \Ind[\sigma \not\in \tT] e^{H_{N,\sim2}(\sigma)}\rt\ra_2.
    \end{align*}
    We will show that that $X_2$ is much smaller than $\E_{\sim2} X_1$ with high probability, and then control the fluctuations of $X_1$. 
    For all $\sigma \in S_N$, $\E_{\sim2}[e^{H_{N,\sim2}(\sigma)}] = e^{N\xi_{\sim2}(1)/2}$, so \Cref{cor:free-energy-typical-truncation} implies
    \begin{equation}
        \label{eq:ZN-ratio-typical-1mt}
        (1-e^{-cN^{1/5}}) e^{N\xi_{\sim2}(1)/2}
        \le \E_{\sim2} [X_1]
        \le  e^{N\xi_{\sim2}(1)/2}.
    \end{equation}
    On the other hand, by \Cref{cor:free-energy-typical-truncation} and Markov's inequality, with probability $1-e^{-cN^{1/5}/2}$ over $H_{N,\sim2}$,
    \begin{equation}
        \label{eq:ZN-ratio-atypical}
        X_2 \le e^{N\xi_{\sim2}(1)/2 - cN^{1/5}/2},
    \end{equation}
    so $X_2 \le e^{-cN^{1/5}/2} \E_{\sim2} X_1$, as desired.
    
    We now control the fluctuations of $X_1$ by estimating
    \[
        \Var_{\sim2} [X_1] \coloneqq \E_{\sim2} [X_1^2] - \E_{\sim2}[X_1]^2.
    \]
    Then, for $\sigma^1,\sigma^2 \sim \mu_{H_{N,2}}$,
    \[
        \E_{\sim2}[X_1^2]
        = \E_{\sim2} \lt\la \Ind[\sigma^1,\sigma^2 \in \tT] e^{H_{N,\sim2}(\sigma^1) + H_{N,\sim2}(\sigma^2)} \rt\ra_2
        \le \E_{\sim2}[Y_1] + \E_{\sim2}[Y_2],
    \]
    where
    \begin{align}
        \notag
        Y_1 &= \lt\la \Ind[|R(\sigma^1,\sigma^2)| \le N^{-2/5}] e^{H_{N,\sim2}(\sigma^1) + H_{N,\sim2}(\sigma^2)} \rt\ra_2, \\
        \label{eq:Y2}
        Y_2 &= \lt\la \Ind[\sigma^1 \in \tT, |R(\sigma^1,\sigma^2)| \ge N^{-2/5}] e^{H_{N,\sim2}(\sigma^1) + H_{N,\sim2}(\sigma^2)} \rt\ra_2.
    \end{align}
    By the definition of $\tT$ and \eqref{eq:ZN-ratio-typical-1mt},
    \begin{equation}
        \label{eq:EY2}
        \E_{\sim2}[Y_2]
        \le \E_{\sim2} \lt\la \Ind[\sigma^1 \in \tT] e^{H_{N,\sim2}(\sigma^1)} \rt\ra_2
        e^{N\xi_{\sim2}(1)/2 - cN^{1/5}}
        \le
        e^{N\xi_{\sim2}(1) - cN^{1/5}}.
    \end{equation}
    We further calculate
    \[
        \E_{\sim2}[Y_1]
        = e^{N\xi_{\sim2}(1)} \lt\la \Ind[|R(\sigma^1,\sigma^2)| \le N^{-2/5}] e^{N\xi_{\sim2}(R(\sigma^1,\sigma^2))}\rt\ra_2.
    \]
    Recall that in \Cref{thm:partition-fn-and-covariance}, we assumed $\gamma_1^2 \le N^{-4/5}$.
    Thus, for $|R| \le N^{-2/5}$,
    \[
        \xi_{\sim2}(R) = \gamma_1^2 R + O(R^3) = O(N^{-6/5}).
    \]
    It follows that $\E_{\sim2}[Y_1] \le (1 + O(N^{-1/5})) e^{N\xi_{\sim2}(1)}$.

    Combining the above estimates shows $\E_{\sim2}[X_1^2] \le (1 + O(N^{-1/5})) e^{N\xi_{\sim2}(1)}$.
    Further combining with the lower bound in \eqref{eq:ZN-ratio-typical-1mt} shows
    \[
        \Var_{\sim2}[X_1] = O(N^{-1/5}) e^{N\xi_{\sim2}(1)}.
    \]
    By Chebyshev's inequality, with probability $1-N^{-1/15}/2$,
    \[
        |X_1 - \E_{\sim2}[X_1]|
        = O(N^{-1/15}) e^{N\xi_{\sim2}(1)/2}
    \]
    Union bounding with the event in \eqref{eq:ZN-ratio-atypical}, and recalling \eqref{eq:ZN-ratio-typical-1mt}, we conclude that with probability $1-N^{-1/15}$,
    \[
        \fr{Z_N}{Z_{N,2}}
        = (1 + O(N^{-1/15})) e^{N\xi_{\sim2}(1)/2}. \qedhere
    \]
\end{proof}
The above proof also implies the following estimate, which will be useful in the sequel.
\begin{corollary}
    \label{cor:big-overlap-integral}
    On the event $E_2$ (\Cref{dfn:E2}), with probability $1-e^{-cN^{1/5}}$ over $H_{N,\sim2}$,
    \[
        \int \Ind[|R(\sigma^1,\sigma^2)| \ge N^{-2/5}\} e^{H_N(\sigma^1) + H_N(\sigma^2)} \dif\rho^{\otimes 2}(\sigma^1,\sigma^2)
        \le Z_{N,2}^2 e^{N\xi_{\sim2}(1) - cN^{1/5}}.
    \]
\end{corollary}
\begin{proof}
    Dividing through by $Z_{N,2}^2$, it suffices to show, for $\sigma^1,\sigma^2 \sim \mu_{H_{N,2}}$,
    \[
        \lt\la \Ind[|R(\sigma^1,\sigma^2)| \ge N^{-2/5}] e^{H_{N,\sim2}(\sigma^1) + H_{N,\sim2}(\sigma^2)} \rt\ra_2
        \le e^{N\xi_{\sim2}(1) - cN^{1/5}}.
    \]
    The left-hand side is bounded by $Y_2+Y_3$, where $Y_2$ is as in \eqref{eq:Y2} and
    \[
        Y_3 \coloneqq \lt\la \Ind[\sigma^1 \not\in \tT] e^{H_{N,\sim2}(\sigma^1) + H_{N,\sim2}(\sigma^2)}\rt\ra_2
        = \lt\la \Ind[\sigma \not\in \tT] e^{H_{N,\sim2}(\sigma)} \rt\ra_2
        \lt\la e^{H_{N,\sim2}(\sigma)} \rt\ra_2.
    \]
    By \eqref{eq:EY2}, $\E_{\sim2} [Y_2] \le e^{N\xi_{\sim2}(1) - cN^{1/5}}$.
    By \Cref{cor:free-energy-typical-truncation},
    \begin{align*}
        \E_{\sim2} \lt\la \Ind[\sigma \not\in \tT] e^{H_{N,\sim2}(\sigma)} \rt\ra_2 &\le e^{N\xi_{\sim2}(1)/2 - cN^{1/5}}, &
        \E_{\sim2} \la e^{H_{N,\sim2}(\sigma)} \ra_2 &\le e^{N\xi_{\sim2}(1)/2}.
    \end{align*}
    So, the following estimates each hold with probability $1-e^{-cN^{1/5}/4}$ over $H_{N,\sim2}$:
    \begin{align*}
        Y_2 &\le e^{N\xi_{\sim2}(1) - cN^{1/5}/2}, &
        \lt\la \Ind[\sigma \not\in \tT] e^{H_{N,\sim2}(\sigma)} \rt\ra_2 &\le e^{N\xi_{\sim2}(1)/2 - 3cN^{1/5}/4}, \\
        &&
        \la e^{H_{N,\sim2}(\sigma)} \ra_2 &\le e^{N\xi_{\sim2}(1)/2+cN^{1/5}/4}.
    \end{align*}
    The conclusion follows on the intersection of these events, after adjusting $c$.
\end{proof}
We now turn to the proof of \Cref{ppn:nondeg2-covariance}.
By rotational invariance of gaussians, we may assume $\nabla^2 H_N(0)$ is diagonal while keeping the law of $H_{N,\sim2}$ unchanged.
For $i,j\in [N]$, define
\[
    X_{i,j} = \lt\la 
        \sigma_i \sigma_j \lt( e^{H_{N,\sim2}(\sigma) - N\xi_{\sim2}(1)/2} - \Ind[i=j] \rt)
    \rt\ra_2,
\]
and note that this equals the $(i,j)$ entry of the matrix appearing in \Cref{ppn:nondeg2-covariance}. 
For $\sigma \in \bbR^N$ and $i\in [N]$, let $\sigma_{\sim i} \in \bbR^{N-1}$ denote $\sigma$ with coordinate $i$ omitted.
Similarly, for $i\neq j$, let $\sigma_{\sim i,j} \in \bbR^{N-2}$ denote $\sigma$ with coordinates $i$ and $j$ omitted, and by slight abuse of notation let $\sigma_{\sim i,i} = \sigma_{\sim i}$.
For $i,j\in [N]$ (possibly with $i=j$) define analogously to $\tT$
\begin{equation}
    \label{eq:tT-ij}
    \tT_{i,j} \coloneqq \lt\{
        \sigma \in S_N :
        \lt\la 
            \Ind[|R(\sigma_{\sim i,j},\tau_{\sim i,j})| \ge 2N^{-2/5}]
            e^{H_{N,\sim2}(\tau)}
        \rt\ra_2
        \le e^{N\xi_{\sim2}(1)/2 - cN^{1/5}}
    \rt\},
\end{equation}
where we recall the Gibbs average is with respect to $\tau \sim \la \cdot \ra_2$.
Then define
\begin{align*}
    \tX_{i,j} &= \lt\la 
        \Ind[|\sigma_i|,|\sigma_j| \le \log N, \sigma \in \tT_{i,j}]
        \sigma_i \sigma_j \lt( e^{H_{N,\sim2}(\sigma) - N\xi_{\sim2}(1)/2} - \Ind[i=j] \rt)
    \rt\ra_2 \\
    \hX_{i,j} &= \bigg\la
        \Ind[|\sigma^1_i|,|\sigma^1_j|,|\sigma^2_i|,|\sigma^2_j| \le \log N, |R(\sigma^1_{\sim i,j},\sigma^2_{\sim i,j})| \le 2N^{-2/5}] \\
        &\qquad
        \sigma^1_i \sigma^1_j \sigma^2_i \sigma^2_j
        \lt( e^{H_{N,\sim2}(\sigma^1) - N\xi_{\sim2}(1)/2} - \Ind[i=j] \rt)
        \lt( e^{H_{N,\sim2}(\sigma^2) - N\xi_{\sim2}(1)/2} - \Ind[i=j] \rt)
    \bigg\ra_2.
\end{align*}
Note that $\hX_{i,j}$ is the contribution to $X_{i,j}^2$ coming from $\sigma^1,\sigma^2$ that are both not localized to coordinate $i$ or $j$ and have small overlap.
The following two lemmas reduce the task of controlling $X_{i,j}^2$ to bounding $\E_{\sim2} \hX_{i,j}$.
They are proved by manipulating the typicality truncations $\tT$ and $\tT_{i,j}$ similarly to the proofs above; we defer these proofs to \Cref{ss:cov-bd-deferred-proofs}.
\begin{lemma}
    \label{lem:nondeg2-covariance-typical}
    For each $i,j\in [N]$, with probability $1-e^{-c\log^2 N}$ over $H_{N,\sim2}$,
    \[
        X_{i,j}^2 \le 2\tX_{i,j}^2 + e^{-c\log^2 N}.
    \]
\end{lemma}
\begin{lemma}
    \label{lem:nondeg2-covariance-2mt}
    For each $i,j\in [N]$,
    \[
        \E_{\sim2} \tX_{i,j}^2 \le \E_{\sim2} \hX_{i,j} + e^{-cN^{1/5}}.
    \]
\end{lemma}
We now turn to bounding the $\E_{\sim2} \hX_{i,j}$.
This is achieved by the following pair of propositions.
\begin{proposition}
    \label{ppn:hX-ii}
    For any $i\in [N]$, we have $\E_{\sim2} \hX_{i,i} \le C(N \gamma_1^4 + N^{-1})$.
\end{proposition}
\begin{proposition}
    \label{ppn:hX-ij}
    For any distinct $i,j\in [N]$, we have $\E_{\sim2} \hX_{i,j} \le C(\gamma_1^4 + N^{-2})$.
\end{proposition}
Throughout the next two proofs, $\la \cdot \ra_2$ denotes expectation w.r.t. $\sigma^1,\sigma^2 \sim \la \cdot \ra_2$, and we write $R = R(\sigma^1,\sigma^2)$, $R_{\sim i} = R(\sigma^1_{\sim i},\sigma^2_{\sim i})$, and $R_{\sim i,j} = R(\sigma^1_{\sim i,j},\sigma^2_{\sim i,j})$.
\begin{proof}[Proof of \Cref{ppn:hX-ii}]
    By direct calculation,
    \begin{align*}
        \E_{\sim2} \hX_{i,i}
        &= \E_{\sim2} \bigg\la
            \Ind[|\sigma^1_i|,|\sigma^2_i| \le \log N, |R_{\sim i}| \le 2N^{-2/5}] \\
            &\qquad
            (\sigma^1_i)^2 (\sigma^2_i)^2
            \lt( e^{H_{N,\sim2}(\sigma^1) - N\xi_{\sim2}(1)/2} - 1 \rt)
            \lt( e^{H_{N,\sim2}(\sigma^2) - N\xi_{\sim2}(1)/2} - 1 \rt)
        \bigg\ra_2 \\
        &= \lt\la
            \Ind[|\sigma^1_i|,|\sigma^2_i| \le \log N, |R_{\sim i}| \le 2N^{-2/5}]
            (\sigma^1_i)^2 (\sigma^2_i)^2
            \lt( e^{N\xi_{\sim2}(R)} - 1 \rt)
        \rt\ra_2\mper
    \end{align*}
    In view of \Cref{ppn:deg2-subgaussianity}, $\sigma^1_i$ and $\sigma^2_i$ are subgaussian of scale $O(1)$ and $R$ is subgaussian of scale $O(N^{-1/2})$.
    We will see that the above integral is dominated by $|\sigma^1_i| \asymp |\sigma^2_i| \asymp 1$ and $|R| \asymp N^{-1/2}$, in which case Taylor expanding $e^{N\xi_{\sim2}(R)}$ shows this integral has the desired scale.
    Formally, we can write the above integral as $Y_{i,i}^{(1)} + Y_{i,i}^{(2)}$, where
    \begin{align*}
        Y_{i,i}^{(1)}
        &= \lt\la
            \Ind[|\sigma^1_i|,|\sigma^2_i| \le \log N, |R_{\sim i}| \le 2N^{-1/2} \log N]
            (\sigma^1_i)^2 (\sigma^2_i)^2
            \lt( e^{N\xi_{\sim2}(R)} - 1 \rt)
        \rt\ra_2 \\
        Y_{i,i}^{(2)}
        &= \lt\la
            \Ind[|\sigma^1_i|,|\sigma^2_i| \le \log N, 2N^{-1/2} \log N \le |R_{\sim i}| \le 2N^{-2/5}]
            (\sigma^1_i)^2 (\sigma^2_i)^2
            \lt( e^{N\xi_{\sim2}(R)} - 1 \rt)
        \rt\ra_2
    \end{align*}
    are the contributions from $|R_{\sim i}|$ smaller and larger than $2N^{-1/2} \log N$.
    
    We first address $Y_{i,i}^{(2)}$.
    Note that on the event in the indicator in $Y_{i,i}^{(2)}$, $|R| \le 3N^{-2/5}$. Thus, as $\gamma_1^2 \le N^{-4/5}$,
    \[
        N|\xi_{\sim2}(R)| \le N\gamma_1^2 R + O(NR^3) = O(N^{-1/5}).
    \]
    It follows that $|e^{N\xi_{\sim2}(R)} - 1| \le 1$.
    Then, by Cauchy--Schwarz,
    \begin{align}
        \notag
        Y_{i,i}^{(2)}
        &\le \lt\la
            \Ind[|R| \ge N^{-1/2} \log N ]
            (\sigma^1_i)^2 (\sigma^2_i)^2
        \rt\ra_2 \\
        \notag
        &\le \lt\la \Ind[|R| \ge N^{-1/2} \log N ]\rt\ra_2^{1/2}
        \lt\la
            (\sigma^1_i)^2 (\sigma^2_i)^2
        \rt\ra_2^{1/2} \\
        \label{eq:Yii-2-bd}
        &\le e^{-c\log^2 N} \cdot O(1) = e^{-c\log^2 N},
    \end{align}
    where we have used \Cref{ppn:deg2-subgaussianity}\eqref{it:deg2-subgaussianity-tail} for the tail probability and \Cref{ppn:deg2-subgaussianity}\eqref{it:deg2-subgaussianity-mmt} for the coordinate moments.
    
    Next we turn to $Y_{i,i}^{(1)}$.
    On the event in the indicator in $Y_{i,i}^{(1)}$, $|R| \le 3N^{-1/2} \log N$, so
    \[
        N|\xi_{\sim2}(R)| \le N\gamma_1^2 R + O(NR^3) = O(N^{-3/10} \log N),
    \]
    where we recall $\gamma_1^2 \le N^{-4/5}$.
    Thus, Taylor expanding the exponential and $\xi_{\sim2}$,
    \begin{align*}
        Y_{i,i}^{(1)}
        &= \bigg\la
            \Ind[|\sigma^1_i|,|\sigma^2_i| \le \log N, |R_{\sim i}| \le 2N^{-1/2} \log N] (\sigma^1_i)^2 (\sigma^2_i)^2 \\
            &\qquad
            \lt( N\xi_{\sim2}(R) + \tfrac{1}{2}N^2 \xi_{\sim2}(R)^2 + \tfrac{1}{6}N^3 \xi_{\sim2}(R)^3 \rt)
        \bigg\ra_2 + O(N^{-6/5} \log^8 N) \\
        &= \bigg\la
            \Ind[|\sigma^1_i|,|\sigma^2_i| \le \log N, |R_{\sim i}| \le 2N^{-1/2} \log N] (\sigma^1_i)^2 (\sigma^2_i)^2 \\
            &\qquad
            \lt( N(\gamma_1^2 R + \gamma_3^2 R^3 + \gamma_4^2 R^4) + \tfrac{1}{2}N^2 (\gamma_1^2 R + \gamma_3^2 R^3)^2 + \tfrac{1}{6}N^3 \gamma_1^6 R^3 \rt)
        \bigg\ra_2 + O(N^{-11/10} \log^9 N).
    \end{align*}
    By exchangeability of $(\sigma^1_i,-\sigma^1_i)$, $(\sigma^2_i,-\sigma^2_i)$, and $(R_{\sim i}, -R_{\sim i})$, all the odd degree in $R$ terms vanish, leaving
    \[
        Y_{i,i}^{(1)}
        = \tfrac{1}{2}N^2 \gamma_1^4 Q_2 + (N\gamma_4^2 + N^2\gamma_1^2\gamma_3^2) Q_4 + \tfrac{1}{2}N^2 \gamma_3^4 Q_6 + o(N^{-1}),
    \]
    and where we have introduced the notation
    \[
        Q_k \defeq \lt\la
            \Ind[|\sigma^1_i|,|\sigma^2_i| \le \log N, |R_{\sim i}| \le 2N^{-1/2} \log N]
            (\sigma^1_i)^2 (\sigma^2_i)^2 R^k
        \rt\ra_2.
    \]
    By Cauchy--Schwarz and \Cref{ppn:deg2-subgaussianity}, for each $k\in \{2,4,6\}$,
    \[
        Q_k \le \lt\la
            (\sigma^1_i)^4 (\sigma^2_i)^4
        \rt\ra_2^{1/2} \la R^{2k} \ra_2^{1/2}
        = O(N^{-k/2}).
    \]
    This implies $Y_{i,i}^{(1)} \le C(N\gamma_1^4 + N^{-1})$.
    Combining with the bound \eqref{eq:Yii-2-bd} on $Y_{i,i}^{(2)}$ implies the result.
\end{proof}
\begin{proof}[Proof of \Cref{ppn:hX-ij}]
    We calculate as above
    \begin{align*}
        \E_{\sim2} \hX_{i,j}
        &\le \E_{\sim2} \bigg\la
            \Ind[|\sigma^1_i|,|\sigma^1_j|,|\sigma^2_i|,|\sigma^2_j| \le \log N, |R_{\sim i,j}| \le 2N^{-2/5}] \\
            &\qquad \sigma^1_i \sigma^1_j \sigma^2_i \sigma^2_j
            e^{H_{N,\sim2}(\sigma^1) + H_{N,\sim2}(\sigma^2) - N\xi_{\sim2}(1)}
        \bigg\ra_2 \\
        &= \lt\la
            \Ind[|\sigma^1_i|,|\sigma^1_j|,|\sigma^2_i|,|\sigma^2_j| \le \log N, |R_{\sim i,j}| \le 2N^{-2/5}]
            \sigma^1_i \sigma^1_j \sigma^2_i \sigma^2_j
            e^{N\xi_{\sim2}(R)}
        \rt\ra_2\mper
    \end{align*}
    Our strategy for evaluating this will be similar as above, except that because this integral contains $\sigma^1_i \sigma^1_j \sigma^2_i \sigma^2_j$ instead of $(\sigma^1_i)^2 (\sigma^2_i)^2$, we will need to expand the exponential more carefully to obtain cancellations in these terms. 
    Formally, we write the above integral as $Y_{i,j}^{(1)} + Y_{i,j}^{(2)}$ for
    \begin{align*}
        Y_{i,j}^{(1)} &= \lt\la
            \Ind[|\sigma^1_i|,|\sigma^1_j|,|\sigma^2_i|,|\sigma^2_j| \le \log N, |R_{\sim i,j}| \le 2N^{-1/2} \log N]
            \sigma^1_i \sigma^1_j \sigma^2_i \sigma^2_j
            e^{N\xi_{\sim2}(R)}
        \rt\ra_2, \\
        Y_{i,j}^{(2)} &= \lt\la
            \Ind[|\sigma^1_i|,|\sigma^1_j|,|\sigma^2_i|,|\sigma^2_j| \le \log N, 2N^{-1/2} \log N \le |R_{\sim i,j}| \le 2N^{-2/5}]
            \sigma^1_i \sigma^1_j \sigma^2_i \sigma^2_j
            e^{N\xi_{\sim2}(R)}
        \rt\ra_2.
    \end{align*}
    Identically to the previous proof, on the event in the indicator in $Y_{i,j}^{(2)}$ we have $N|\xi_{\sim2}(R)| = O(N^{-1/5})$, so $e^{N\xi_{\sim2}(R)} \le 2$.
    Then, by Cauchy--Schwarz and \Cref{ppn:deg2-subgaussianity},
    \begin{align*}
        |Y_{i,j}^{(2)}|
        &\le 2\lt\la
            \Ind[|R| \ge N^{-1/2} \log N]
            |\sigma^1_i \sigma^1_j \sigma^2_i \sigma^2_j|
        \rt\ra_2 \\
        &\le 2\la \Ind[|R| \ge N^{-1/2} \log N] \ra_2^{1/2}
        \lt\la
            (\sigma^1_i)^4 (\sigma^2_i)^4
        \rt\ra_2^{1/4}
        \lt\la
            (\sigma^1_j)^4 (\sigma^2_j)^4
        \rt\ra_2^{1/4} \\
        &\le 2e^{-c\log^2 N} \cdot O(1) \cdot O(1)
        = e^{-c\log^2 N}.
    \end{align*}
    To address $Y_{i,j}^{(1)}$, define $\Delta_i = \sigma^1_i \sigma^2_i / N$ and $\Delta_j = \sigma^1_j \sigma^2_j / N$, the contributions to $R$ coming from the $i$th and $j$th coordinate, respectively.
    Then, by exchangeability of $(\sigma^1_i,-\sigma^1_i)$ and $(\sigma^1_j,-\sigma^1_j)$,
    \begin{align*}
        4Y_{i,j}^{(1)} &= \bigg\la
            \Ind[|\sigma^1_i|,|\sigma^1_j|,|\sigma^2_i|,|\sigma^2_j| \le \log N, |R_{\sim i,j}| \le 2N^{-1/2} \log N] \\
            &\qquad
            \sigma^1_i \sigma^1_j \sigma^2_i \sigma^2_j
            (e^{N\xi_{\sim2}(R_{\sim i,j} + \Delta_i + \Delta_j)} - e^{N\xi_{\sim2}(R_{\sim i,j} + \Delta_i - \Delta_j)} - e^{N\xi_{\sim2}(R_{\sim i,j} - \Delta_i + \Delta_j)} + e^{N\xi_{\sim2}(R_{\sim i,j} - \Delta_i - \Delta_j)})
        \bigg\ra_2.
    \end{align*}
    Note that on the event in this indicator, $|R_{\sim i,j} \pm \Delta_i \pm \Delta_j| \le 3N^{-1/2} \log N$.
    Define
    \[
        \kappa(x) = e^{N\xi_{\sim2}(x)},
    \]
    and note that
    \begin{align*}
        \sup_{|x| \le 3N^{-1/2} \log N} \kappa^{(4)}(x)
        &= \sup_{|x| \le 3N^{-1/2} \log N} \Big(
            N\xi_{\sim2}^{(4)}(x) + 4N^2 \xi'_{\sim2}(x) \xi^{(3)}_{\sim2}(x) + 3N^2 \xi''_{\sim2}(x)^2 + \\
            &\qquad + 6N^3 \xi'_{\sim2}(x)^2 \xi''_{\sim2}(x) + N^4 \xi'_{\sim2}(x)^4
        \Big) \kappa(x) = O(N^{6/5}),
    \end{align*}
    where we have used that $\sup_{\abs{x} \le 3N^{-1/2}\log N} \kappa(x) \le 2$ and $\gamma_1^2 \le N^{-4/5}$.
    Since $|\Delta_i|,|\Delta_j| \le N^{-1} \log^2 N$ on the event in the indicator, for $s_i,s_j \in \{\pm 1\}$,
    \begin{align*}
        e^{N\xi_{\sim2}(R_{\sim i,j} + s_i\Delta_i + s_j\Delta_j)}
        &= \kappa(R_{\sim i,j})
        + \kappa'(R_{\sim i,j}) (s_i\Delta_i + s_j\Delta_j)
        + \tfrac{1}{2} \kappa''(R_{\sim i,j}) (s_i\Delta_i + s_j\Delta_j)^2 \\
        &\qquad + \tfrac{1}{6} \kappa^{(3)}(R_{\sim i,j}) (s_i\Delta_i + s_j\Delta_j)^3
        + O(N^{6/5}) \cdot (2N^{-1} \log^2 N)^4.
    \end{align*}
    It follows that
    \begin{align*}
        &e^{N\xi_{\sim2}(R_{\sim i,j} + \Delta_i + \Delta_j)} - e^{N\xi_{\sim2}(R_{\sim i,j} + \Delta_i - \Delta_j)} - e^{N\xi_{\sim2}(R_{\sim i,j} - \Delta_i + \Delta_j)} + e^{N\xi_{\sim2}(R_{\sim i,j} - \Delta_i - \Delta_j)} \\
        &= 4\kappa''(R_{\sim i,j}) \Delta_i \Delta_j + O(N^{-14/5} \log^8 N),
    \end{align*}
    and thus
    \begin{align*}
        Y_{i,j}^{(1)}
        &= \lt\la
            \Ind[|\sigma^1_i|,|\sigma^1_j|,|\sigma^2_i|,|\sigma^2_j| \le \log N, |R_{\sim i,j}| \le 2N^{-1/2} \log N]
            \sigma^1_i \sigma^1_j \sigma^2_i \sigma^2_j \Delta_i \Delta_j \kappa''(R_{\sim i,j})
        \rt\ra_2 \\
        &\qquad + O(N^{-14/5} \log^{12} N) \\
        &= \bigg\la
            \Ind[|\sigma^1_i|,|\sigma^1_j|,|\sigma^2_i|,|\sigma^2_j| \le \log N, |R_{\sim i,j}| \le 2N^{-1/2} \log N] \\
            &\qquad
            (\sigma^1_i \sigma^1_j \sigma^2_i \sigma^2_j)^2
            \lt(N^{-1} \xi''_{\sim2}(R_{\sim i,j}) + \xi'_{\sim2}(R_{\sim i,j})^2\rt)
            e^{N\xi_{\sim2}(R_{\sim i,j})}
        \bigg\ra_2 + o(N^{-2}).
    \end{align*}
    On the event in this indicator, $e^{N\xi_{\sim2}(R_{\sim i,j})} \le 2$, and therefore $\xi'_{\sim2}$ and $\xi''_{\sim2}$ can be Taylor expanded to obtain
    \[
        Y_{i,j}^{(1)} \le Y_{i,j}^{(3)} + Y_{i,j}^{(4)} + o(N^{-2}),
    \]
    where
    \begin{align*}
        Y_{i,j}^{(3)}
        &= 6\gamma_3^2 N^{-1} \bigg\la
            \Ind[|\sigma^1_i|,|\sigma^1_j|,|\sigma^2_i|,|\sigma^2_j| \le \log N, |R_{\sim i,j}| \le 2N^{-1/2} \log N] \\
            &\qquad
            (\sigma^1_i \sigma^1_j \sigma^2_i \sigma^2_j)^2
            R_{\sim i,j}
            e^{N\xi_{\sim2}(R_{\sim i,j})}
        \bigg\ra_2, \\
        Y_{i,j}^{(4)}
        &= \E_{\mu_{H_{N,2}}} \bigg\la
            \Ind[|\sigma^1_i|,|\sigma^1_j|,|\sigma^2_i|,|\sigma^2_j| \le \log N, |R_{\sim i,j}| \le 2N^{-1/2} \log N] \\
            &\qquad
            (\sigma^1_i \sigma^1_j \sigma^2_i \sigma^2_j)^2
            \lt(12 \gamma_4^2 N^{-1} R_{\sim i,j}^2 + (\gamma_1^2 + 3\gamma_3^2 R_{\sim i,j}^2)^2 \rt)
            e^{N\xi_{\sim2}(R_{\sim i,j})}
        \bigg\ra_2
    \end{align*}
    On the event in these indicators, we further have
    \begin{equation}
        \label{eq:Rij2-to-R2}
        |R^2 - R_{\sim i,j}^2| = |R - R_{\sim i,j}| |R + R_{\sim i,j}|
        \le (|\Delta_i| + |\Delta_j|) \cdot 5N^{-1/2} \log N = O(N^{-3/2} \log^3 N).
    \end{equation}
    From this it readily follows that
    \begin{align*}
        Y_{i,j}^{(4)}
        &\le 2 \bigg\la
            \Ind[|\sigma^1_i|,|\sigma^1_j|,|\sigma^2_i|,|\sigma^2_j| \le \log N, |R_{\sim i,j}| \le 2N^{-1/2} \log N] \\
            &\qquad
            (\sigma^1_i \sigma^1_j \sigma^2_i \sigma^2_j)^2
            \lt(12 \gamma_4^2 N^{-1} R^2 + (\gamma_1^2 + 3\gamma_3^2 R^2)^2 \rt)
        \bigg\ra_2 + o(N^{-2}) \\
        &= 2\gamma_1^4 \tQ_0 + (12 \gamma_1^2 \gamma_3^2 + 24 \gamma_4^2 N^{-1}) \tQ_2 + 18\gamma_3^2 \tQ_4 + o(N^{-2}),
    \end{align*}
    where
    \[
        \tQ_k = \lt\la
            \Ind[|\sigma^1_i|,|\sigma^1_j|,|\sigma^2_i|,|\sigma^2_j| \le \log N, |R_{\sim i,j}| \le 2N^{-1/2} \log N]
            (\sigma^1_i \sigma^1_j \sigma^2_i \sigma^2_j)^2 R^k
        \rt\ra_2.
    \]
    By Cauchy--Schwarz and \Cref{ppn:deg2-subgaussianity}, for each $k\in \{0,2,4\}$,
    \[
        \tQ_k \le \lt\la
            (\sigma^1_i)^8 (\sigma^2_i)^8
        \rt\ra_2^{1/4} \lt\la
            (\sigma^1_j)^8 (\sigma^2_j)^8
        \rt\ra_2^{1/4}
        \la R^{2k} \ra_2^{1/2}
        = O(N^{-k/2}).
    \]
    This implies $Y_{i,j}^{(4)} \le C(\gamma_1^4 + N^{-2})$.
    To control $Y_{i,j}^{(3)}$, we recall that $|N\xi_{\sim2}(R_{\sim i,j})| = O(N^{-3/10} \log N)$ and Taylor expand the exponential:
    \begin{align*}
        Y_{i,j}^{(3)}
        &= 6\gamma_3^2 N^{-1} \bigg\la
            \Ind[|\sigma^1_i|,|\sigma^1_j|,|\sigma^2_i|,|\sigma^2_j| \le \log N, |R_{\sim i,j}| \le 2N^{-1/2} \log N] \\
            &\qquad
            (\sigma^1_i \sigma^1_j \sigma^2_i \sigma^2_j)^2
            R_{\sim i,j}
            (1 + N\xi_{\sim2}(R_{\sim i,j}))
        \bigg\ra_2 + o(N^{-2}) \\
        &= 6\gamma_3^2 N^{-1} \bigg\la
            \Ind[|\sigma^1_i|,|\sigma^1_j|,|\sigma^2_i|,|\sigma^2_j| \le \log N, |R_{\sim i,j}| \le 2N^{-1/2} \log N] \\
            &\qquad
            (\sigma^1_i \sigma^1_j \sigma^2_i \sigma^2_j)^2
            (R_{\sim i,j} + N\gamma_1^2 R_{\sim i,j}^2 + N\gamma_3^2 R_{\sim i,j}^4)
        \bigg\ra_2 + o(N^{-2}).
    \end{align*}
    By exchangeability of $(R_{\sim i,j},-R_{\sim i,j})$, the contribution of the term $R_{\sim i,j}$ vanishes.
    By \eqref{eq:Rij2-to-R2}, we can further estimate $R_{\sim i,j}^2$ with $R^2$, obtaining
    \begin{align*}
        Y_{i,j}^{(3)}
        &= 6\gamma_3^2 \bigg\la
            \Ind[|\sigma^1_i|,|\sigma^1_j|,|\sigma^2_i|,|\sigma^2_j| \le \log N, |R_{\sim i,j}| \le 2N^{-1/2} \log N] \\
            &\qquad
            (\sigma^1_i \sigma^1_j \sigma^2_i \sigma^2_j)^2
            (\gamma_1^2 R^2 + \gamma_3^2 R^4)
        \bigg\ra_2 + o(N^{-2}) \\
        &= 6\gamma_1^2 \gamma_3^2 \tQ_2 + 6\gamma_3^4 \tQ_4 + o(N^{-2})
        \le C(\gamma_1^2 N^{-1} + N^{-2}).
    \end{align*}
    Combining all of the above estimates concludes the proof.
\end{proof}
\begin{proof}[Proof of \Cref{ppn:nondeg2-covariance}]
    By a union bound, the event in \Cref{lem:nondeg2-covariance-typical} holds for all $i,j\in [N]$ with probability $1-e^{-c\log^2 N}$ (over $H_{N,\sim2}$).
    On this event,
    \begin{equation}
        \label{eq:frob-to-squares}
        \lt\|
            \fr{Z_N}{Z_{N,2} e^{N\xi_{\sim2}(1)/2}}
            \la \sigma\sigma^{\top} \ra
            - \la \sigma\sigma^{\top} \ra_2
        \rt\|_F^2
        = \sum_{i,j=1}^N X_{i,j}^2
        \le 2\sum_{i,j=1}^N \tX_{i,j}^2 + e^{-c\log^2 N}.
    \end{equation}
    Combining \Cref{lem:nondeg2-covariance-2mt,ppn:hX-ii,ppn:hX-ij} shows that
    \[
        \E_{\sim2} \sum_{i,j=1}^N \tX_{i,j}^2
        \le 2C(N^2 \gamma_1^4 + 1) + e^{-cN^{1/5}}
        \le 3C(N^2 \gamma_1^4 + 1).
    \]
    Thus, with probability $2/3$ over $H_{N,\sim2}$, $\sum_{i,j=1}^N \tX_{i,j}^2 \le 9C(N^2 \gamma_1^4 + 1)$.
    Combining with \eqref{eq:frob-to-squares} and taking a final union bound shows that with probability $1/2$ over $H_{N,\sim2}$,
    \[
        \lt\|
            \fr{Z_N}{Z_{N,2} e^{N\xi_{\sim2}(1)/2}}
            \la \sigma\sigma^{\top} \ra
            - \la \sigma\sigma^{\top} \ra_2
        \rt\|_F^2
        \le 18C(N^2 \gamma_1^4 + 1) + e^{-c\log^2 N}
        \le 20C(N^2 \gamma_1^4 + 1).
    \]
    The result follows after adjusting $C$.
\end{proof}

\subsection{From positive to very high probability}
\label{ss:nondeg2-high-prob}
In this section, we boost the positive probability bound on the second moment matrix to a very high probability bound.
To this end, we will show that an appropriate proxy function for the second moment matrix is very Lipschitz.
This will imply the desired concentration by standard gaussian concentration.

Let $\bg \in \R^{N + N^3 + N^4 + \cdots + N^{p_*}}$ be the vectorized collection of all gaussian interactions corresponding to $H_{N,\sim2}$.
Throughout, as in the previous subsection, we will condition on the event $E_2$ from \Cref{dfn:E2} over $H_{N,2}$, which holds with probability $1-e^{-cN^{1/5}}$.
We define the following functions of $\bg$:
\begin{align*}
    F_1(\bg) &= B(1+\gamma_1^2N) \cdot \left(1 - \abs{\log \frac{Z_N}{Z_{N, 2}} - \frac{N\xi_{\sim2}(1)}{2}}\right) \\
    F_2(\bg) &= B(1+\gamma_1^2 N) - \norm{\la \sigma \sigma^\top\ra }_{\op} \\
    F(\bg) &= \max(\min(F_1(\bg), F_2(\bg)), 0)\mcom
\end{align*}
where $B$ is a sufficiently large constant (specified in the proof of \Cref{lem:f-pos-prob}).
\sloppy If we can show that with high probability over $\bg$, $\min(F_1(\bg), F_2(\bg)) \ge 0$, then the conclusion follows.
Indeed, from $F_2(\bg) \ge 0$, we immediately obtain $\norm{\la \sigma \sigma^\top \ra}_{\op} \le B(1+\gamma_1^2N)$.
$F_1$ allows control over the free energy of the $p$-spin model in terms of that of the corresponding $2$-spin model. This gives good control over the overlaps (in a manner to be made precise shortly), which is crucial for establishing the high probability statement. It is also important earlier in this section, in showing that the free energy of the $p$-spin model concentrates well.

Towards this, we first start with the positive probability statement, which was essentially established in the previous subsections.
\begin{lemma}\label{lem:f-pos-prob}
    There exists a constant $B > 0$ such that with probability at least $\frac{1}{3}$, we have $F(\bg) \ge \tfrac{B}{2}(1+\gamma_1^2 N)$.
\end{lemma}
\begin{proof}
    By \Cref{ppn:nondeg2-partition-fn}, with probability $1-O(N^{-1/15})$ over $\bg$, we have $F_1(\bg) \ge \frac{B}{2}(1+\gamma_1^2N)$, so $\frac{Z_N}{Z_{N,2} e^{N\xi_{\sim2}(1)/2}} \ge e^{-1/2}$.
    Intersecting this with the event from \Cref{ppn:nondeg2-covariance} implies that with probability at least $\frac{1}{3}$, 
\begin{align*}
    \norm{ \la \sigma\sigma^{\top} \ra}_{\op} &\le e^{1/2}\norm{ \la \sigma\sigma^{\top} \ra_2}_{\op} + e^{1/2}\sqrt{C(1+\gamma_1^4 N^2)} \\
            &\le \frac{B}{2}(1 + \gamma_1^2N)\mcom
\end{align*}
where we have used $E_2$ to apply \Cref{ppn:deg2-subgaussianity} and after appropriately picking $B$. 
\end{proof}

Let $\calE$ denote the $H_{N, 2}$-measurable event from \Cref{cor:big-overlap-integral}:
\begin{align*}
    \qty{\bg : \int \Ind[|R(\sigma^1,\sigma^2)| \ge N^{-2/5}\} e^{H_N(\sigma^1) + H_N(\sigma^2)} \dif\rho^{\otimes 2}(\sigma^1,\sigma^2) \le Z_{N,2}^2 e^{N\xi_{\sim2}(1) - cN^{1/5}}}\mcom
\end{align*}
which holds with probability $1 - e^{-cN^{1/5}}$ over $\bg$.
The key observation is that this gives us good control on the overlaps.
\begin{lemma}\label{lem:overlap-norm}
    On $\calE$, if $F_1(\bg) \ge 0$, then for any $0 \le p \le \log^2 N$, we have
    \begin{align*}
        \norm{\la \sigma^{\otimes p}\ra}_F^2 \le O(N^{3p/5}).
    \end{align*}
\end{lemma}
\begin{proof}
    By splitting up the expectation based on whether $|R(\sigma^1, \sigma^2)| \ge N^{-2/5}$, on $\calE$ we have
    \begin{align*}
        \norm{\la \sigma^{\otimes p} \ra}_F^2 &= \la NR(\sigma^1, \sigma^2)\ra^p \\
        &\le N^{3p/5} + N^{2p}\frac{1}{Z_N^2} \int\Ind[|R(\sigma^1,\sigma^2)| \ge N^{-2/5}\} e^{H_N(\sigma^1) + H_N(\sigma^2)} \dif\rho^{\otimes 2}(\sigma^1,\sigma^2) \\
        &\le N^{3p/5} + N^{2p} \frac{Z_{N,2}^2}{Z_N^2} e^{N\xi_{\sim2}(1) - cN^{1/5}} \tag{Definition of $\calE$} \\
        &\le N^{3p/5} + N^{2p} e^{1 - cN^{1/5}}\mcom
    \end{align*}
    where the last line used $F_1(\bg) \ge 0$.
    Since $p \le \log^2 N$, the above quantity is $O(N^{3p/5})$, as desired.
\end{proof}
The above is a crucial input to prove Lipschitzness of $F$ on $\calE$.
\begin{lemma}\label{lem:f-lipschitz}
    The function $F$ is $O((1+\gamma_1^2N)N^{-1/10})$-Lipschitz restricted to $\calE$.
\end{lemma}
Before we prove this, let us see how it implies \Cref{ppn:nondeg2}, restated for convenience.

\propfreeenergycovconcentration*

\begin{proof}[Proof of \Cref{ppn:nondeg2}]
    By Kirszbraun's extension theorem, we can extend $F$ to $\wt{F}$ such that each $\wt{F}$ has the same Lipschitz constant as $F$ and agrees with $F$ on $\calE$.
    We can now apply gaussian concentration to $\wt{F}$ to conclude that
    \begin{equation}
        \Pr\left[|\wt{F}(\bg) - \E \wt{F}(\bg)| \ge \tfrac{B}{4}(1+\gamma_1^2N)\right] \ge 1-e^{-cN^{1/5}}\mper \label{eq:gaussian-conc}
    \end{equation}
    By \Cref{lem:f-pos-prob}, with probability at least $\frac{1}{3}$, we have $F(\bg) \ge \frac{B}{2}(1+\gamma_1^2N)$.
    Upon further intersection with $\calE$ (where $\wt{F}(\bg) = F(\bg)$) and the event from \eqref{eq:gaussian-conc}, we conclude $\E \wt{F}(\bg) \ge \frac{B}{4}(1+\gamma_1^2N)$.
    Thus,
    \begin{align*}
        \Pr[F(\bg) = 0]
        &\le \Pr[\calE^c] + \Pr[\wt{F}(\bg) = 0] \\
        &\le e^{-cN^{1/5}} + \Pr\left[|\wt{F}(\bg) - \E \wt{F}(\bg)| \ge \frac{B(1+\gamma_1^2N)}{4}\right] \\
        &\le e^{-cN^{1/5}},
    \end{align*}
    after adjusting $c$.
    
\end{proof}
Finally, let us prove the Lipschitz bound.
\begin{proof}[Proof of \Cref{lem:f-lipschitz}]
    The set $\calE$ is a convex set in $\bg$. Indeed, $e^{H_N(\sigma)}$ is a convex function of $\bg$, so the LHS of the inequality $\calE$ is convex in $\bg$, whereas the RHS does not depend on $\bg$, and sublevel sets of convex functions are convex.
    Furthermore, $F$ is  absolutely continuous (hence differentiable almost everywhere), so to prove $F$ is Lipschitz on $\calE$ it suffices to bound $\norm{\grad F}$ on $\calE$, wherever it is defined.

    The easier case is if $\min(F_1(\bg), F_2(\bg)) < 0$. In this case, $F(\bg) = 0$ in an open neighborhood of $\bg$, so $\grad F(\bg) = 0$ identically.
    Therefore, for the rest of the proof, assume $\min(F_1(\bg), F_2(\bg)) \ge 0$. We will  compute the gradient of the $F_i$'s, and to simplify the calculation, we will take the gradient with respect to $\bg_p \in \R^{N^p}$ corresponding to the degree-$p$ disorder in $H_{N,\sim2}$.

    For $F_1(\bg)$, note that its only dependence on $\bg$ is via $\log Z_N$, we have
    \begin{align*}
        \norm{\grad_{\bg_p} F_1(\bg)} = B(1+\gamma_1^2N)\norm{\grad_{\bg_p} \log Z_N} = B(1+\gamma_1^2N) \cdot \frac{\gamma_p}{N^{(p-1)/2}}\norm{\la \sigma^{\otimes p} \ra}_F,
    \end{align*}
    and since $F_1(\bg) \ge 0$, we can apply  \Cref{lem:overlap-norm} to conclude that
    \begin{align}
        \norm{\grad_{\bg} \log Z_N}^2 &\lesssim \sum_{p \in [p_*] \setminus \qty{2}} \gamma_p^2 N^{-(p-1)} \cdot N^{3p/5} \nonumber \\
        &\le \gamma_1^2 \cdot N^{3/5} + \sum_{p \ge 3} \gamma_p^2 N^{1 - 2p/5}\nonumber \\
        &\lesssim \gamma_1^2 \cdot N^{3/5} + N^{-1/5} \label{eq:f1-grad-bound}
    \end{align}
    Since $\gamma_1^2 \le N^{-4/5}$, we conclude that $\norm{\grad_{\bg} F_1(\bg)} \lesssim (1+\gamma_1^2 N)N^{-1/10}$, as desired.

    Turning now to $F_2(\bg)$, we observe that $\norm{\la \sigma \sigma^\top \ra}_{\op} 
 = \la \angles{u, \sigma}^2 \ra$, where $u$ is the top eigenvector of $\la \sigma\sigma^\top \ra$ with $\norm{u}_2 = 1$. 
 By the envelope theorem, we can evaluate the gradient with $u$ fixed. 
 For any $v \in \R^{N^p}$ with $\norm{v}_2 = 1$, we will upper bound $\angles{v, \grad_{\bg_p} \la \angles{u, \sigma}^2\ra }$. 
 Applying the quotient rule yields
 \begin{align*}
     \angles{v, \grad_{\bg_p} \la \angles{u, \sigma}^2\ra } &= \lt\la \angles{u, \sigma}^2\angles{v, \grad_{\bg_p} H(\sigma)} \rt\ra - \lt\la \angles{u, \sigma}^2\rt\ra \lt\la \angles{v, \grad_{\bg_p} H(\sigma)}\rt\ra \\
     &= \frac{\gamma_p}{N^{(p-1)/2}}\qty(\lt\la \angles{u, \sigma}^2 \angles{v,\sigma^{\otimes p}}\rt\ra - \lt\la \angles{u, \sigma}^2\rt\ra \lt\la \angles{v, \sigma^{\otimes p}}\rt\ra )\mper
 \end{align*}
 
 Consider the first term $\lt\la \angles{u, \sigma}^2 \angles{v, \sigma^{\otimes p}}\rt\ra$. 
 Using H\"{o}lder's inequality with $q = 1 + \log N$ and $q' = 1 + \tfrac{1}{\log N}$, we see
 \begin{align*}
 \lt\la \angles{u, \sigma}^2 \angles{v, \sigma^{\otimes p}}\rt\ra &\le \lt\la \angles{u, \sigma}^{2q'}\rt\ra ^{1/q'} \lt\la \angles{v, \sigma^{\otimes p}}^{q}\rt\ra ^{1/q} \\
 &\le N^{1/\log N} \lt\la \angles{u, \sigma}^{2}\rt\ra  \angles{v^{\otimes q}, \lt\la \sigma^{\otimes pq}\rt\ra }^{1/q} \\
 &\lesssim (1+\gamma_1^2N) \angles{v^{\otimes q}, \lt\la \sigma^{\otimes pq}\rt\ra }^{1/q} \\
 &\le (1+\gamma_1^2N) \cdot N^{3p/10}\mcom
 \end{align*}
 where in the second to last line we have used $F_2(\bg) \ge 0$ to apply the bound $\lt\la \angles{u, \sigma}^{2}\rt\ra  \le O(1+\gamma_1^2N)$, and in the last line we have used $F_1(\bg) \ge 0$, along with $pq \le O(\log N)$, to apply \Cref{lem:overlap-norm}.
The same argument upper bounds the contribution of the second term as $O\left((1+\gamma_1^2N)N^{3p/10}\right)$. 
These bounds (aside from the common factor of $O(1+\gamma_1^2N)$, which we can pull out), exactly match the ones used in the calculation as carried out for $F_1(\bg)$ in \eqref{eq:f1-grad-bound}. 
Hence, the same argument ultimately yields $\norm{\grad_{\bg} F_2(\bg)} \lesssim (1+\gamma_1^2 N)N^{-1/10}$, completing the proof.
\end{proof}

\addcontentsline{toc}{section}{References}
\bibliographystyle{alpha}
\bibliography{main}

\appendix

\section{Annealed Glauber dynamics on discrete domains}
\label{sec:glauber-app}
In this section, we collect the analogous results for weak functional inequalities for Glauber dynamics. 
\begin{definition}[Weak Poincar\'{e} for the hypercube]\label{def:glauber-weak-poincare}
    We say $\pi$ on $\qty{\pm 1}^n$ satisfies a $(\CPoi,\eps)$-weak \Poincare inequality for Glauber dynamics if for all functions $f$,
    \[ \Var_{\pi}[f] \le \frac{1}{\CPoi} \cdot \calE(f,f) + \eps \cdot \osc(f)^2. \]
    Similarly, we say $\pi$ satisfies a $(\MLSI,\eps)$-weak modified log-Sobolev inequality if for all functions $f$, 
    \[ \Ent_{\pi}[f] \le \frac{1}{\MLSI} \cdot \calE(f,\log f) + \eps \cdot \osc(\sqrt{f})^2. \]
\end{definition}
\begin{remark}
The above definition is related to the continuous setting by using the discrete gradient, which can be bounded by $\osc(f)^2$. 
\end{remark}
First, we will need a concavity property for the Dirichlet form for Glauber dynamics, which is well-known. 
We provide a proof for the sake of self-containedness.
\begin{fact}\label{fact:dirichlet-form-convexity}
    Let $\pi$ be a distribution on $\qty{\pm 1}^n$, and $\pi = \E_{\bz \sim \rho} \pi_{\bz}$ a measure decomposition of $\pi$. 
    Then $\calE_{\pi}(f, f) \ge \E_{\bz \sim \rho} \E_{\pi_{\bz}}(f, f)$.
\end{fact}
\begin{proof}
   For Glauber dynamics on the hypercube, we have 
    \begin{align*}
        \calE_{\pi}(f, f) &= \frac{1}{n}\sum_{\norm{x-y}_1 = 2} \frac{\pi(x) \pi(y)}{\pi(x) + \pi(y)} (f(x)-f(y))^2 \\
        &= \frac{1}{n}\sum_{\norm{x-y}_1 = 2} \frac{\E_{\bz \sim \rho}\pi_{\bz}(x) \E_{\bz \sim \rho}\pi_{\bz}(y)}{\E_{\bz \sim \rho}\pi_{\bz}(x) + \E_{\bz \sim \rho}\pi_{\bz}(y)} (f(x)-f(y))^2 \\
        & \ge \E_{\bz \sim \rho} \calE_{\pi_{\bz}}(f, f),
    \end{align*}
    where the last line follows from concavity of the map $(a, b) \mapsto \frac{ab}{a+b}$ for $a, b > 0$. 
\end{proof}
The following lemma transfers a true \Poincare inequality on $\pi$ to a weak \Poincare inequality on $\pi'$ for Glauber dynamics on the hypercube. 
\begin{lemma}
    \label{lem:close-to-poincare-weak-poincare}
    Let $\pi, \pi'$ be distributions on $\qty{\pm 1}^n$ such that $\pi$ satisfies a $\CPoi$-\Poincare inequality for Glauber dynamics and $\dtv{\pi}{\pi'} \le \delta$. Then, $\pi'$ satisfies a weak $\left( \CPoi, 2\delta\right)$-\Poincare inequality for Glauber dynamics. 
\end{lemma}
\begin{proof}
    Again, there exists a coupling $\mathcal{C}$ of $(\pi,\pi')$ such that for $(x,x') \sim \mathcal{C}$, $\Pr[x\neq x'] \le \delta$.
    The main difference is that the Dirichlet form comparison can be bounded in terms of $\osc(f)^2$.
    Arguing as before yields
    \begin{align*}
        \calE_{\pi'}(f,f)
        &\ge \calE_{\pi}(f,f) - \delta \cdot \osc(f)^2  \\ 
            &\ge \CPoi \cdot \Var_{\pi}[f] - \delta \cdot \osc(f)^2 \\
            &\ge \CPoi \cdot \Var_{\pi'}[f] - \delta \cdot \osc(f)^2\left(1 + \CPoi\right)\mper
    \end{align*}
\end{proof}
\begin{remark}
    The above two results also hold more generally if $P$ is the Markov chain associated to a Doob localization scheme (cf. \cite[Section 2.3]{CE22}), such as when $P$ is Glauber dynamics for a general product domain.
\end{remark}

\begin{lemma}
    \label{lem:measure-decomp-weak-poincare-glauber}
    Let $\pi$ be a distribution over $\{\pm 1\}^n$, and $\pi = \E_{\bz \sim \rho} \pi_{\bz}$ a measure decomposition of $\pi$ such that
    \begin{itemize}
        \item for all functions $f$, $\E_{\bz \sim \rho} \Var_{\pi_{\bz}}[f] \ge \Cvar \Var_{\pi}[f]$, and
        \item with probability $1-\eta$ over $\bz \sim \rho$, $\pi_{\bz}$ satisfies a  $(\CPoi,\delta)$-weak \Poincare~inequality with respect to Glauber.
    \end{itemize}
    Then, $\pi$ satisfies a $\left( \CPoi\Cvar, \frac{\delta+\eta}{\Cvar} \right)$-weak \Poincare inequality.
\end{lemma}
\begin{proof}
    The proof is the same as that of \Cref{lem:measure-decomp-weak-poincare}, except in the Langevin case we have $\calE_{\pi}(f, f) = \E_{\bz \sim \rho} \calE_{\pi_{\bz}}(f, f)$, whereas here we apply \Cref{fact:dirichlet-form-convexity} to get the desired inequality.
\end{proof}

Finally, we record the following simple observation connecting weak functional inequalities in discrete domains.
\begin{fact}\label{fact:pi-to-mlsi}
    \sloppy Let $\pi$ be a distribution on a finite state space $\Omega$, and set $C_{\pi} = \frac{1 - 2\pi_{\min}}{\log(1/\pi_{\min} - 1)}$. If $\pi$ satisfies a $(\CPoi,\eps)$-weak \Poincare inequality, then $\pi$ also satisfies a $\left(4\CPoi C_{\pi}, \frac{\eps}{C_{\pi}}\right)$-weak MLSI and a $(\CPoi C_{\pi}, \frac{\eps}{C_{\pi}})$-weak LSI. 
\end{fact}
\begin{proof}
    For finite state spaces, it is well-known that the LSI of the complete graph Markov chain $P_{K(\pi)}$ has $\MLSI = \frac{1-2\pi_{\min}}{\log(1/\pi_{\min} - 1)}$ (see e.g., \cite{DS96}). 
    Furthermore, observe that $\calE_{P_{K(\pi)}}(f, f) = \Var_{\pi}[f]$. 
    Hence, 
    \begin{align*}
        \frac{1}{\CPoi}\calE(f,f) &\ge \Var_{\pi}[f] - \eps \cdot \osc(f)^2 \tag{Weak PI}\\
        &\ge C_{\pi} \Ent_{\pi}[f^2] - \eps \cdot \osc(f)^2, \tag{LSI of $P_{K(\pi)}$}
    \end{align*}
    which establishes the weak LSI.
    For the weak MLSI, one applies the inequality $4\calE(f, f) \le \calE(f^2, \log f^2)$, whose proof  reduces to checking the two-variable inequality $4(\sqrt{u} - \sqrt{v})^2 \le (u-v)\log \frac{u}{v}$ for positive $u, v$. 
\end{proof}

\section{Deferred calculations for spherical spin glasses}  \label{app:spinglass-calcs}

\subsection{The TAP Hamiltonian}

In this subsection, we will prove \Cref{lem:TAP-hamiltonian}, which we restate for convenience.

\lemTAPhamiltonian*

To prove the above, we will need the following formulas for any $p$-spin Hamiltonian $H_N$ with mixture function $\xi$.

\begin{fact}    \label{fact:product-of-partials}
    For any $u,v,m\in\R^N$, we have:
    \[
        \frac{1}{N} \E \angles*{u,\grad H_N(m)} \angles*{v,\grad H_N(m)} = R(u,v)\xi'(R(m,m)) + R(m,u)R(m,v)\xi''(R(m,m))\mper
    \]
\end{fact}

\begin{proof}
    Once we write the derivative as its definition as a limit, the order of the limit and the expectation operator can be swapped by the dominated convergence theorem.
    \begin{align*}
        \frac{1}{N} \E &\angles*{u,\grad H_N(m)} \angles*{v,\grad H_N(m)} =
        \frac{1}{N} \E \lim_{\delta,\eps\to 0} \frac{H_N(m+\delta u) - H_N(m)}{\delta}\cdot\frac{H_N(m+\eps v) - H_N(m) }{\eps} \\
        &= \frac{1}{N} \lim_{\delta,\eps\to 0} \frac{1}{\delta\eps} \E (H_N(m+\delta u) - H_N(m))\cdot(H_N(m+\eps v) - H_N(m)) \\
        &= \lim_{\delta,\eps\to 0} \frac{1}{\delta \eps} \bracks*{ \xi(R(m+\delta u, m+\eps v)) - \xi(R(m+\delta u, m)) - \xi(R(m, m+\eps v)) + \xi(R(m, m)) } \\
        &= R(u,v)\xi'(R(m,m)) + R(m,u)R(m,v)\xi''(R(m,m))\mper  \qedhere
    \end{align*}    
\end{proof}

\begin{fact}    \label{fact:partial-times-hamiltonian}
    For any $u,v,m\in\R^N$, we have:
    \[
        \frac{1}{N} \E \angles*{u,\grad H_N(m)} H_N(v) = R(u,v) \xi'(R(m,v))\mper
    \]
\end{fact}

The proof of the above is analogous to the proof of \Cref{fact:product-of-partials}, and hence omitted.We now prove \Cref{lem:TAP-hamiltonian}.

\begin{proof}[Proof of \Cref{lem:TAP-hamiltonian}]
    The distribution of $\HTAP(\sigma)$ is the same as that of $H_{N,t}(\sigma)|\bx,\grad\FTAP(\boldm) = 0$.
    Recall that $H_{N,t}(\sigma) = N\xi_t\parens*{R(\bx,\sigma)} + \wt{H}(\sigma)$ where $\wt{H}(\sigma)$ is a centered Gaussian process.
    Next, observe that conditioning on $\grad \FTAP(\boldm) = 0$ is the same as conditioning on
    \[
        \grad \wt{H}(\boldm) = -\bx\cdot\xi_t'\parens*{q_{\bx}} - \boldm\cdot\parens*{\theta'\parens*{q_{\boldm}} - \frac{1}{1-q_{\boldm}}}\mper    \numberthis \label{eq:gradient-H}
    \]
    Observe that $\parens*{\HTAP(\sigma)}_{\sigma\in\ScS_N}$ is a Gaussian process, since it is obtained by conditioning on another Gaussian process satisfying affine constraints.
    First, observe that we can write
    \[
        \HTAP(\sigma) = N\xi_t\parens*{R(\bx,\sigma)} + \tHTAP(\sigma) \numberthis \label{eq:TAP-decomp}
    \]
    where $\tHTAP(\sigma) = \wt{H}(\sigma) | \bx,\grad\FTAP(\boldm) = 0$.
    To understand the behavior of $\tHTAP(\sigma)$, we break $\wt{H}(\sigma)$ into a sum of two terms: one term for its projection onto the space $U\coloneqq\left\{\angles*{\grad\wt{H}(\boldm), u}:u\in\R^N\right\}$, and the part that is orthogonal to $U$, and thus independent of $\grad\wt{H}(\boldm)$.
    Concretely, let us write
    \[
        \wt{H}(\sigma) = \angles*{\grad \wt{H}(\boldm), v(\sigma)} + \parens*{ \wt{H}(\sigma) - \angles*{ \grad \wt{H}(\boldm), v(\sigma)} }.   \numberthis \label{eq:indep-decomp}
    \]
    This is true for any $v(\sigma)$, but we have set up the definition such that the second summand is independent of $\grad \wt{H}(\boldm)$. To verify this, since these two random variables are each mean $0$, it suffices to check that for any $u \in S_N$,
    \[
        \E \angles*{\grad \wt{H}(\boldm), u} \parens*{ \wt{H}(\sigma) - \angles*{ \grad \wt{H}(\boldm), v(\sigma)} } = 0\mper
    \]
    By \Cref{fact:product-of-partials,fact:partial-times-hamiltonian}, the left-hand-side of the above is:
    \[
        R(u, \sigma) \xi_t'\parens*{R(\boldm, \sigma)} - R(u, v(\sigma))\xi_t'\parens*{q_{\boldm}} - R(\boldm, u) R(\boldm, v(\sigma))\xi''(q_{\boldm})\mper
    \]
    We would like $v(\sigma)$ to be such that this is $0$ for \emph{all} $u$. Setting $u$ orthogonal to $\boldm$ and $\sigma$ shows that we must have $v(\sigma)$ in the subspace spanned by $\boldm$ and $\sigma$. 
    
    Suppose that $v(\sigma) = \alpha \sigma + \beta \boldm$. Then, plugging this into the above requires that
    \begin{align*}
        0 &= R(\sigma,u) \xi_t'(R(\boldm,\sigma)) - \left( \alpha R(\sigma,u) + \beta R(\boldm,u)\right) \xi_t'(q_{\boldm}) - R(\boldm,u) \left(\alpha R(\boldm,\sigma) + \beta q_{\boldm}\right) \xi_t''(q_{\boldm}) \\
            &= R(\sigma,u) \left( \xi_t'(R(\boldm,\sigma)) - \alpha \xi_t'(q_{\boldm}) \right)\\
            &\qquad- R(\boldm,u) \left( \beta \xi_t'(q_{\boldm}) - \alpha R(\boldm,\sigma) \xi_t''(q_{\boldm}) - \beta q_{\boldm} \xi_t''(q_{\boldm}) \right).
    \end{align*}
    Since this is true for all $u$, each of these two terms must be $0$. That is,
    \[ \alpha = \frac{\xi_t'(R(\boldm,\sigma))}{\xi_t'(q_{\boldm})} \]
    and
    \[ \beta = -\alpha \cdot \frac{R(\boldm,\sigma) \xi_t''(q_{\boldm})}{\xi_t'(q_{\boldm}) + q_{\boldm} \xi_t''(q_{\boldm})}, \]
    so
    \begin{align*}
        v(\sigma) &= \frac{\xi_t'(R(\boldm,\sigma))}{\xi_t'(q_{\boldm})} \left( \sigma - \boldm \cdot \frac{R(\boldm,\sigma) \xi_t''(q_{\boldm})}{\xi_t'(q_{\boldm}) + q_{\boldm} \xi_t''(q_{\boldm})} \right) \\
            &= \frac{\xi_t'(R(\boldm,\sigma))}{\xi_t'(q_{\boldm})} \left( \Id - \frac{R(\boldm,\sigma) \xi_t''(q_{\boldm})}{\xi_t'(q_{\boldm}) + q_{\boldm} \xi_t''(q_{\boldm})} \cdot \frac{\boldm \boldm^\top}{N} \right) \sigma
    \end{align*}
    as defined.

    Now returning to \eqref{eq:indep-decomp}, when we condition on $\bx$ and $\grad\FTAP(\boldm) = 0$, by plugging in \eqref{eq:gradient-H}, we get
    \begin{align*}
        \tHTAP(\sigma) = -\angles*{\bx, v(\sigma)}\cdot\xi'_t(q_{\bx}) &- \angles*{\boldm, v(\sigma)}\cdot\parens*{\theta'(q_{\boldm}) - \frac{1}{1-q_{\boldm}}}\\
        &+ \parens*{\wt{H}(\sigma) - \angles*{ \grad \wt{H}(\boldm), v(\sigma) }}   \numberthis \label{eq:centered-TAP-conditioning}
    \end{align*}
    We use $\wh{H}(\sigma)$ to denote the random variable $\wt{H}(\sigma) - \angles*{\grad\wt{H}(\boldm), v(\sigma)}$, whose distribution remains unaffected by the conditioning, as this random variable is independent of $\bx$ and the event $\grad\FTAP(\boldm) = 0$.
    Since $\wh{H}(\sigma)$ is centered, our expression for $\E\, \HTAP(\sigma)$ follows from \eqref{eq:TAP-decomp} and \eqref{eq:centered-TAP-conditioning}, and the observation that
    \[ R(\boldm,v(\sigma)) = \frac{\xi_t'(R(\boldm,\sigma))}{\gamma'(q_{\boldm})} \cdot R(\boldm,\sigma). \]
    It remains to compute $N^{-1}\Cov\parens*{\HTAP(\sigma),\HTAP(\sigma')}$ for any $\sigma,\sigma'\in\ScS_N$.
    Observe that this is equal to $N^{-1}\E\wh{H}(\sigma)\wh{H}(\sigma')$.
    By \Cref{fact:product-of-partials,fact:partial-times-hamiltonian}, we have that this is equal to:
    \begin{align*}
        \xi_t&\parens*{R(\sigma,\sigma')} - R(v(\sigma), \sigma')\xi_t'(R(\boldm,\sigma')) - R(v(\sigma'), \sigma)\xi_t'(R(\boldm,\sigma)) \\
        &+ R(v(\sigma), v(\sigma'))\xi_t'(q_{\boldm}) + R(\boldm, v(\sigma)) R(\boldm, v(\sigma')) \xi''(q_{\boldm})\mper
    \end{align*}
    The formula for the covariance can be obtained from the above by expanding $v(\sigma)$.
\end{proof}

Next, we look at the mixture function of these ``TAP planted distributions on slices''.

\corslicedist*

\begin{proof}
    Let $\tau,\tau' \in S_{N-2}$, and
    \[ \sigma = v(a,b) + r_{a,b} Q \tau \text{ and } \sigma' = v(a,b) + r_{a,b} Q \tau'. \]
    Recall from \Cref{lem:TAP-hamiltonian} that
    \begin{align*}
        &N^{-1} \Cov \left( \HTAP(\sigma) , \HTAP(\sigma') \right) \\
        &\qquad= \xi_t'(R(\sigma,\sigma')) - R(\sigma,\sigma') \frac{\xi_t'(R(\boldm,\sigma))\xi_t'(R(\boldm,\sigma')}{\xi_t'(q_{\boldm})} + \frac{\xi_t''(q_{\boldm})}{\gamma'(q_{\boldm})\xi_t'(q_{\boldm})} \gamma (R(\boldm,\sigma)) \gamma (R(\boldm,\sigma')).
    \end{align*}
    By the definition of $\sigma$ and $\sigma'$, we have $R(\boldm,\sigma) = R(\boldm,\sigma') = R(\boldm,v(a,b)) = \left( 1 + \frac{a}{\sqrt{N}} \right) q_{\boldm}$, and $R(\sigma,\sigma') = R\left( \|v(a,b)\|^2 + r_{a,b}^2 R(\tau,\tau') \right)$, since $Q$ is an isometry. As a result,
    \begin{align*}
        &N^{-1} \Cov \left( \HTAP(\sigma) , \HTAP(\sigma') \right) \\
        &\qquad= \xi_t'(\|v(a,b)\|^2 + r_{a,b}^2 R(\tau,\tau')) - \left( \|v(a,b)\|^2 + r_{a,b}^2 R(\tau,\tau') \right) \frac{\xi_t'\left( \left( 1 + \frac{a}{\sqrt{N}} \right) q_{\boldm} \right)^2}{\xi_t'(q_{\boldm})} \\
        &\qquad\qquad+ \frac{\xi_t''(q_{\boldm})}{\gamma'(q_{\boldm})\xi_t'(q_{\boldm})} \gamma \left( \left( 1 + \frac{a}{\sqrt{N}} \right) q_{\boldm} \right)^2.
    \end{align*}
    This may be written as
    \[ N^{-1} \Cov\left( \HTAP(\sigma) , \HTAP(\sigma') \right) = \xi_{a,b}(R(\tau,\tau')) + V(a,b). \]
    This implies that $\HTAP(\sigma)$ is equal to $H_{a,b}(\tau) + g_{a,b}$ for some Gaussian process $(H_{a,b}(\tau))_{\tau \in S_{N-2}}$, where $g_{a,b}$ is a centered Gaussian of variance $V(a,b)$. To complete the proof, we must show that the correlation structure of $H_{a,b}$ can be achieved by a $p$-spin model with mixture function $\xi_{a,b}$. To do this, it suffices to show that $\xi_{a,b}$ is indeed a valid mixture function, in that $\xi_{a,b}^{(p)}(0) \ge 0$ for all $p \ge 1$, and $\xi_{a,b}(0) = 0$. The latter of these is clearly true by construction. The former is easily seen to be true for $p \ge 2$, since for such $p$,
    \[ \xi_{a,b}^{(p)}(0) = r_{a,b}^{2p} \xi_{t}^{(p)}(\|v(a,b)\|^2) \ge 0 \]
    since $\xi_t$ is a valid mixture function. For $p = 1$,
    \begin{align*}
        \xi_{a,b}'(0) &= r_{a,b}^2 \cdot \left( \xi_t' \left( \|v(a,b)\|^2 \right) - \frac{\xi_t'\left( q_{\boldm} \left(  1 + \frac{a}{\sqrt{N}} \right) \right)^2}{\xi_t'(q_{\boldm})} \right) \\
            &\stackrel{\eqref{eq:rab-bound-sos}}{\ge} r_{a,b}^2 \cdot \left( \xi_t'\left( q_{\boldm} \left( 1 + \frac{a}{\sqrt{N}} \right)^2 \right) - \frac{\xi_t'\left( q_{\boldm} \left(  1 + \frac{a}{\sqrt{N}} \right) \right)^2}{\xi_t'(q_{\boldm})} \right) \\
            &= \frac{r_{a,b}^2}{\xi_t'(q_{\boldm})} \cdot \left( \xi_t'\left( q_{\boldm} \left( 1 + \frac{a}{\sqrt{N}} \right)^2 \right) \xi_t'(q_{\boldm}) - \xi_t'\left( q_{\boldm} \left(  1 + \frac{a}{\sqrt{N}} \right) \right)^2 \right) \ge 0.
    \end{align*}
    In the first inequality above, we use the fact that $\xi_t'$ is non-decreasing. The final inequality is an application of Cauchy-Schwarz.
\end{proof}

\subsection{Understanding concentration around the codimension-$2$ slice}

Next, we bound the variance of $g_{a,b} - g_{0,0}$.

\gabconcentration*

\begin{proof}[Proof of \Cref{lem:gab-concentration}]
    The strategy is to prove that for any $a,b,a',b'\in\R$ of magnitude $\ll N^{1/4}$, $g_{a,b}-g_{a',b'}$ is a Gaussian of variance 
    $O\left( \cdot\|v(a,b) - v(a',b')\|^4 \right) = O \left( \frac{(a-a')^4 + (b-b')^4}{N^2} \right)$.
    The desideratum then immediately follows by applying Slepian's lemma on $(|g_{a,b}-g_{0,0}|)_{a,b}$ comparing it to the Gaussian process $\langle G, (v(a,b) - v(0,0))(v(a,b) - v(0,0))^{\top}\rangle$ for a standard Gaussian matrix $G$.

    We carry out the calculation for $a',b' = 0$; the general case follows similarly.
    We have $g_{a,b} = N^{-1/2}\parens*{ \HTAP(\sqrt{N}\cdot v(a,b)) - \E \HTAP(\sqrt{N} \cdot v(a,b)) }$, and $g_{0,0} = N^{-1/2} \parens*{ \HTAP(\boldm) - \E \HTAP(\boldm) }$.
    Clearly, $g_{a,b} - g_{0,0}$ is a centered Gaussian process.
    As in the proof of \Cref{lem:TAP-hamiltonian},
    we have that the distribution of $\HTAP(\sigma) - \E \HTAP(\sigma)$ is the same as that of
    \[
        \wt{H}(\sigma) - \angles*{\grad \wt{H}(\boldm), v(\sigma)},
    \]
    where $\wt{H}$ is a Hamiltonian distributed according to the mixture function $\xi_t$.
    Thus, the distribution of $g_{a,b} - g_{0,0}$ is:
    \[
         \wt{H}(u(a,b)) - \angles*{\grad \wt{H}(\boldm), v(u(a,b))} - \wt{H}(\boldm) + \angles*{\grad \wt{H}(\boldm), v(\boldm)}\mper
    \]
    For brevity, we denote $v(a,b)$ as $\boldm+\eps$.
    We express $\wt{H}(\boldm+\eps)$ in its Taylor expansion, and we get:
    \[
        \sqrt{N}(g_{a,b} - g_{0,0}) = \sum_{i\ge 1} \frac{1}{i!} \angles*{ \DIF_i \wt{H}(\boldm), \eps^{\otimes i} } - \angles*{ \grad \wt{H}(\boldm), v(\boldm+\eps) - v(\boldm) }\mper  \numberthis \label{eq:taylor-gab}
    \]
    Expanding out $v(\boldm+\eps) - v(\boldm)$ ultimately yields:
    \[
        v(\boldm+\eps) - v(\boldm) = \eps + R(\boldm,\eps) \eps \frac{\xi_t''(q_{\boldm})}{\gamma'(q_{\boldm})} - \frac{R(\boldm,\eps)^2 \xi_t''(q_{\boldm})^2}{\xi_t'(q_{\boldm})\gamma'(q_{\boldm})}\boldm \mper
    \]
    Plugging in the above into \Cref{eq:taylor-gab} gives:
    \begin{multline*}
        \sqrt{N}(g_{a,b} - g_{0,0}) = \sum_{i\ge 2} \frac{1}{i!} \angles*{ \DIF_i \wt{H}(\boldm), \eps^{\otimes i}}
        \\ - \angles*{\grad\wt{H}(\boldm), \eps} R(\boldm, \eps) \frac{\xi_t''(q_{\boldm})}{\gamma'(q_{\boldm})}
        - \angles*{\grad\wt{H}(\boldm), \boldm} \frac{R(\boldm,\eps)^2 \xi_t''(q_{\boldm})^2}{\xi_t'(q_{\boldm}) \gamma'(q_{\boldm}) }
    \end{multline*}
    We have an explicit expression for $\eps$:
    \[
        \eps = \frac{aq_{\boldm} - bq_{\bx}^2}{\sqrt{N}(q_{\boldm} - q_{\bx}^2)} \boldm + \frac{q_{\boldm} q_{\bx}}{q_{\boldm} - q_{\bx}^2} \parens*{\frac{b-a}{\sqrt{N}}} \bx.
    \]
    This explicit expression can be used to obtain the following bounds on the variances of the above terms:
    \begin{align*}
        \Var\bracks*{ \frac{1}{i!}\angles*{ \DIF_i \wt{H}(\boldm), \eps^{\otimes i} } } &\le \frac{O\parens*{ a^{2i} + b^{2i} }  }{N^{i-1}} \\
        \Var\bracks*{ \angles*{\grad\wt{H}(\boldm), \eps} R(\boldm, \eps) \frac{ \xi_t''(q_{\boldm}) }{ \gamma'(q_{\boldm})} } &= \frac{O\parens*{  a^4 + b^4 }}{N} \\
        \Var\bracks*{  \angles*{ \grad\wt{H}(\boldm), \boldm } \frac{R(\boldm, \eps)^2 \xi_t''(q_{\boldm})}{ \xi_t'(q_{\boldm}) \gamma'(q_{\boldm}) }  } &\le \frac{   O\parens*{a^4 + b^4 }  }{N}
    \end{align*}
    The expression for $\sqrt{N}(g_{a,b} - g_{0,0})$ only involves a constant number of terms, and since the first term enumerates over $i \ge 2$, and since $|a|,|b| \le \sqrt{N}$, we have an overall bound of $\frac{O\parens*{a^{4} + b^{4} }}{N}$.
    Dividing by $\sqrt{N}$ gives the desired variance bound.
\end{proof}

\nuenergygradient*

\begin{proof}
    Recall
    \begin{align*}
        \wh{E}_{a,b} &= \frac{1}{2} \underbrace{\left( \log r_{a,b}^2 - \xi_t(\|v(a,b)\|^2) - r_{a,b}^2 \cdot \frac{\xi_t'\left( q_{\boldm} \left( 1 + \frac{a}{\sqrt{N}} \right) \right)^2}{\xi_t'(q_{\boldm})} \right)}_{\text{(I)}} \\
            &\qquad + \underbrace{\xi_t\left( q_{\bx} \left( 1 + \frac{b}{\sqrt{N}} \right) \right) + \frac{\gamma\left( q_{\boldm} \left( 1 + \frac{a}{\sqrt{N}} \right) \right)}{\gamma'(q_{\boldm})} \cdot \left( (1-q_{\boldm}) \xi_t''(q_{\boldm}) + \frac{1}{1-q_{\boldm}} \right)}_{\text{(II)}} \\
            &\qquad - \frac{\gamma(q_{\bx})}{\xi_t'(q_{\boldm})} \cdot \underbrace{\xi_t'\left( q_{\boldm} \left( 1 + \frac{a}{\sqrt{N}} \right) \right) \cdot \left( \left( 1 + \frac{b}{\sqrt{N}} \right) -  q_{\boldm} \cdot \frac{\xi_t''(q_{\boldm})}{\gamma'(q_{\boldm})} \cdot \left( 1 + \frac{a}{\sqrt{N}} \right) \right)}_{\text{(III)}}.
    \end{align*}
    Because
    \[ \|v(a,b)\|^2 = q_{\boldm} \left( 1 + \frac{a}{\sqrt{N}} \right)^2 + \frac{q_{\boldm} q_{\bx}}{q_{\boldm} - q_{\bx}^2} \cdot \left( \frac{a-b}{\sqrt{N}} \right)^2, \]
    we have
    \[ \restr{\grad \|v(a,b)\|^2}{(a,b) = (0,0)} = - \restr{\grad r_{a,b}^2}{(a,b) = (0,0)} = \left(\frac{2q_{\boldm}}{\sqrt{N}},0\right). \]
    We also have $r_{0,0}^2 = 1-q_{\boldm}$ and $\|v(0,0)\|^2 = q_{\boldm}$.
    Let us start by computing the derivative with respect to $a$. We have
    \begin{align*}
        \sqrt{N} \cdot \restr{\partial_a \text{(III)}}{(a,b) = (0,0)} &= \xi_t''(q_{\boldm}) \cdot q_{\boldm} \cdot \left( 1 - \frac{q_{\boldm} \xi_t''(q_{\boldm})}{\gamma'(q_{\boldm})} \right) + \xi_t'(q_{\boldm}) \cdot \left( - q_{\boldm} \cdot \frac{\xi_t''(q_{\boldm})}{\gamma'(q_{\boldm})} \right) \\
            &= \frac{q_{\boldm} \xi_t''(q_{\boldm})}{\gamma'(q_{\boldm})} \left( \gamma'(q_{\boldm}) - q_{\boldm} \xi_t''(q_{\boldm}) - \xi_t'(q_{\boldm}) \right) = 0.
    \end{align*}
    Next,
    \begin{align*}
        &\sqrt{N} \cdot \restr{\partial_a \text{(I)}}{(a,b) = (0,0)} \\
        &\qquad = \frac{1}{r_{0,0}^2} \cdot (-2q_{\boldm}) - \xi_t'(\|v(0,0)\|^2) \cdot (2q_{\boldm}) - (-2q_{\boldm}) \cdot \left( \frac{\xi_t'(q_{\boldm})^2}{\xi_t'(q_{\boldm})} \right) - r_{0,0}^2 \cdot \frac{2\xi_t'(q_{\boldm}) \cdot \xi_t''(q_{\boldm}) \cdot q_{\boldm}}{\xi_t'(q_{\boldm})} \\
        &\qquad = \frac{-2q_{\boldm}}{1-q_{\boldm}} - 2q_{\boldm}\xi_t'(q_{\boldm}) + 2q_{\boldm} \xi_t'(q_{\boldm}) - 2q_{\boldm}(1-q_{\boldm}) \xi_t''(q_{\boldm}) \\
        &\qquad = \frac{-2q_{\boldm}}{1-q_{\boldm}} - 2q_{\boldm}(1-q_{\boldm}) \xi_t''(q_{\boldm}).
    \end{align*}
    Finally,
    \begin{align*}
        &\sqrt{N} \cdot \restr{\partial_a\text{(II)}}{(a,b)=(0,0)} \\
        &\qquad = \frac{\gamma'(q_{\boldm}) \cdot q_{\boldm}}{\gamma'(q_{\boldm})} \cdot \left( (1-q_{\boldm}) \xi_t''(q_{\boldm}) + \frac{1}{1-q_{\boldm}} \right) \\
        &= - \frac{1}{2} \cdot \sqrt{N} \cdot \restr{\partial_a \text{(I)}}{(a,b) = (0,0)}
    \end{align*}
    as desired. The derivative with respect to $b$ is much simpler, since the derivative of $r_{a,b}^2$ with respect to $b$ is $0$ at $(0,0)$. Consequently, $\restr{\partial_b \text{(I)}}{(a,b) = (0,0)} = 0$, $\restr{\partial_b \text{(II)}}{(a,b) = (0,0)} = q_{\bx} \xi_t'(q_{\bx}) = \gamma(q_{\bx})$, and $\restr{\partial_b \text{(III)}}{(a,b) = (0,0)} = \xi_t'(q_{\boldm})$, completing the proof.
\end{proof}

\lemenergyconcavity*

\begin{proof}
    For ease of notation, define $\wt{E}_{a,b} = \wh{E}_{\sqrt{N}a,\sqrt{N}b}$, $\wt{r}_{a,b} = r_{\sqrt{N}a,\sqrt{N}b}$, and $\wt{v}(a,b) = v(\sqrt{N}a,\sqrt{N}b)$. As in the previous lemma, recall
    \begin{align*}
        \wt{E}_{a,b} &= \frac{1}{2} \left( \log \wt{r}_{a,b}^2 - \xi_t(\|\wt{v}(a,b)\|^2) - \wt{r}_{a,b}^2 \cdot \frac{\xi_t'\left( q_{\boldm} \left( 1 + a \right) \right)^2}{\xi_t'(q_{\boldm})} \right) \\
            &\qquad + \xi_t\left( q_{\bx} \left( 1 + b \right) \right) + \frac{\gamma\left( q_{\boldm} \left( 1 + a \right) \right)}{\gamma'(q_{\boldm})} \cdot \left( (1-q_{\boldm}) \xi_t''(q_{\boldm}) + \frac{1}{1-q_{\boldm}} \right) \\
            &\qquad - \frac{\gamma(q_{\bx})}{\xi_t'(q_{\boldm})} \cdot \xi_t'\left( q_{\boldm} \left( 1 + a \right) \right) \cdot \left( \left( 1 + b \right) -  q_{\boldm} \cdot \frac{\xi_t''(q_{\boldm})}{\gamma'(q_{\boldm})} \cdot \left( 1 + a \right) \right).
    \end{align*}
    Because the Hessian is Lipschitz in all the parameters involved, it suffices to prove the negative definiteness of the Hessian at $(0,0)$, under the assumption that $q_{\boldm}=q_{\bx}=q$, where $q$ (formerly denoted $q_*(t)$) satisfies $\xi_t'(q) = \frac{q}{1-q}$. Under these constraints, we have $\gamma(q) = q \xi_t'(q) = \frac{q^2}{1-q}$ and $\gamma'(q) = \xi_t'(q) + q \xi_t''(q) = \frac{q}{1-q} \left( 1 + (1-q) \xi_t''(q) \right)$. $\wt{E}_{a,b}$ simplifies as
    \begin{align*}
        \wt{E}_{a,b} &= \frac{1}{2} \underbrace{\left( \log \wt{r}_{a,b}^2 - \xi_t(\|\wt{v}(a,b)\|^2) -\wt{r}_{a,b}^2 \cdot \frac{\xi_t'\left( q \left( 1 + a \right) \right)^2}{\xi_t'(q)} \right)}_{\text{(I)}} \\
            &\qquad + \underbrace{\xi_t\left( q \left( 1 + b \right) \right) + \frac{\gamma\left( q \left( 1 + a \right) \right)}{q} \cdot \frac{1 + (1-q)^2 \xi_t''(q)}{1 + (1-q)\xi_t''(q)}}_{\text{(II)}} \\
            &\qquad - \underbrace{q \left( 1 + b \right) \cdot \xi_t'\left( q \left( 1 + a \right) \right) + q^2 \left( 1 + a \right) \cdot \frac{\xi_t''(q)}{\gamma'(q)}}_{\text{(III)}}.
    \end{align*}
    We have $\restr{\partial_b^2 \text{(III)}}{(0,0)} = 0$, and
    \[ \restr{\partial_b^2 \text{(II)}}{(0,0)} = \xi_t''(q) \cdot q^2. \]
    We have that $\restr{\partial_b \wt{r}_{a,b}^2}{(0,0)} = 0$, and $\restr{\partial_b^2 \wt{r}_{a,b}^2}{(0,0)} = \frac{-2q^2}{1-q}$. Consequently,
    \begin{align*}
        \restr{\partial_b^2 \text{(I)}}{(0,0)} &= \frac{1}{2} \left( \frac{1}{r_{0,0}^2} \cdot \restr{\partial_b^2 \wt{r}_{a,b}^2}{(0,0)} - \xi_t'(\|v(0,0)\|^2) \cdot \restr{\partial_b^2 \|\wt{v}(a,b)\|^2}{(0,0)} - \restr{\partial_b^2 \wt{r}_{a,b}^2}{(0,0)} \cdot \xi_t'(q) \right) \\
            &= \frac{-q^2}{(1-q)^2}.
    \end{align*}
    It follows that
    \[ \restr{\partial_b^2 \wh{E}_{a,b}}{(0,0)} = \xi_t''(q) \cdot q^2 - \frac{q^2}{(1-q)^2}. \]
    
    Similarly, we have $\restr{\partial_a\partial_b \text{(II)}}{(0,0)} = 0$, and
    \[ \restr{\partial_a\partial_b \text{(III)}}{(0,0)} = \xi_t''(q) \cdot q^2. \]
    We have that $\restr{\partial_a\partial_b \wt{r}_{a,b}^2}{(0,0)} = \frac{-2q^2}{1-q}$. Consequently,
    \begin{align*}
        \restr{\partial_a \partial_b \text{(I)}}{(0,0)} &= \frac{1}{2} \left( \frac{1}{r_{0,0}^2} \cdot \restr{\partial_a \partial_b \wt{r}_{a,b}^2}{(0,0)} - \xi_t'(\|v(0,0)\|^2) \cdot \restr{\partial_a\partial_b \|\wt{v}(a,b)\|^2}{(0,0)} - \restr{\partial_a\partial_b \wt{r}_{a,b}^2}{(0,0)} \cdot \xi_t'(q) \right) \\
            &= \frac{-q^2}{(1-q)^2}.
    \end{align*}
    It follows that
    \[ \restr{\partial_a\partial_b \wh{E}_{a,b}}{(0,0)} = \xi_t''(q) \cdot q^2 - \frac{q^2}{(1-q)^2}. \]
    It remains to compute the second derivative with respect to $a$. We have
    \[ \restr{\partial_a^2 \text{(III)}}{(0,0)} = q^3 \xi_t'''(q). \]
    We also have
    \begin{align*}
        \restr{\partial_a^2 \text{(II)}}{(0,0)} &= q \gamma''(q) \cdot \frac{1 + (1-q)^2 \xi_t''(q)}{1 + (1-q)\xi_t''(q)} \\
            &= q \cdot \frac{1 + (1-q)^2 \xi_t''(q)}{1 + (1-q)\xi_t''(q)} \cdot \left( 2 \xi_t''(q) + q \xi_t'''(q) \right) \\
            &\le q^2 \xi_t'''(q) + 2q\xi_t''(q).
    \end{align*}
    We have $\restr{\partial_a^2 \wt{r}_{a,b}^2}{(0,0)} = - \left( 2q + \frac{2q^2}{1-q} \right) = - \frac{2q}{1-q}$ and $\restr{\partial_a \wt{r}_{a,b}^2}{(0,0)} = -2q$. Finally,
    \begin{align*}
        \restr{\partial_a^2 \text{(I)}}{(0,0)} &= \frac{1}{r_{0,0}^2} \cdot \frac{-2q}{1-q} - \frac{1}{r_{0,0}^4} \cdot (-2q)^2 - \xi_t'(q) \cdot \frac{-2q}{1-q} - \xi_t''(q) \cdot (-2q)^2 - \frac{-2q}{1-q} \cdot \xi_t'(q) \\
            &\qquad- 2 \cdot (-2q) \cdot \frac{2 \xi_t'(q) \xi_t''(q) q}{\xi_t'(q)} - (1-q) \cdot \frac{2q^2 (\xi_t'(q) \xi_t'''(q) + \xi_t''(q)^2)}{\xi_t'(q)} \\
            &= \frac{-2q}{(1-q)^2} - \frac{4q^2}{(1-q)^2} - 4q^2 \xi_t''(q) + 8q^2 \xi_t''(q) - 2q^2(1-q)\xi_t'''(q) - 2q^2(1-q) \cdot \frac{\xi_t''(q)^2}{\xi_t'(q)} \\
            &= \frac{-2q(2q+1)}{(1-q)^2} + 4q^2\xi_t''(q) - 2q^2(1-q) \xi_t'''(q) - 2q(1-q)^2 \xi_t''(q)^2.
    \end{align*}
    Therefore,
    \begin{align*}
        \restr{\partial_a^2 \wh{E}_{a,b}}{(0,0)} &\le -q^3 \xi_t'''(q) + q^2 \xi_t'''(q) + 2q\xi_t''(q) - \frac{q(2q+1)}{(1-q)^2} \\
            &\qquad + 2q^2\xi_t''(q) - q^2(1-q) \xi_t'''(q) - q(1-q)^2 \xi_t''(q)^2 \\
            &= \frac{-q(2q+1)}{(1-q)^2} + 2q\xi_t''(q) + 2q^2 \xi_t''(q) - q(1-q)^2 \xi_t''(q)^2.
    \end{align*}
    To conclude, let us check that the Hessian is negative definite. Because $\xi_t''(q) \cdot q^2 - \frac{q^2}{(1-q)^2} < 0$ by the SL condition \eqref{eq:SL-condition}, it suffices to check that
    \[ \restr{\partial_a^2 \wh{E}_{a,b}}{(0,0)} < q^2 \xi_t''(q) - \frac{q^2}{(1-q)^2}. \]
    This is true if and only if
    \[ q(1-q)^2 \xi_t''(q)^2 - q(q+2) \xi_t''(q) + \frac{q(q+1)}{(1-q)^2} > 0. \]
    It is not difficult to see that this is true if $\xi_t''(q)$ is less than the smaller root of the above quadratic, which is equal to
    \[ \frac{(q+2) - \sqrt{(q+2)^2 - 4(q+1)}}{2(1-q)^2} = \frac{1}{(1-q)^2}. \]
    This is true by the SL condition \eqref{eq:SL-condition}, concluding the proof.
\end{proof}

Next, we shall prove \Cref{lem:nu-energy-bounded-error-2}. Recall the definition
\[ \Error^{(2)}_{a,b} = \frac{N\xi''_{a,b}(0)}{4} - \frac{1}{2} \log \det \left( (1+\xi''_{a,b}(0)) \Id - \grad^2 H_{a,b}(0) \right), \]
where $\grad^2 H_{a,b}(0)$ is equal to the restriction of $r_{a,b}^2 \cdot \grad^2 \HTAP(v(a,b))$ restricted to the codimension-$2$ subspace orthogonal to $\boldm$ and $\bx$.

\smallerrorinnu*

Let us start by computing the correlation structure of the random matrices $\grad^2 H_{a,b}(0)$. Note that $\grad^2 H_{a,b}(0)$ is an $(N-2)$-dimensional GOE matrix scaled by $\sqrt{\xi''_{a,b}(0)}$.

\begin{fact} \label{fact: hessian-correlation-structure}
    For $\sigma^1,\sigma^2$,
    \begin{align*}
        &\frac{1}{N} \E \langle \grad^2 \HTAP(\sigma^1) , u^{1} \otimes u^{2} \rangle \langle \grad^2 \HTAP(\sigma^2) , v^{1} \otimes v^{2} \rangle \\
        &\qquad\qquad= \xi_t''\left( R(\sigma^1,\sigma^2) \right) \cdot \left( R(u^{1} , v^{2}) \cdot R(u^{2} , v^{1}) + R(u^{1},v^{1}) \cdot R(u^{2},v^{2}) \right).
    \end{align*}
    In particular,
    \[ \frac{1}{N} \E \langle \grad^2 \HTAP(\sigma^1) , \grad^2 \HTAP(\sigma^2) \rangle = \xi_t''\left( R(\sigma^1,\sigma^2) \right). \]
\end{fact}
The above follows from calculations similar to those involved in the proofs of \Cref{fact:product-of-partials,fact:partial-times-hamiltonian}; we omit the details.

\begin{proof}[Proof of \Cref{lem:nu-energy-bounded-error-2}]
    We shall prove the statement for a fixed $a,b$; a union bound over $a,b$ implies the boundedness for all $a,b \le \iota N^{1/4}$.
    
    Recalling that $\grad^2 H_{a,b}(0) = r_{a,b}^{2} \grad^2 \HTAP(\sqrt{N} \cdot v(a,b))$ is a GOE matrix scaled by $\sqrt{\xi''_{a,b}(0)}$. By \Cref{fact: hessian-correlation-structure},
    \begin{align*}
        \frac{1}{N} \E \langle \grad^2 H_{a,b}(0) , \grad^2 H_{0,0}(0) \rangle &= r_{a,b}^{2} \cdot r_{0,0}^{2} \cdot \frac{1}{N} \E \langle \grad^2 \HTAP(\sqrt{N} \cdot v(a,b)) , \grad^2 \HTAP(\sqrt{N} \cdot v(0,0)) \rangle \\
            &= r_{a,b}^{2} \cdot r_{0,0}^{2} \cdot \xi_t''\left( \langle v(a,b) , v(0,0) \rangle \right).
    \end{align*}
    For comparison, we have
    \[ \frac{1}{N} \E \norm*{ \grad^2 H_{a,b}(0) }_F^2 = \xi_{a,b}''(0) = r_{a,b}^4 \cdot \xi_t''\left( \|v(a,b)\|^2 \right). \]
    For succinctness of notation, let
    \[ \begin{pmatrix}
        \alpha_1 & \rho \\ \rho & \alpha_2
    \end{pmatrix}
    =
    \begin{pmatrix}
        r_{0,0}^4 \xi_{t}''(\|v(0,0)\|^2) & r_{a,b}^2 r_{0,0}^2 \xi_t''\left( \langle v(a,b) , v(0,0) \rangle \right) \\
        r_{a,b}^2 r_{0,0}^2 \xi_t''\left( \langle v(a,b) , v(0,0) \rangle \right) & r_{a,b}^4 \xi_t''(\|v(a,b)\|^2)
    \end{pmatrix}
    \]
    be the covariance structure of the scaled GOE matrices $\grad^2 H_{0,0}(0)$ and $\grad^2 H_{a,b}(0)$. It is not difficult to see that $\alpha_2 = \alpha_1 + O\left( \frac{a^2+b^2}{N} \right)$, and $\rho$ is between $\alpha_1$ and $\alpha_2$. Also note that $\alpha_1 = \xi_{0,0}''(0)$ and $\alpha_2 = \xi_{a,b}''(0)$. Then, for some choice of GOE matrices $G$ and $\wt{G}$, we may write
    \begin{align*}
        \grad^2 H_{0,0}(0) &= \sqrt{\alpha_1} G \\
        \grad^2 H_{a,b}(0) &= \frac{\rho}{\sqrt{\alpha_1}} G + \sqrt{ \alpha_2 - \frac{\rho^2}{\alpha_1} } \wt{G}.
    \end{align*}
    We thus have
    \begin{align*}
        2 \left( \Error_{0,0}^{(2)} - \Error_{a,b}^{(2)} \right) &= \frac{N\alpha_1}{2} - \frac{N\alpha_2}{2} - \log \det \left( \underbrace{\left( 1 + \alpha_1 \right) \Id - \sqrt{\alpha_1} G}_{M_1} \right) \\
            &\qquad + \log \det \left( (1 + \alpha_2) \Id - \frac{\rho}{\sqrt{\alpha_1}} G - \sqrt{\alpha_2 - \frac{\rho^2}{\alpha_1}} \wt{G} \right).
    \end{align*}
    We may write the matrix inside the final $\log\det$ as
    \begin{align*}
        &(1 + \alpha_2) \Id - \frac{\rho}{\sqrt{\alpha_1}} G - \sqrt{\alpha_2 - \frac{\rho^2}{\alpha_1}} \wt{G} \\ &\qquad\qquad\qquad= \left( (1+\alpha_1)\Id - \sqrt{\alpha_1} G \right) + \underbrace{(\alpha_2 - \alpha_1) \Id - \left( \frac{\rho}{\sqrt{\alpha_1}} - \sqrt{\alpha_1} \right) G}_{M_2} - \underbrace{\sqrt{ \alpha_2 - \frac{\rho^2}{\alpha_1} } \wt{G}}_{M_3}.
    \end{align*}
    The difference of the two $\log \det$ terms is thus equal to
    \[ \log \det \left( \Id + \underbrace{M_1^{-1/2} M_2 M_1^{-1/2} + M_1^{-1/2} M_3 M_1^{-1/2}}_{M} \right). \]
    Observe that $M_3$ is a scaled GOE matrix independent of $M_1$ (and $M_2$). Taylor expanding the above, we shall control the trace and Frobenius norm of $M$. It may be verified that the higher-order terms, corresponding to higher Schatten norms, are $O(1)$. To control the trace and Frobenius norm, we shall essentially control their values in expectation. Standard concentration arguments for GOE matrices, along the lines of \Cref{lem:ZN2-laplace} using \cite[Lemma 1.2(b) and Corollary 1.6(b)]{GZ00}, allow us to assume (with probability $1-e^{-cN}$) that the eigenvalues of $G$ are distributed according to the semicircular distribution up to some small Wasserstein perturbation. That is, with probability $1-e^{-cN}$, denoting by $\lambda_i(G)$ the eigenvalues of $G$,
    \begin{align*}
        \Tr M_1^{-1/2} M_2 M_1^{-1/2} &= \sum_{1 \le i \le N} \frac{ (\alpha_2 - \alpha_1) - \left( \frac{\rho}{\sqrt{\alpha_1}} - \sqrt{\alpha_1} \right) \lambda_i(G) }{ (1+\alpha_1) - \sqrt{\alpha_1} \lambda_i(G) } \\
            &= N \cdot \int{\frac{ (\alpha_2 - \alpha_1) - \left( \frac{\rho}{\sqrt{\alpha_1}} - \sqrt{\alpha_1} \right) u }{ (1+\alpha_1) - \sqrt{\alpha_1} u } \dif \mu_{\semi}(u)} + O(1) \\
            &= N (\alpha_2 - \rho) + O(1),
    \end{align*}
    where the final equality follows from the standard semicircle integral $\int \frac{1}{x-u} \dif \mu_{\semi}(u) = \frac{1}{2} \left( x - \sqrt{x^2 - 4} \right)$. On the other hand, because $\wt{G}$ is independent of $G$,
    \[ \Tr M_1^{-1/2} M_3 M_1^{-1/2} = O(1) \]
    with probability $1-e^{-cN}$. Let us next control the Frobenius norms of these matrices. Again, because $\wt{G}$ is independent of $G$, with very high probability,
    \[ \norm*{ M }_F^2 = O(1) + \norm*{ M_1^{-1/2} M_2 M_1^{-1/2} }_F^2 + \norm*{M_1^{-1/2} M_3 M_1^{-1/2}}_F^2. \]
    The first squared Frobenius norm is equal to
    \[ \sum \left( \frac{ (\alpha_2 - \alpha_1) - \left( \frac{\rho}{\sqrt{\alpha_1}} - \sqrt{\alpha_1} \right) \lambda_i(G) }{ (1+\alpha_1) - \sqrt{\alpha_1} \lambda_i(G) } \right)^2. \]
    Let $\iota$ such that $(1+\alpha_1) - (2+\iota)\sqrt{\alpha_1} > \iota$ (this uses strict replica symmetry). Then, with probability $1-e^{-cN}$, $|\lambda_i(G)| \le 2+\iota$ for all $i$. Conditioned on this event happening, and recalling that $\alpha_2 - \alpha_1 = O\left( \frac{a^2 + b^2}{N} \right)$, the above is $\frac{O(a^2+b^2)^2}{N}$. This is $O(1)$ for all choices of $a,b \le \iota N^{1/4}$.

    We must next control the squared Frobenius norm of $M_1^{-1/2}M_3M_1^{-1/2}$. Let us condition on a typical realization of $M_1$: all its eigenvalues are smaller than $2+\iota$ in magnitude, and the empirical spectral distribution is Wasserstein-close to the semicircle law in the same sense as the previous section (where we controlled the trace), in that
    \[ \sum \frac{1}{(1+\alpha_1) - \sqrt{\alpha_1} \lambda_i(G)} = N \cdot \int \frac{1}{(1+\alpha_1) - \sqrt{\alpha_1} u} \dif \mu_{\semi}(u) + O(1) = N + O(1) \]
    Because $M_3$ is independent of $M_1$, it suffices to control the expected Frobenius norm of the matrix -- the true realization concentrates around its expectation to additive $O(1)$ factors. It is not difficult to see that this expectation is equal to
    \begin{align*}
        &\frac{1}{N} \cdot  \left( \alpha_2 - \frac{\rho^2}{\alpha_1} \right) \cdot \left( \sum \frac{ 1 }{ (1+\alpha_1) - \sqrt{\alpha_1} \lambda_i(G) } \right)^2 \\
        &\qquad\qquad= N \cdot \left( \alpha_2 - \frac{\rho^2}{\alpha_1} \right) \cdot \left( \int \frac{ 1 }{ (1+\alpha_1) - \sqrt{\alpha_1} u } \dif \mu_{\semi}(u) \right)^2 + O(1) \\
        &\qquad\qquad= N \cdot \left( \alpha_2 - \frac{\rho^2}{\alpha_1} \right) + O(1).
    \end{align*}
    Putting the pieces together and returning to the Taylor expansion, we get that with very high probability,
    \begin{align*}
        &2 \left( \Error_{0,0}^{(2)} - \Error_{a,b}^{(2)} \right) \\
        &\qquad= \frac{N\alpha_1}{2} - \frac{N\alpha_2}{2} + \log \det \left( \Id + M \right) \\
        &\qquad= O(1) + \frac{N\alpha_1}{2} - \frac{N\alpha_2}{2} + \Tr \left( M \right) - \frac{1}{2} \norm*{ M }_F^2 \\
        &\qquad= O(1) + N \cdot \left( \frac{\alpha_1}{2} - \frac{\alpha_2}{2} + (\alpha_2 - \rho) - \frac{1}{2} \left( \alpha_2 - \frac{\rho^2}{\alpha_1} \right) \right) \\
        &\qquad= O(1) + N \cdot \left( \frac{\alpha_1}{2} - \rho + \frac{\rho^2}{2\alpha_1} \right) \\
        &\qquad = O(1) + \frac{N}{2} \cdot \frac{(\alpha_1 - \rho)^2}{\alpha_1}.
    \end{align*}
    Because $\alpha_1 - \rho = O\left( \frac{a^2+b^2}{N} \right)$, this is $O(1)$ for $a,b \le \iota N^{1/4}$, completing the proof.
\end{proof}

\subsection{Moment calculations for covariance bounds}
\label{ss:cov-bd-deferred-proofs}
We first prove subgaussian concentration for the covariance of the degree-$2$ part.
\begin{proof}[Proof of \Cref{ppn:deg2-subgaussianity}]
    Part \eqref{it:deg2-subgaussianity-mmt} follows from part \eqref{it:deg2-subgaussianity-tail} by a standard tail integration argument.
    Indeed, the random variable $W$ is bounded by $N^{1/2}$, so the contribution to $\E[W^k]$ from the event $|W| \ge N^{1/5}$ is bounded by
    \[
        N^{k/2} \bbP(|W| \ge N^{1/5})
        \le N^{k/2} e^{-cN^{2/5}},
    \]
    which is vanishing for any constant $k$.
    So, we focus on proving part \eqref{it:deg2-subgaussianity-tail}.
    In the case $W = \la \sigma^1, \sigma^2 \ra / \sqrt{N}$, this is a special case of \cite[Lemma 7.5]{HMP24} (where we take $u = 0$).
    We consider the case $W = \la \sigma, v_i \ra$.
    Recall that
    \[
        H_{N,2}(\sigma) = \la A\sigma,\sigma \ra = \fr{\sqrt{\xi''(0)}}{2} \la M\sigma, \sigma \ra
    \]
    where $M \sim \GOE(N)$.
    For $0\le s\le N^{1/5} \log N$, we will evaluate
    \[
        \int e^{H_{N,2}(\sigma)} \dif\rho(\sigma)
        \qquad \text{and} \qquad
        \int e^{H_{N,2}(\sigma) + s\la v_i, \sigma \ra} \dif\rho(\sigma)
    \]
    using \cite[Lemma 7.3]{HMP24}.
    We recall the function $G : (\lambda_{\max}(A),+\infty)$ defined in \eqref{eq:def-G}, which we copy below for convenience, and define $\tG$ by
    \[
        G(\gamma) = \gamma - \fr{1}{2N} \log\det(\gamma I - A), \qquad
        \tG(\gamma) = G(\gamma) + \fr{s^2}{4N (\gamma - \lambda_i(A))}.
    \]
    Recall from below \eqref{eq:def-G} that $G'$ has a unique root $\gamma_*$ on $(\lambda_{\max}(A),+\infty)$.
    By the same argument, $\tG'$ has a unique root $\tgam_*$ on the same interval.
    As argued in the proof of \Cref{lem:ZN2-laplace}, the conditions of \cite[Lemma 7.3]{HMP24} apply.
    Applying this lemma with $u=0$ and $u = sv_i$, respectively, shows that with probability $1-e^{-cN}$,
    \begin{align}
        \int e^{H_{N,2}(\sigma)} \dif\rho(\sigma)
        &= (1 + O(N^{-c})) \sqrt{\fr{2}{G''(\gamma_*)}} (2e)^{-N/2} \exp(NG(\gamma_*)), \nonumber\\
        \int e^{H_{N,2}(\sigma) + s\la v_i, \sigma \ra} \dif\rho(\sigma)
        &= (1 + O(N^{-c})) \sqrt{\fr{2}{\tG''(\tgam_*)}} (2e)^{-N/2} \exp(N\tG(\tgam_*)).\label{eq:ext-field-laplace}
    \end{align}
    Suppose further the probability $1-e^{-cN}$ event in \Cref{lem:gz00-spam} holds.
    As argued in the proof of \Cref{ppn:deg2-partition-fn}, $|\gamma_* - \gamma_0| \le \fr{1}{C\sqrt{N}}$.
    Also, with probability $1-e^{-cN}$, $\lambda_{\max}(A) \le \fr{\sqrt{\xi''(0)}}{2} (2+\eps^2/8)$.

    We will show that on the intersection of these events $|\gamma_* - \tgam_*| = O(N^{-3/5} \log^2 N)$.
    First note that
    \[
        \gamma_* - \lambda_{\max}(A)
        \ge \gamma_0 - \lambda_{\max}(A) - \fr{1}{C\sqrt{N}}
        \ge \fr{1 + \xi''(0) - (2 + \eps^2/8) \sqrt{\xi''(0)}}{2} - \fr{1}{C\sqrt{N}}
        \ge \eps^2/32
    \]
    is bounded below by a constant, as in the proof of \Cref{fac:nabla2HN2-psd-domination}.
    Since $G'(\gamma_*) = 0$ and $\wt{G}'(\gamma) = G'(\gamma) - \frac{s^2}{4N(\gamma - \lambda_i(A))^2}$, we have $\wt{\gamma}_* \ge \gamma_*$ and so $\wt{\gamma}_* - \lambda_i(A)$ is also bounded below by a constant.
    Thus, as $s \le N^{1/5}\log N$, we have
    \[
        0 \ge \tG'(\gamma_*) = -\fr{s^2}{4N(\gamma - \lambda_i(A))^2} = O(N^{-3/5} \log^2 N).
    \]
    By direct computation,
    \[
        \wt{G}''(\gamma) = G''(\gamma) + \frac{s^2}{8N(\gamma - \lambda_i(A))^2}.
    \]
    Since $\wt{\gamma}_* \ge \gamma_*$,  we have $\gamma - \lambda_i(A)$ is bounded below by a constant for all $\gamma \in [\gamma_*, \wt{\gamma}_*]$. 
    We claim that $\wt{\gamma}_* \in [\gamma_0 - N^{-1/2}, \gamma_0 + N^{-1/2}]$.
    Combining both conclusions of \Cref{lem:gz00-spam}, we obtain that $G''(\gamma) \ge \Omega_{\eps}(1)$ for $\gamma \in [\gamma_0 - N^{-1/2}, \gamma_0 + N^{-1/2}]$, and consequently the desired claim on $\wt{\gamma}_*$ holds.

    We thus have $\wt{G}''(\gamma) = O_{\eps}(1)$ in the same interval, and consequently we can conclude the stronger statement $|\gamma_* - \tgam_*| = O(N^{-3/5} \log^2 N)$. 
    Furthermore, since $\wt{G}^{(3)}(\gamma) = O_{\eps}(1)$, the same logic allows us to also conclude
    \[
        \tG''(\tgam_*) / \tG''(\gamma_*) = 1 + O(N^{-3/5} \log^2 N).
    \]
    By Taylor expanding $\wt{G}$ around $\wt{\gamma}_*$, we see
    \[
        N|\tG(\gamma_*) - \tG(\tgam_*)|
        \le \fr{N}{2} |\gamma_* - \tgam_*|^2 \sup_{\gamma \in [\gamma_0 - N^{-1/2}, \gamma_0 + N^{-1/2}]} G''(\gamma)
        = O(N^{-1/5} \log^4 N).
    \]
    The above two displays allow us to replace instances of $\wt{\gamma}_*$ with $\gamma_*$ in \eqref{eq:ext-field-laplace}, yielding
    \[
        \int e^{H_{N,2}(\sigma) + s\la v_i, \sigma \ra} \dif\rho(\sigma)
        = (1 + O(N^{-c})) \sqrt{\fr{2}{\tG''(\gamma_*)}} (2e)^{-N/2} \exp(N\tG(\gamma_*)),
    \]
    and thus
    \begin{align*}
        \la e^{s\la v_i, \sigma \ra} \ra_2
        &= \fr{\int e^{H_{N,2}(\sigma) + s\la v_i, \sigma \ra} \dif\rho(\sigma)}{\int e^{H_{N,2}(\sigma)} \dif\rho(\sigma)} \\
        &= (1 + O(N^{-c})) \sqrt{\fr{G''(\gamma_*)}{\tG''(\gamma_*)}} \exp(N(\tG(\gamma_*) - G(\gamma_*))) \\
        &= (1 + O(N^{-c})) \exp(s^2/(4(\gamma_* - \lambda_i(A))))
        = (1 + O(N^{-c})) \exp(cs^2),
    \end{align*}
    where the last two steps again use that $\gamma_* - \lambda_i(A)$ is bounded away from $0$.
    The tail estimate on $W = \la v_i, \sigma \ra$ now follows from a standard Chernoff bound.
\end{proof}
We next turn to \Cref{lem:free-energy-typical-truncation}, proving each part in turn.
\begin{proof}[Proof of \Cref{lem:free-energy-typical-truncation}, \Cref{eq:free-energy-typical-1}]
    We reproduce \Cref{eq:free-energy-typical-1} below for convenience:
    \[
        \E \int_{S_N} \Ind[\sigma \not\in T(H_N)] e^{H_N(\sigma)} \dif\rho(\sigma)
        \le e^{N\xi(1)/2 - cN^{1/5}}.
    \]
    The proof follows \cite[Proposition 3.1]{HS23}, except with more precise control of overlaps between $N^{-2/5}$ and a small constant.
    By symmetry of the sphere, for any deterministic $\bx \in S_N$,
    \begin{equation}
        \label{eq:free-energy-typical-1-step1}
        \E \int_{S_N} \Ind[\sigma \not\in T(H_N)] e^{H_N(\sigma)} \dif\rho(\sigma)
        = \E \lt[
            \Ind[\bx \not\in T(H_N)]
            e^{H_N(\bx)}
        \rt].
    \end{equation}
    Let $\HamDist_{\pl}(\cdot | \bx)$ denote the planted model (\Cref{def:planted}) conditional on spike $\bx$.
    A Gaussian change of measure calculation implies that the right-hand side of \eqref{eq:free-energy-typical-1-step1} equals
    \[
        e^{N\xi(1)/2}
        \bbP_{H_N^{\bx,\RomI} \sim \HamDist_{\pl}(\cdot | \bx)} [\bx \not\in T(H_N^{\bx,\RomI})].
    \]
    Thus it suffices to show
    \[
        \bbP_{H_N^{\bx,\RomI} \sim \HamDist_{\pl}(\cdot | \bx)} [\bx \not\in T(H_N^{\bx,\RomI})]
        \le e^{-cN^{1/5}}.
    \]
    Recall (\Cref{rmk:planted-interpretation}) that a sample $H_N^{\bx,\RomI} \sim \HamDist_{\pl}(\cdot | \bx)$ can be generated by
    \begin{equation}
        \label{eq:planted-conditional-on-bx}
        H_N^{\bx,\RomI}(\sigma) = N\xi(R(\bx,\sigma)) + \wt{H}_N(\sigma),
    \end{equation}
    where $\wt{H}_N \sim \HamDist_{\nullmodel}$.
    Furthermore, from the definition, $\bx \in T(H_N^{\bx,\RomI})$ is equivalent to
    \begin{equation}
        \label{eq:bx-in-T-I}
        \int_{S_N}
        \Ind[|R(\bx,\tau)| \ge N^{-2/5}]
        e^{H_N^{\bx,\RomI}(\tau)}
        \dif\rho(\tau)
        \le e^{N\xi(1)/2 - cN^{1/5}}
    \end{equation}
    We will show this occurs with probability at least $1-e^{-cN^{1/5}}$.
    Let $\psi$ denote the probability density of $R(\bx,\tau) \in [-1,1]$, where $\tau$ is sampled from the Haar measure on $S_N$.
    Then it is known that
    \[
        \psi(q) = \fr{1}{Z_\psi} (1-q^2)^{(N-3)/2}
    \]
    where $Z_\psi = \Theta(N^{-1/2})$.
    Define the codimension-$1$ band
    \[
        \Band(q) = \Band(q;\bx) \coloneqq \{\tau \in S_N : R(\bx,\tau) = q\}
    \]
    and let
    \[
        Z^{\bx,\RomI}(q) = \int_{\Band(q)}
        e^{H_N^{\bx,\RomI}(\tau)}
        \dif\rho_q(\tau),
    \]
    where $\rho_q$ is the Haar measure on $\Band(q)$, normalized so that $\rho_q(\Band(q)) = \psi(q)$.
    Then the left-hand side of \eqref{eq:bx-in-T-I} is equal to
    \[
        \int_{N^{-2/5} \le |q| \le 1}
        Z^{\bx,\RomI}(q) \dif{q}.
    \]
    An application of Guerra's interpolation as in \cite[Lemma 3.3]{HS23} shows that for any $q\in [-1,1]$ and constant $\eta > 0$, with probability $1-e^{-cN}$,
    \[
        \fr1N \log Z^{\bx,\RomI}(q) \le \fr12 \lt(\xi(1) + \xi(|q|) + |q| + \log(1-|q|) \rt) + \eta.
    \]
    Since $\xi_{\sim1}$ is $\eps$-strictly replica symmetric and $\gamma_1^2 \le N^{-4/5}$, this implies
    \[
        \fr1N \log Z^{\bx,\RomI}(q) \le \fr{\xi(1)}{2} - \fr{\eps q^2}{4} + 2\eta.
    \]
    Let $\delta > 0$ be small depending on $\eps$, and $\eta$ small depending on $\delta$.
    This implies that for any $|q| \ge \delta$, with probability $1-e^{-cN}$,
    \[
        \fr1N \log Z^{\bx,\RomI}(q) \le \fr{\xi(1)}{2} - \fr{\eps \delta^2}{8}.
    \]
    Taking a union bound over a $N^{-1}$-net of $|q|\ge \delta$ as in \cite[Lemma 3.4]{HS23} implies that with probability $1-e^{-cN}$,
    \begin{equation}
        \label{eq:typicality-far-from-0}
        \int_{\delta \le |q| \le 1}
        Z^{\bx,\RomI}(q) \dif{q}
        \le e^{N\xi(1)/2 - cN}.
    \end{equation}
    We address the remaining range of $q$ by a first moment bound.
    Note that
    \begin{equation}
        \label{eq:typicality-close-to-0}
        \E
        \int_{N^{-2/5} \le |q| \le \delta}
        Z^{\bx,\RomI}(q) \dif{q}
        =
        e^{N\xi(1)/2}
        \int_{N^{-2/5} \le |q| \le \delta}
        e^{N\xi(q)} \psi(q)
        \dif{q}.
    \end{equation}
    Recall from \Cref{fac:xi2nd-at-0} that $\xi''(0) \le 1-\eps$.
    Thus, for sufficiently small $\delta$, for all $|q| \le \delta$,
    \[
        \xi_{\sim1}(q) + \fr12 \log(1-q^2)
        \le -\eps q^2 / 4.
    \]
    Thus, for all $N^{-2/5} \le |q| \le \delta$,
    \begin{align*}
        \fr1N \log \lt(e^{N\xi(q)} \psi(q)\rt)
        = \xi(q) + \fr1N \log \psi(q)
        &= \gamma_1^2 q + \xi_{\sim1}(q) + \fr12 \log(1-q^2) + O(N^{-1} \log N) \\
        &\le \gamma_1^2 q - \eps q^2 / 4 + O(N^{-1} \log N)
        \le -\eps N^{-4/5} / 8.
    \end{align*}
    Combining with \eqref{eq:typicality-close-to-0} shows
    \[
        \E
        \int_{N^{-2/5} \le |q| \le \delta}
        Z^{\bx,\RomI}(q) \dif{q}
        \le e^{N\xi(1)/2 - cN^{1/5}},
    \]
    so by Markov's inequality, with probability $1-e^{-cN^{1/5}/2}$,
    \[
        \int_{N^{-2/5} \le |q| \le \delta}
        Z^{\bx,\RomI}(q) \dif{q}
        \le e^{N\xi(1)/2 - cN^{1/5}/2}.
    \]
    Combining with \eqref{eq:typicality-far-from-0} proves \eqref{eq:bx-in-T-I} after adjusting $c$.
\end{proof}
\begin{proof}[Proof of \Cref{lem:free-energy-typical-truncation}, \Cref{eq:free-energy-typical-2}]
    By the same argument leading to \eqref{eq:bx-in-T-I}, it suffices to prove
    \begin{equation}
        \label{eq:bx-in-T-II}
        \int_{S_N}
        \Ind[|R(\bx,\tau)| \ge N^{-2/5}]
        e^{H_N^{\bx,\RomII}(\tau)}
        \dif\rho(\tau)
        \le e^{N\xi(1)/2 - cN^{1/5}}
    \end{equation}
    holds with probability at least $1-e^{-cN^{1/5}}$, where now
    \[
        H_N^{\bx,\RomII}(\sigma) = N\gamma_2^2 R(\bx,\sigma)^2 + \wt{H}_N(\sigma),
    \]
    i.e. we have replaced the spike in $H_N^{\bx,\RomI}$ with only its degree-$2$ part.
    Then $H_N^{\bx,\RomII}(\sigma) \le H_N^{\bx,\RomI}(\sigma)$ almost surely for all $\sigma$ such that $R(\bx,\sigma) \ge 0$, so \eqref{eq:bx-in-T-I} implies
    \[
        \int_{S_N}
        \Ind[R(\bx,\tau) \ge N^{-2/5}]
        e^{H_N^{\bx,\RomII}(\tau)}
        \dif\rho(\tau)
        \le e^{N\xi(1)/2 - cN^{1/5}}
    \]
    with probability $1-e^{-cN}$.
    Moreover, by symmetry of the degree-$2$ spike,
    \[
        \int_{S_N}
        \Ind[R(\bx,\tau) \ge N^{-2/5}]
        e^{H_N^{\bx,\RomII}(\tau)}
        \dif\rho(\tau)
        \stackrel{d}{=}
        \int_{S_N}
        \Ind[R(\bx,\tau) \le -N^{-2/5}]
        e^{H_N^{\bx,\RomII}(\tau)}
        \dif\rho(\tau).
    \]
    This implies \eqref{eq:bx-in-T-II} and thus \eqref{eq:free-energy-typical-2}.
\end{proof}
\begin{proof}[Proof of \Cref{lem:free-energy-typical-truncation}, \Cref{eq:free-energy-typical-3}]
    This follows from a slightly more complex form of the same strategy, where the planted Hamiltonian now has two spikes.
    For $\bx^1,\bx^2 \in S_N$, let
    \[
        H_N^{\bx^1,\bx^2,\RomIII}(\sigma) = N\xi(R(\bx^1,\sigma)) + N\xi(R(\bx^2,\sigma)) + \wt{H}_N(\sigma).
    \]
    By the same gaussian change of measure argument as above, the expectation in the left-hand side of \eqref{eq:free-energy-typical-3} equals
    \[
        e^{N\xi(1)} \int
        \Ind[|R(\sigma^1,\sigma^2)| \le 3N^{-2/5}]
        \bbP(\sigma^1 \not\in T(H_N^{\sigma^1,\sigma^2,\RomIII}))
        \dif\rho^{\otimes 2}(\sigma^1,\sigma^2).
    \]
    Thus it suffices to show that for all $\bx^1,\bx^2 \in S_N$ with $|R(\bx^1,\bx^2)| \le 3N^{-2/5}$,
    \begin{equation}
        \label{eq:bx-in-T-III}
        \int_{S_N}
        \Ind[|R(\bx^1,\tau)| \ge N^{-2/5}]
        e^{H_N^{\bx^1,\bx^2,\RomIII}(\tau)}
        \dif\rho(\tau)
        \le e^{N\xi(1)/2 - cN^{1/5}}
    \end{equation}
    with probability at least $1-e^{-cN^{1/5}}$.
    Let $\lambda = R(\bx^1,\bx^2) \in [-3N^{-2/5},3N^{2/5}]$ and
    \[
        \bx^2 = \lambda \bx^1 + \sqrt{1-\lambda^2} \bx^2_\perp,
    \]
    where $\bx^2_\perp \in S_N$ and $R(\bx^1,\bx^2_\perp) = 0$.
    Let $\psi_2$ denote the probability density of $(R(\bx^1,\tau),R(\bx^2_\perp,\tau)) \in [1,1]^2$, where $\tau$ is sampled from the Haar measure on $S_N$.
    It is known that
    \[
        \psi_2(q) = \fr{\Ind[q_1^2+q_2^2 \le 1]}{Z_{\psi_2}} (1-q_1^2-q_2^2)^{(N-4)/2}
    \]
    where $Z_\psi = \Theta(N^{-1})$.
    Define the codimension-$2$ band
    \[
        \Band(q_1,q_2) = \Band(q_1,q_2;\bx^1,\bx^2)
        \coloneqq \{\tau \in S_N : R(\bx^1,\tau) = q_1, R(\bx^2_\perp,\tau) = q_2\}.
    \]
    and let
    \[
        Z^{\bx^1,\bx^2,\RomIII}(q_1,q_2)
        = \int_{\Band(q_1,q_2)}
        e^{H_N^{\bx^1,\bx^2,\RomIII}(\tau)}
        \dif\rho_{q_1,q_2}(\tau),
    \]
    where $\rho_{q_1,q_2}$ is the Haar measure on $\Band(q_1,q_2)$, normalized so that $\rho_{q_1,q_2}(\Band(q_1,q_2)) = \psi_2(q_1,q_2)$.
    Then the left-hand side of \eqref{eq:bx-in-T-III} is equal to
    \begin{equation}
        \label{eq:part-III-integral}
        \int
        \Ind[|q_1| \ge N^{-2/5}]
        Z^{\bx^1,\bx^2,\RomIII}(q_1,q_2)
        \dif{(q_1,q_2)}.
    \end{equation}
    Note that
    \[
        \fr1N \log Z^{\bx^1,\bx^2,\RomIII}(q_1,q_2)
        = \xi(q_1) + \xi\lt(\lambda q_1 + \sqrt{1-\lambda^2} q_2\rt)
        + \fr1N \log \int_{\Band(q_1,q_2)}
        e^{\wt{H}_N(\tau)}
        \dif\rho_{q_1,q_2}(\tau).
    \]
    Let $\wt{q} = \wt{q}(q_1,q_2) \coloneqq \sqrt{q_1^2 + q_2^2}$.
    Applying Guerra's interpolation as in \cite[Lemma 3.3]{HS23} shows that for any $\eta > 0$, with probability $1-e^{-cN}$
    \[
        \fr1N \log \int_{\Band(q_1,q_2)}
        e^{\wt{H}_N(\tau)}
        \dif\rho_{q_1,q_2}(\tau)
        \le \fr12 \lt(
            \xi(1) - \xi(\wt{q}) + \wt{q} + \log(1-\wt{q})
        \rt) + \eta.
    \]
    and thus
    \begin{align*}
        \fr1N \log Z^{\bx^1,\bx^2,\RomIII}(q_1,q_2)
        &\le \xi(q_1) + \xi\lt(\lambda q_1 + \sqrt{1-\lambda^2} q_2\rt)
        + \fr12 \lt(
            \xi(1) - \xi(\wt{q}) + \wt{q} + \log(1-\wt{q})
        \rt) + \eta \\
        &\le \xi_{\sim1}(|q_1|) + \xi_{\sim1}(|q_2|) + \fr12 \lt(
            \xi(1) - \xi_{\sim1}(\wt{q}) + \wt{q} + \log(1-\wt{q})
        \rt) + 2\eta.
    \end{align*}
    Since $\xi_{\sim1}$ only includes terms that are degree $2$ or larger, $\xi_{\sim1}(|q_1|) + \xi_{\sim1}(|q_2|) \le \xi_{\sim1}(\wt{q})$.
    Thus the last display is bounded by
    \[
        \fr12 \lt(
            \xi(1) + \xi_{\sim1}(\wt{q}) + \wt{q} + \log(1-\wt{q})
        \rt) + 2\eta
        \le \fr{\xi(1)}{2} - \fr{\eps \wt{q}^2}{4} + 2\eta.
    \]
    Arguing as in the proof of equation \eqref{eq:free-energy-typical-1} then shows that for any $\delta > 0$ depending only on $\eps$,
    \begin{equation}
        \label{eq:typicality-far-from-0-III}
        \int
        \Ind[\wt{q}(q_1,q_2) > \delta]
        Z^{\bx^1,\bx^2,\RomIII}(q_1,q_2)
        \dif{(q_1,q_2)}
        \le e^{N\xi(1)/2 - cN}
    \end{equation}
    with probability $1-e^{-cN}$.
    The remaining part of the integral \eqref{eq:part-III-integral} has expectation
    \begin{align}
        \notag
        &\E
        \int
        \Ind[|q_1| \ge N^{-2/5}, \wt{q}(q_1,q_2) \le \delta]
        Z^{\bx^1,\bx^2,\RomIII}(q_1,q_2)
        \dif{(q_1,q_2)} \\
        \label{eq:typicality-close-to-0-III-expectation}
        &= e^{N\xi(1)/2}
        \int
        \Ind[|q_1| \ge N^{-2/5}, \wt{q}(q_1,q_2) \le \delta]
        e^{N\xi(q_1) + N\xi(\lambda q_1 + \sqrt{1-\lambda^2}q_2)}
        \psi_2(q_1,q_2)
        \dif{(q_1,q_2)}.
    \end{align}
    Recall $\gamma_1^2 \le N^{-4/5}$.
    For all $(q_1,q_2)$ in this indicator,
    \begin{align*}
        &\fr1N \log \lt(e^{N\xi(q_1) + N\xi(\lambda q_1 + \sqrt{1-\lambda^2}q_2)}
        \psi_2(q_1,q_2)\rt) \\
        &= \xi(q_1) + \xi\lt(\lambda q_1 + \sqrt{1-\lambda^2}q_2\rt)
        + \fr12 \log (1-\wt{q}^2)
        + O(N^{-1} \log N) \\
        &\le 2N^{-4/5} \wt{q} - \fr12 (1 - \xi''(0)) \wt{q}^2 + O(\wt{q}^3 + N^{-1} \log N) \\
        &\le -\fr{\eps \wt{q}^2}{2} + 2N^{-4/5} \wt{q} + O(\wt{q}^3 + N^{-1} \log N).
    \end{align*}
    Since $N^{-2/5} \le \wt{q} \le \delta$, for $\delta$ sufficiently small depending on $\eps$ this is bounded by $-cN^{-4/5}$.
    Combining with \eqref{eq:typicality-close-to-0-III-expectation} shows that with probability $1-e^{-cN^{1/5}}$,
    \[
        \int
        \Ind[|q_1| \ge N^{-2/5}, \wt{q}(q_1,q_2) \le \delta]
        Z^{\bx^1,\bx^2,\RomIII}(q_1,q_2)
        \dif{(q_1,q_2)}
        \le e^{N\xi(1)/2 - cN^{1/5}}.
    \]
    Further combining with \eqref{eq:typicality-far-from-0-III} completes the proof.
\end{proof}
\begin{proof}[Proof of \Cref{lem:free-energy-typical-truncation}, \Cref{eq:free-energy-typical-4}]
    This is proved identically to equation \eqref{eq:free-energy-typical-3}, except with spiked Hamiltonian
    \[
        H_N^{\bx^1,\bx^2,\RomIV}(\sigma) = N\gamma_2^2 R(\bx^1,\sigma)^2 + N\xi(R(\bx^2,\sigma)) + \wt{H}_N(\sigma),
    \]
    i.e. the spike involving $\bx^1$ is replaced with just its degree-$2$ component.
    The same argument applies and we omit details.
\end{proof}
We turn to the proofs of \Cref{lem:nondeg2-covariance-typical,lem:nondeg2-covariance-2mt}.
The following fact will be useful in the proofs of both lemmas.
\begin{fact}
    \label{fac:tT-ij-to-tT}
    If $\sigma \in S_N$ satisfies $|\sigma_i|,|\sigma_j| \le \log N$ and $\sigma \not\in \tT_{i,j}$, then $\sigma \not\in \tT$.
\end{fact}
\begin{proof}
    Consider any $\tau \in S_N$ satisfying $|R(\sigma_{\sim i,j},\tau_{\sim i,j})| \ge 2N^{-2/5}$.
    Since $|\tau_i|,|\tau_j| \le N^{1/2}$,
    \[
        |R(\sigma,\tau)|
        \ge |R(\sigma_{\sim i,j},\tau_{\sim i,j})| - \fr{|\sigma_i||\tau_i| + |\sigma_j||\tau_j|}{N}
        \ge N^{-2/5}.
    \]
    Thus the expectation over $\tau$ in \eqref{eq:tT} is larger than the expectation over $\tau$ in \eqref{eq:tT-ij}, and so $\sigma \not\in \tT$.
\end{proof}
\begin{proof}[Proof of \Cref{lem:nondeg2-covariance-typical}]
    We can write $X_{i,j} = \tX_{i,j} + \tX^{(1)}_{i,j} + \tX^{(2)}_{i,j}$, where
    \begin{align*}
        \tX^{(1)}_{i,j} &= \lt\la 
            \Ind[|\sigma_i|,|\sigma_j| \le \log N, \sigma \not\in \tT_{i,j}]
            \sigma_i \sigma_j \lt( e^{H_{N,\sim2}(\sigma) - N\xi_{\sim2}(1)/2} - \Ind[i=j] \rt)
        \rt\ra_2, \\
        \tX^{(2)}_{i,j} &= \lt\la 
            \Ind[|\sigma_i| \vee |\sigma_j| > \log N]
            \sigma_i \sigma_j \lt( e^{H_{N,\sim2}(\sigma) - N\xi_{\sim2}(1)/2} - \Ind[i=j] \rt)
        \rt\ra_2.
    \end{align*}
    By using $(a+b)^2 \le 2a^2+2b^2$, we deduce that 
    \[
        X_{i,j}^2 \le 2\wt{X}_{i,j}^2 + 2(\wt{X}_{i,j}^{(1)} + \wt{X}_{i,j}^{(2)})^2,
    \]
    so the rest of the proof is dedicated to showing that $|\wt{X}_{i,j}^{(1)} + \wt{X}_{i,j}^{(2)}| \le e^{-c\log^2 N}$ with probability $1-e^{-c\log^2 N}$. 
    To do so, we will simply apply Markov to control 
    $|\wt{X}_{i,j}^{(1)}|$ and $|\wt{X}_{i,j}^{(2)}|$.
    
    By \Cref{fac:tT-ij-to-tT} and using $\abs{\sigma_i}, \abs{\sigma_j} \le \sqrt{N}$,
    \begin{align*}
        |\tX^{(1)}_{i,j}|
        &\le N \lt\la
            \Ind[|\sigma_i|,|\sigma_j| \le \log N, \sigma \not\in \tT_{i,j}]
            \lt( e^{H_{N,\sim2}(\sigma) - N\xi_{\sim2}(1)/2} + \Ind[i=j] \rt)
        \rt\ra_2 \\
        &\le N \lt\la
            \Ind[\sigma \not\in \tT]
            \lt( e^{H_{N,\sim2}(\sigma) - N\xi_{\sim2}(1)/2} + \Ind[i=j] \rt)
        \rt\ra_2.
    \end{align*}
    By the first two equations from \Cref{cor:free-energy-typical-truncation}, $\E_{\sim 2} |\tX^{(1)}_{i,j}| \le e^{-cN^{1/5}}$.
    Furthermore,
    \begin{align*}
        \E_{\sim2} |\tX^{(2)}_{i,j}| &\le N\E_{\sim2} \lt\la
            \lt(\Ind[|\sigma_i| > \log N] + \Ind[|\sigma_j| > \log N]\rt)
            \lt( e^{H_{N,\sim2}(\sigma) - N\xi_{\sim2}(1)/2} + \Ind[i=j] \rt)
        \rt\ra_2 \\
        &\le 2N \lt\la
            \Ind[|\sigma_i| > \log N] + \Ind[|\sigma_j| > \log N]
        \rt\ra_2
        \le e^{-c\log^2 N}
    \end{align*}
    by \Cref{ppn:deg2-subgaussianity}\eqref{it:deg2-subgaussianity-tail}.
    By Markov's inequality, with probability $1-e^{-c\log^2 N}$, $|\tX^{(1)}_{i,j}| + |\tX^{(2)}_{i,j}| \le e^{-c\log^2 N}$, after adjusting $c$ as necessary.
    This concludes the proof.
\end{proof}
\begin{proof}[Proof of \Cref{lem:nondeg2-covariance-2mt}]
    We can write $\tX_{i,j}^2 = \hX_{i,j} + \hX_{i,j}^{(1)} - \hX_{i,j}^{(2)}$, where
    \begin{align*}
        \hX_{i,j}^{(1)}
        &= \bigg\la
            \Ind[|\sigma^1_i|,|\sigma^1_j|,|\sigma^2_i|,|\sigma^2_j| \le \log N,
            \sigma^1,\sigma^2 \in \tT_{i,j},
            |R(\sigma^1_{\sim i,j},\sigma^2_{\sim i,j})| > 2N^{-2/5}] \\
            &\qquad
            \sigma^1_i \sigma^1_j \sigma^2_i \sigma^2_j
            \lt( e^{H_{N,\sim2}(\sigma^1) - N\xi_{\sim2}(1)/2} - \Ind[i=j] \rt)
            \lt( e^{H_{N,\sim2}(\sigma^2) - N\xi_{\sim2}(1)/2} - \Ind[i=j] \rt)
        \bigg\ra_2, \\
        \hX_{i,j}^{(2)}
        &= \bigg\la
            \Ind[|\sigma^1_i|,|\sigma^1_j|,|\sigma^2_i|,|\sigma^2_j| \le \log N,
            (\sigma^1\not\in \tT_{i,j}\,\lor\,\sigma^2 \not\in \tT_{i,j}),
            |R(\sigma^1_{\sim i,j},\sigma^2_{\sim i,j})| \le 2N^{-2/5}] \\
            &\qquad
            \sigma^1_i \sigma^1_j \sigma^2_i \sigma^2_j
            \lt( e^{H_{N,\sim2}(\sigma^1) - N\xi_{\sim2}(1)/2} - \Ind[i=j] \rt)
            \lt( e^{H_{N,\sim2}(\sigma^2) - N\xi_{\sim2}(1)/2} - \Ind[i=j] \rt)
        \bigg\ra_2.
    \end{align*}
    By the same argument as before, it suffices to control $\E_{\sim2} |\hX_{i,j}^{(1)}|$ and $\E_{\sim2} |\hX_{i,j}^{(2)}|$.
    Note that almost surely,
    \begin{align*}
        |\hX_{i,j}^{(1)}|
        &\le N^2 \bigg\la
            \Ind[|\sigma^1_i|,|\sigma^1_j|,|\sigma^2_i|,|\sigma^2_j| \le \log N,
            \sigma^1,\sigma^2 \in \tT_{i,j},
            |R(\sigma^1_{\sim i,j},\sigma^2_{\sim i,j})| > 2N^{-2/5}] \\
            &\qquad
            \lt( e^{H_{N,\sim2}(\sigma^1) - N\xi_{\sim2}(1)/2} + 1 \rt)
            \lt( e^{H_{N,\sim2}(\sigma^2) - N\xi_{\sim2}(1)/2} + 1 \rt)
        \bigg\ra_2
        \le N^2 (\hX_{i,j}^{(3)} + \hX_{i,j}^{(4)})
    \end{align*}
    where
    \begin{align*}
        \hX_{i,j}^{(3)} &= \lt\la
            \Ind[ \sigma^1 \in \tT_{i,j}, |R(\sigma^1_{\sim i,j},\sigma^2_{\sim i,j})| > 2N^{-2/5}]
            e^{H_{N,\sim2}(\sigma^1) + H_{N,\sim2}(\sigma^2) - N\xi_{\sim2}(1)}
        \rt\ra_2, \\
        \hX_{i,j}^{(4)} &= \bigg\la
            \Ind[|\sigma^1_i|,|\sigma^1_j|,|\sigma^2_i|,|\sigma^2_j| \le \log N,
            |R(\sigma^1_{\sim i,j},\sigma^2_{\sim i,j})| > 2N^{-2/5}] \\
            &\qquad (e^{H_{N,\sim2}(\sigma^1) - N\xi_{\sim2}(1)/2} + e^{H_{N,\sim2}(\sigma^2) - N\xi_{\sim2}(1)/2} + 1)
        \bigg\ra_2.
    \end{align*}
    By definition of $\tT_{i,j}$,
    \[
        \hX_{i,j}^{(3)} \le \lt\la
            \Ind[ \sigma^1 \in \tT_{i,j}]
            e^{H_{N,\sim2}(\sigma^1) - N\xi_{\sim2}(1) / 2 - cN^{1/5}}
        \rt\ra_2,
    \]
    and thus $\E_{\sim2} \hX_{i,j}^{(3)} \le e^{-cN^{1/5}}$.
    Furthermore,
    \[
        \hX_{i,j}^{(4)} \le \lt\la
            \Ind[ |R(\sigma^1,\sigma^2)| > N^{-2/5}]
            (e^{H_{N,\sim2}(\sigma^1) - N\xi_{\sim2}(1)/2} + e^{H_{N,\sim2}(\sigma^2) - N\xi_{\sim2}(1)/2} + 1)
        \rt\ra_2,
    \]
    and thus
    \[
        \E_{\sim2} \hX_{i,j}^{(4)}
        \le 3 \lt\la \Ind[|R(\sigma^1,\sigma^2)| > N^{-2/5}] \rt\ra_2
        \le e^{-cN^{1/5}}
    \]
    by \Cref{ppn:deg2-subgaussianity}\eqref{it:deg2-subgaussianity-tail}.
    Combining shows $\E_{\sim2} |\hX_{i,j}^{(1)}| \le e^{-cN^{1/5}}$.
    Similarly
    \begin{align*}
        |\hX_{i,j}^{(2)}|
        &\le N^2 \bigg\la
            \Ind[|\sigma^1_i|,|\sigma^1_j|,|\sigma^2_i|,|\sigma^2_j| \le \log N,
            (\sigma^1\not\in \tT_{i,j}\,\lor \,\sigma^2 \not\in \tT_{i,j}), \\
            &\qquad |R(\sigma^1_{\sim i,j},\sigma^2_{\sim i,j})| \le 2N^{-2/5}]
            \lt( e^{H_{N,\sim2}(\sigma^1) - N\xi_{\sim2}(1)/2} + 1 \rt)
            \lt( e^{H_{N,\sim2}(\sigma^2) - N\xi_{\sim2}(1)/2} + 1 \rt)
        \bigg\ra_2
    \end{align*}
    By \Cref{fac:tT-ij-to-tT}, on the indicator in this expectation, $\sigma^1,\sigma^2 \in \tT$.
    Moreover
    \[
        |R(\sigma^1,\sigma^2)|
        \le |R(\sigma^1_{\sim i,j},\sigma^2_{\sim i,j})| + \fr{|\sigma^1_i||\sigma^2_i| + |\sigma^1_j||\sigma^1_j|}{N}
        \le 3N^{-2/5}.
    \]
    Thus
    \begin{align*}
        |\hX_{i,j}^{(2)}|
        &\le 2N^2 \bigg\la
            \Ind[\sigma^1\not\in \tT, |R(\sigma^1,\sigma^2)| \le 3N^{-2/5}] \\
            &\qquad \lt( e^{H_{N,\sim2}(\sigma^1) - N\xi_{\sim2}(1)/2} + 1 \rt)
            \lt( e^{H_{N,\sim2}(\sigma^2) - N\xi_{\sim2}(1)/2} + 1 \rt)
        \bigg\ra_2
    \end{align*}
    and \Cref{cor:free-energy-typical-truncation} implies $\E_{\sim2} |\hX_{i,j}^{(2)}| \le e^{-cN^{1/5}}$.
\end{proof}

\end{document}